\documentclass[11pt,dvipsnames,reqno,hypertexnames=false]{amsart}
\usepackage[T1]{fontenc}
\usepackage{a4wide}
\usepackage{yhmath,mathrsfs,amsthm,amsmath,amssymb,amsfonts,enumerate,lipsum, appendix,mathtools, bm}
\usepackage{bbold}
\usepackage{pdfpages}
\usepackage[mathscr]{euscript}
\usepackage{thmtools}
\usepackage{graphicx}
\usepackage{pgf,tikz}
\usepackage{tkz-euclide}
\usepackage{caption}
\usepackage{subcaption}
\usetikzlibrary{shapes.geometric}
\usetikzlibrary{arrows,hobby}
\usepackage[shortlabels]{enumitem}
\usepackage[english]{babel}
\allowdisplaybreaks

\usepackage[sort,numbers]{natbib}
\usepackage{thmtools}

\usepackage{hyperref}
\hypersetup{
    colorlinks=true,
    linkcolor=black,
    filecolor=black,
    urlcolor=black,
    citecolor=blue,
}
\usepackage{orcidlink} 
\usepackage{cleveref}  
\usepackage[margin=0.90in, heightrounded]{geometry}
\usepackage{bigints}
\usepackage{url}
\usepackage{esint}
\newtheorem{theorem}{Theorem}[section]
\newtheorem{lemma}[theorem]{Lemma}
\newtheorem{proposition}[theorem]{Proposition}

\newtheorem{definition}[theorem]{Definition}
\theoremstyle{definition}

\newtheorem{remark}[theorem]{Remark}
\numberwithin{equation}{section}
\usepackage[hyperpageref]{backref}
\DeclarePairedDelimiter\abs{\lvert}{\rvert}
\DeclarePairedDelimiter\norm{\lVert}{\rVert}
\makeatletter
\let\oldnorm\norm
\def\norm{\@ifstar{\oldnorm}{\oldnorm*}}

\DeclareMathOperator*{\esssup}{ess\,sup}
\DeclareMathOperator*{\essinf}{ess\,inf}

\makeatletter
\newcommand{\al} {\alpha}

\newcommand{\be} {\beta}

\newcommand{\om} {\omega}
\newcommand{\Om} {\Omega}
\newcommand{\la} {\lambda}

\newcommand{\sig} {\sigma}

\newcommand{\no} {\nonumber}
\newcommand{\noi} {\noindent}
\newcommand{\var} {\varepsilon}
\newcommand{\ra} {\rightarrow}

\DeclareMathAlphabet{\mathpzc}{T1}{pzc}{m}{it}

\def\K{{\mathcal{K}^{s+\log}}}
\def\Kpm{{\mathcal{K}^{s+\log}_{\pm}}}
\def\Kp{{\mathcal{K}^{s+\log}_{+}}}
\def\Km{{\mathcal{K}^{s+\log}_{-}}}

\def\Tail{\text{\rm Tail}}
\def\ps{p_{s}^{*}}
\def\dx{{\,\rm d}x}
\def\dy{{\,\rm d}y}
\def\dz{{\,\rm d}z}
\def\d1{{\,\rm d}z_1}
\def\d2{{\,\rm d}z_2}
\def\dxy{{\,\rm d}x{\,\rm d}y}
\def\dt{\,{\rm d}t}
\def\dtau{\,{\rm d}\tau}

\def\dr{\,{\rm d}r}

\def\ds{\,{\rm d}s}

\def\sb2{{{\mathcal D}^{1,2}_0(B_1^c)}}

\newcommand*\rd{\mathbb{R}^d}

\def\C{{\mathcal C}}

\def\R{{\mathbb R}}
\def\N{{\mathbb N}}

\def\({{\Big(}}
\def\){{\Big)}}
\def\cc{{\mathcal{C}_c^\infty}}

\def\logp{{(- \Delta_p)^{s+\log}}}
\def\l2{{(- \Delta)^{s+\log}}}

\def\c1Loc{{\C_{loc}^1}}

\DeclareMathOperator\supp{supp}

\usepackage{xpatch}
\makeatletter   
\xpatchcmd{\@tocline}
{\hfil\hbox to\@pnumwidth{\@tocpagenum{#7}}\par}
{\ifnum#1<0\hfill\else\dotfill\fi\hbox to\@pnumwidth{\@tocpagenum{#7}}\par}
{}{}

\def\l@subsection{\@tocline{2}{0pt}{2pc}{6pc}{}}
\def\l@subsubsection{\@tocline{3}{0pt}{8pc}{8pc}{}}
\makeatother
\usepackage{fancyhdr}
\title[Fractional logarithmic $p$-Laplacian: Harnack inequality and H\"older regularity]{Regularity for the fractional logarithmic $p$-Laplacian}
\author[]{Nirjan Biswas$^1$, Stuti Das$^2$, and Abhrojyoti Sen$^3$
}
\address{\rm$^1$ Department of Mathematics, Indian Institute of Science Education and Research Pune, Dr. Homi Bhabha Road,
Pune 411008, India\,}
\address{\rm$^2$Department of Mathematics, Indian Institute of Technology Kanpur, Kanpur-208016, India\,}
\address{\rm$^3$Goethe-Universit\"{a}t Frankfurt, Institut f\"{u}r Mathematik, Robert-Mayer-Str. 10, D-60629 Frankfurt, Germany\,}
\email[N. Biswas]{nirjan.biswas@acads.iiserpune.ac.in, nirjaniitm@gmail.com} 
\email[S. Das]{stutid21@iitk.ac.in}
\email[A. Sen]{sen@math.uni-frankfurt.de}
\thanks{$^3$Corresponding author}	
\begin{document}
\begin{abstract}
We prove the Harnack inequality (with tails) and local H\"older regularity for the fractional logarithmic $p$-Laplace operator, which is derived by differentiating the fractional $p$-Laplace operator with respect to its order. To be more precise, for a suitable function $u,$ the operator reads as the first order derivative
\begin{align*}
    (-\Delta_p)^{s+\log} u:= \frac{{\rm d}}{{\rm d}t}(-\Delta_p)^t u \Big|_{t=s}
\end{align*}
at any arbitrary order $s\in (0, 1).$ The kernel of this operator involves a
logarithmic factor. As a consequence, it changes sign at large scales and, near the diagonal, is more singular than the kernel of the fractional $p$-Laplacian. To achieve our regularity estimates, we adopt the classical De Giorgi-Nash-Moser techniques in this setting. We also construct an example showing that the Harnack inequality fails without tail terms. Our results are new even in the linear setup $p=2$. 
\end{abstract}
\maketitle
\tableofcontents

\section{Introduction and main results}
\subsection{Problem setup and motivations}
Let $\Omega$ be a bounded open set in $\rd$ with $d \ge 1$. The main goal of this work is to establish the local H\"older regularity and Harnack type inequalities for the weak solutions to a nonlocal equation 
\begin{equation}\label{main_PDE}
(-\Delta_p)^{s+\log}u = f  \text{ in } \Omega,  \quad f\in L^{q}_{\mathrm{loc}}(\Omega),\quad q>1.
\end{equation}
For a given differentiable parameter $s \in (0,1)$ and an integrability exponent $p \in (1, \infty)$, the operator $\logp$ is the differentiation of $(- \Delta_p)^s$ with respect to $s$, i.e., 
\begin{align*}
    \logp u(x) := \frac{{\rm d}}{\dt} (- \Delta_p)^t u(x) \bigg|_{t=s},
\end{align*}
where $x \in \Omega$, and $t \in (0,1).$ The fractional $p$-Laplace operator $(- \Delta_p)^t$ is defined as
\begin{align*}
    (- \Delta_p)^tu(x) := C_{d,t,p}  \text{ P.V. } \int_{\rd} \frac{\abs{u(x)-u(y)}^{p-2}(u(x)-u(y))}{\abs{x-y}^{d+tp}} \dy,  
\end{align*}
where $C_{d,t,p}>0$ is the normalization constant and P.V. means ``in the principal value sense''. 
For $u \in \C_c^2(\rd)$ and $x \in \rd$, Bahrouni et. al. in \cite{logp} derived the following integral representation:
\begin{align*}
   \logp u(x) = B_{d,s,p} (- \Delta_p)^s u(x) 
    - pC_{d,s,p} \text{ P.V. } \int_{\rd} \frac{\abs{u(x)-u(y)}^{p-2}(u(x)-u(y))\log(|x-y|)}{\abs{x-y}^{d+sp}} \dy, 
\end{align*}
where 
\begin{align*}
    B_{d,s,p} := \frac{{\rm d}}{\dt} \log(C_{d,t,p}) \Big|_{t=s}= 2 \log 2+\frac 1s+\frac p2\psi\left(\frac{d+sp}{2}\right)+\psi(1-s).
\end{align*}
We emphasize that, to define $\logp u(x)$ pointwisely for $u \in \C_c^2(\rd)$, the range $s<2(p-1)/p$ is required when $p<2$ (see Remark \ref{well-definedness}). 
Further, since $B_{d,s,p} \ra +\infty$ as $s \ra 0$, there exists a threshold $s_0 <\min\left\{2(p-1)/p,1\right\}$ such that
$$
B_{d,s,p} \geq 0, \; \text { for all } s \in\left(0, s_0\right) .
$$

When $p=2$, the idea of differentiating the classical fractional Laplace operator $(- \Delta)^s$ at $s=0$ goes back to Chen-Weth \cite{ChenWeth}. They considered the following operator
\begin{align*}
    L_{\Delta}u(x):= \frac{{\rm d}}{\ds}(-\Delta)^{s}\Big|_{s=0} = \lim_{s \ra 0} \frac{(-\Delta)^s u- u}{s}.
\end{align*}
From the Fourier representation, $\mathcal{F}((-\Delta)^s u) (\xi) = |\xi|^{2s} \hat{u}(\xi)$, where $\xi \in \rd$, one gets $\mathcal{F}(L_{\Delta}u)(\xi) = (2 \log |\xi| ) \hat{u}(\xi)$, which indicates that the logarithmic Laplacian is a weakly singular Fourier integral operator.
The operator $L_{\Delta}$ appears in the study of the asymptotic behavior of $(-\Delta)^su$ as $s \to 0$. Given a function $u \in \C^{\alpha}_c(\rd)$ for some $\alpha>0,$ $L_{\Delta}u(x)$ can be uniquely defined by the asymptotic expansion:
\begin{align*}
    (-\Delta)^s u(x)=u(x)+sL_{\Delta}u(x)+o(s),
\end{align*}
where $o(s)$ denotes the small $o$ that satisfies $\frac{o(s)}{s} \to 0$ as $s\to 0.$ So, $L_{\Delta}u$ appears as the linear correction term in the small-order regime of the fractional Laplacian. The operator admits an explicit (weakly singular) integral representation $L_{\Delta}u(x)$, which, for $u\in \C^{\alpha}_c(\rd)$, $\alpha>0$, and $x \in \rd$, is given by
\begin{align*}
    L_{\Delta}u(x)=C_d \text{P.V.}\int_{B_1(x)}\frac{u(x)-u(y)}{|x-y|^d}\dy-C_d\int_{\rd \setminus B_1(x)}\frac{u(y)}{|x-y|^{d}}\dy+\rho_d u(x).
\end{align*}
Here, $B_r(x)$ is the ball centered at $x$ with radius $r>0$ and 
\begin{align*}
    C_d=\frac{\Gamma(\frac{d}{2})}{\pi^{\frac{d}{2}}}=\frac{2}{\omega_{d-1}}, \quad \rho_d=2 \log 2+\psi\left(\frac{d}{2}\right)-\gamma,
\end{align*}
where, $\omega_{d-1}$ denotes the $(d-1)$-dimensional volume of the unit sphere in $\rd$,  $\gamma=-\Gamma^{\prime}(1)$ is the Euler-Mascheroni constant, and $\psi=\frac{\Gamma^{\prime}}{\Gamma}$ is the Digamma function. For $p \in (1, \infty)$, the quasilinear counterpart of logarithmic Laplacian
\begin{align*}
        L_{\Delta_p}u(x):= \frac{{\rm d}}{\ds}(-\Delta_p)^{s}\Big|_{s=0}, 
\end{align*}
was studied by Dyda-Jarohs-Sk \cite{DydaJarohsSk}. These operators can be viewed as zero order nonlocal operators. Concerning regularity of solutions for zero order nonlocal equations, Jarohs-Kassmann-Weth \cite{JarohsKassmannWeth2026} (also see the work of Kassmann-Mimica \cite{KassmannMimica2014}) proved the local boundedness and continuity for a very large class of kernels, whereas classical solutions are studied by ChangLara-Salda\~na \cite{ChangLaraSaldana2024}, and boundary regularity is studied by Hern\'andez-Santamaría-L\'opez Ríos-Saldaña \cite{HernandezSantamariaLopezRiosSaldana2025}. Also, in \cite{FeulefackJarohs2023} Feulefack-Jarohs showed that if the right hand side of nonlocal equation with small order is smooth, then the solution is smooth.

All of these studies are associated with the endpoint $s=0$. Recently, Chen-Chen-Hauer \cite{ChenChenHauer} took another step of differentiating
$(-\Delta)^{t}$ with respect to order $t$ at an interior point $s\in(0,1)$, rather
than at the endpoint $s=0$. This yields the  fractional logarithmic Laplacian operator
$$(-\Delta)^{s+\log}=\frac{{\rm d}}{\dt}(-\Delta)^{t}\big|_{t=s},$$ the
first order term in
$$(-\Delta)^{t}u=(-\Delta)^{s}u+(t-s)(-\Delta)^{s+\log}u+o(|t-s|).$$
For $p=2$, the fractional logarithmic operator has a clean Fourier representation.
There $(-\Delta)^{s+\log}$ acts as the Fourier multiplier with symbol
$|\xi|^{2s}\log|\xi|^{2}$-the symbol $|\xi|^{2s}$ of $(-\Delta)^{s}$
weighted by the logarithmic factor $\log|\xi|^{2},$ which is what the
notation reflects. As $s\to0$ this symbol reduces to $\log|\xi|^{2}$ which is
the symbol of $L_{\Delta}.$ Thus, the two operators match as $s\to 0.$ Although, the operator $(-\Delta)^{s+\log}$ appeared in studying the properties of $\partial_s u_s,$ where $u_s$ solves the Dirichlet problem of fractional Laplacian, some time ago \cite{JarohsSaldanaWeth2020, JarohsSaldanaWeth2025} (also see \cite{JarohsSenWeth2026} for a recent contribution), the specific study on  $(-\Delta)^{s+\log}$ with $0<s<s_0$ is relatively new and initiated by Chen-Chen-Hauer \cite{ChenChenHauer}. They provided various representation formulas and studied the extension problem for this operator.  Also, using Moser iteration, it is shown that $u \in L^{\infty}(\Omega)$, whenever $u$ weakly solves the Dirichlet eigenvalue problem $(-\Delta)^{s+\log}u=\la u$ in $\Omega$, with $\la \in \R$. 
For its quasilinear variant,  $(-\Delta_p)^{s+\log},$ various inequalities, the Dirichlet eigenvalue problem and the corresponding boundedness result have been studied very recently by Bahrouni-Gouasmia-Hajaiej-Ouannasse \cite{logp}. 

The regularity questions for the logarithmic type nonlocal equations happen to be quite delicate in nature. Formally, for $p=2$ and suitable $u,$ from the semigroup property \cite[Theorem 1.1]{ChenChenHauer}, we see that
\begin{align*}
    (-\Delta)^{s+\log}u=(-\Delta)^s L_{\Delta}u=L_{\Delta}  (-\Delta)^s u=f.
\end{align*}
From the perspective of Schauder theory, if we consider the second identity, i.e., $L_{\Delta}(-\Delta)^s u=f,$ then we know from \cite{JarohsKassmannWeth2026} that if $f\in L^{\infty}(\Omega)$ then $(-\Delta)^s u$ is continuous, which ensures that $u$ is H\"older continuous. Through this direction, no analogous result currently exists for $f$ in the locally scaling subcritical Lebesgue space. On the other hand, if we consider $(-\Delta)^s L_{\Delta}u=f$, with locally scaling subcritical $f$, from \cite{Br2018}, we know that $L_{\Delta}u$ is locally H\"older continuous. Although, it is still unknown whether $L_{\Delta}u$ is locally H\"older continuous implies $u$ is locally H\"older continuous.

Due to nonlinearity, the semigroup property even fails when $p\neq 2$. These observations naturally lead to the following questions:

\begin{center}
\begin{minipage}{14cm}
 \begin{enumerate}
 \item[{\bf Question 1.}] Can one obtain a Harnack type inequality for weak solutions to \eqref{main_PDE}?\vspace{5mm}
\item[{\bf Question 2.}] Can one establish a systematic local regularity theory for weak solutions to \eqref{main_PDE}? 
\end{enumerate}   
\end{minipage}
\end{center}

\subsection{Main results} Our main results answer the above questions affirmatively. To begin with, 
we recall the definition of local weak (super and sub) solutions of \eqref{main_PDE}.
\begin{definition}[Local weak solution]\label{weaksoldef}
A function $u\in W_{\text{loc}}^{s+\log,p}(\Omega) \cap L_{\log,p}(\rd)$ is called a local weak supersolution (subsolution)  to \eqref{main_PDE} if for every compact $K \subset \subset \Omega$ and non-negative $v\in W_0^{s+\log,p}(K)$, it holds
\begin{align}\label{weak1}
 &\iint_{\rd \times \rd} \K(|x-y|) |u(x)-u(y)|^{p-2}(u(x)-u(y)) (v(x) -v(y)) \dxy\no\\
  \quad \quad &\geq \,(\leq )\, \int_{\Omega} f(x)v(x) \dx.
\end{align}
We say $u$ is a local weak solution if the equality holds in \eqref{weak1} for every compact $K \subset \subset \Omega$ and non-negative $v\in W_0^{s+\log,p}(K)$.
\end{definition}
Moreover, we consider the scaling subcritical $f$ defined as
\begin{equation}\label{weight}
f \in L^q_{\text{loc}}(\Omega) \text{ with } \left\{\begin{aligned}
&q>\frac{d}{sp},\,&\text{if}\;d>sp; \\
&q> 1,\,&\text{if}\;d\le sp.
\end{aligned}
\right.
\end{equation}

Now, we state our first main result, the Harnack inequality.
\begin{theorem}[Harnack inequality with tails]{\label{harnack_intro}}
Let $s \in (0,s_0), d>sp$, and $f$ be as given in \eqref{weight}. We consider $B_R(x_0)\subset \Omega$ with $\text{diam}(B_R(x_0))\leq 1$.  Assume that $u$ is a weak solution to \eqref{main_PDE} satisfying $u \geq 0$ in $B_R(x_0)$. Let $0<r\leq \frac{1}{24}$ be such that $B_r(x_0) \subset B_{\frac R{16}}(x_0)$. Then there exists $C=C(d,s,p)>0$ such that the following holds: 
\begin{align*}
    \underset{B_{\frac{r}{2}}(x_0)}{\esssup}\, u \le C \,\underset{B_r(x_0)}{\essinf}\, u &+ C \left(1+\log \frac{1}{R}\right)^{-\frac{1}{p-1}}\left(\left(\frac{r}{R}\right)^{\frac{sp}{p-1}} \left[\operatorname{Tail}_+\left(u_{-} ; x_0, R\right)+\operatorname{Tail}_-\left(u_{+} ; x_0, R\right)\right] \right) \no \\
    & + C \left(1+\log \frac{1}{R}\right)^{-\frac{\sigma}{p \sigma -1}} R^{(sp-\frac{d}{q})\frac{1}{p-1}} \norm{f}_{L^q(B_R(x_0))}^{\frac{1}{p-1}},
\end{align*}
where $\sigma=\frac{(dp-d+sp)q-d}{(dp-d+sp)q-dp}$ and $p\sigma<\ps$.
\end{theorem}
\begin{remark}
(i) Our Harnack inequality is structurally different from the Harnack inequality for the fractional $p$-Laplacian. There, global nonnegativity $u \geq 0$ in $\rd$ removes the tail term entirely and reduces the nonlocal Harnack to the local one. In our setting, even when $u \geq 0$ in $\rd$, the term $\operatorname{Tail}_{-}(u; x_0, R)$ still stays on the right hand side. The presence of one part of the tail term reflects the sign changing structure of the kernel; its negative part continues to couple with the exterior (positive) values of $u$ even in the absence of any negative part of $u$.

(ii) From our Harnack inequality, one further recovers a purely local, tail-free Harnack inequality whenever $u$ is compactly supported in $B_R(x_0)$ with $\text{diam}(B_R(x_0))\le 1$. This seems to be a natural condition for a sign changing kernel. We also provide a counterexample showing that this assumption is essentially sharp. In particular, if $u$  is supported outside of $B_R(x_0),$ the tail terms cannot, in general, be removed.

(iii) It should be noted that our constant $C$ depends on $s$, and we do not know the explicit form of dependence. Thus, our estimates are not stable as $s\to 0^{+}$.
We further note that, in contrast with the fractional $p$-Laplacian,  the limit as $s\to 1^{-}$ of $(-\Delta_p)^{s+\log}$ does not converge to a  ``local'' operator. Moreover, the ``local'' Harnack inequality of (ii) is to be understood as the tail-free estimate under a support condition, rather than as the Harnack inequality of any limiting local counterpart.
\end{remark}


Our next result is the local H\"older estimate for the weak solutions to \eqref{main_PDE}.
\begin{theorem}[H\"older regularity]\label{thm: holder estimate}
 Let $s\in(0,s_0),$ $d>sp$ and $f$ be as given in \eqref{weight}. Let $u$ be a weak solution of \eqref{main_PDE}. Then $u$ is locally H\"older continuous. More precisely, there exist $$\alpha=\alpha(d,s,p,q)\in \left(0, \frac{sp-\frac{d}{q}}{p-1}\right),$$ and a constant $C=C(d,s,p,q)>0$ such that, whenever $B_R(x_0)\subset\Omega$ with $\text{diam} (B_R(x_0))\leq 1$, the following estimate holds:
\begin{align*}
\mathop{\mathrm{osc}}_{B_r(x_0)}u & \le C\Big(\frac{r}{R}\Big)^{\alpha}
\Bigg(\left(1+\log \frac{1}{R}\right)^{-\frac{1}{p-1}}\operatorname{Tail}_{\mathrm{mod}}\big(u;x_0,\frac{R}{2}\big)+\Big(\fint_{B_R(x_0)}|u|^{p\sigma}\,dx\Big)^{\frac1{p\sigma}}\no \\
& \quad +\left(1+\log \frac{1}{R}\right)^{-\frac{\sigma}{p\sigma-1}}R^{\left(sp-\frac dq\right)\frac{1}{p-1}}\norm{f}_{L^q(B_R(x_0))}^{\frac1{p-1}}\Bigg),
\end{align*}
for every $r\in(0,\frac R2]$, where $\sigma=\frac{(dp-d+sp)q-d}{(dp-d+sp)q-dp}$ and $p\sigma<\ps$.   
\end{theorem}
\begin{remark}
(i) The restrictive range of the H\"older exponent arises due to the nonhomogeneous term $f$. In particular, if $f$ is locally bounded, then we get H\"older regularity for some $\al \in \left(0, sp/(p-1)\right)$. 

(ii) We consider \eqref{main_PDE} with an extra homogeneous term containing a bounded weight $V$, i.e., 
\begin{align}\label{plusV}
        (-\Delta_p)^{s+\log}u +V(x)|u|^{p-2}u=f, \quad V \in L^{\infty}(\Omega).
\end{align}
Following the proof of Theorem \ref{harnack_intro} and Theorem \ref{thm: holder estimate} for weak solutions to \eqref{plusV}, we observe that the presence of a bounded potential term $V$ does not influence the local behavior of $u$. In fact, we get a similar Harnack estimate and H\"older regularity (with the same range of $\alpha$) for $u$. 
\end{remark}
\begin{remark}
We mention that the statements of the main results include the assumption $d>sp,$ however, Theorems~\ref{harnack_intro} 
and~\ref{thm: holder estimate} remain valid for $d\leq sp$. Indeed, if $d\leq sp$, then $q>\frac{d}{sp}$ holds for every $q>1$, and in Proposition \ref{poincare} the exponent $p^*_s$ will be replaced by any $\ell$ for which the embedding $W^{s+\log, p}(B_r)\hookrightarrow W^{s,p}(B_r)\hookrightarrow L^{\ell}(B_r)$ holds, namely $\ell\in[1,\infty]$ when $d<sp$ and $\ell\in[1,\infty)$ when $d=sp$. These changes will also persist in the local boundedness result.
\end{remark}
\subsection{Related works} 
The regularity theory for nonlocal equations, particularly Harnack inequalities and H\"older regularity, has attracted considerable attention in recent years. The classical De Giorgi-Nash-Moser theory has been extensively and  successfully extended to integro-differential operators with symmetric kernels. We only mention some of them below, and our list is necessarily incomplete.

Beginning with pioneering works \cite{Kas09, DK20} and continuing through \cite{KuusiHarnack}, \cite{Palatucci2016}, and \cite{Coz17}, H\"older continuity estimates and Harnack inequalities for weak solutions of the corresponding elliptic equations have been established via nonlocal adaptations of the Moser and De Giorgi iteration techniques. 

Subsequently, these results have been extended to broader classes of operators, including those with singular jumping measures \cite{CK20} and nonlinear operators exhibiting nonstandard growth conditions \cite{FangZhang2023, CK23, BOS22, CKW22a, CKW22b, BKO22, Ok23}.

We further mention the contributions in \cite{Br2018, Anup2025, BDLM2025, BOGELEIN2025111078, Bogelein2025Gradient, Garain2024, Palatucci2016, giova2025}, where significant developments concerning interior H\"older regularity, higher Sobolev and H\"older regularity, Lipschitz regularity, and $\mathcal{C}^{1,\alpha}$ regularity for weak solutions to $(-\Delta_p)^s u =f$ in $\Omega$, (for a suitable function $f$) have been investigated.

\subsection{Strategy of the proof}
The regularity theory for \eqref{main_PDE} relies on the nonlocal extension of De Giorgi-Nash-Moser
theory by Di Castro-Kuusi-Palatucci \cite{KuusiHarnack,Palatucci2016}. The two main ingredients of this theory are a local boundedness estimate (Proposition \ref{local-boundedness-I}) for subsolutions and a weak Harnack inequality (Proposition \ref{weakharnack-1}) for supersolutions, which is also known as the half Harnack inequality. As a consequence of the local boundedness and a tail control (Proposition \ref{Tail}), we derive the other half of the Harnack inequality. Finally, the combination of these two halves, gives the full Harnack inequality.

The H\"older estimate is built on an argument of De Giorgi iteration, which leads to a reduction of oscillation in smaller balls. To run the De Giorgi iteration, we give a lower bound of the so-called energy by using a logarithmic version of the fractional Sobolev-Poincar\'e inequality (Proposition \ref{poincare}) and an upper bound is obtained by estimating the right-hand side terms of the Caccioppoli inequality (Lemma \ref{energy estimate}). The kernel of the operator 
\[
   \K(r) = C_{d,s,p}\,\frac{B_{d,s,p}-p\log r}{r^{d+sp}}, \quad r>0,
\]
differs from that of the fractional $p$-Laplacian in two ways. It is more singular than
$r^{-d-sp}$ near the diagonal, and it changes sign at large scales:
\begin{center}
$\K(r)\le 0$\,\,\, as soon as\,\,\, $r\ge e^{\frac{B_{d,s,p}}{p}}$.
\end{center}
Both features dictate the long range part of the analysis to be reorganized.

\smallskip
\textbf{$\bullet$ A single tail does not suffice.} In the regularity of nonlocal equations, usually, the long range interaction is recorded by a single nonlocal tail $\operatorname{Tail}\,(u;x_0,R)$ (see, for instance \eqref{full tail}).
Here, this is not the right object. Since $\K$ becomes negative on $\{|x_0-y|\ge e^{B_{d,s,p}/p}\}$, the integrand changes sign and $\operatorname{Tail}(u;x_0,R)$ no longer captures the far field error. The negative contributions cancel against the positive ones, so this quantity neither dominates the error terms produced in the energy estimates nor can be fitted in the iteration argument any more. To record the true effect of the far field, we split the kernel into its fractional parts and its two sign definite logarithmic parts,
\[
\K(r) \sim \,r^{-d-sp} +\Kp(r) -\Km(r), \qquad \Kp,\,\,\Km \ge 0.
\]
 We measure the far field effect by two tails, each built from the positive and negative parts of the logarithmic weight:
\begin{align*}
    \operatorname{Tail}_{\pm} (u; x_0, R)\sim \int_{\rd\setminus B_{R}(x_0)}\frac{1+\left(\mp \log (|x_0-y|)\right)_{+}}{|x_0-y|^{d+sp}}|u(y)|^{p-1} \dy.
\end{align*}
The positive part $\Kp$ is supported where the logarithm is positive, i.e, on the intermediate scales $|x_0-y|<1$, and records the near to intermediate field interactions, whereas the negative part $\Km$ is supported on the far field $\{|x_0-y|>1\}$, where the kernel turns negative. Both tails are finite on the natural tail space $L_{\log,p}(\rd)$ (Remark~\ref{tail remark}), and being sign definite, they fit into the iteration process as required.

\smallskip
\textbf{$\bullet$ How the two tails are controlled.} Motivated by the Caccioppoli inequality, in every estimate, the two tails enter together and the pairing is dictated by the sign of $u$ and the signed part of the kernel. For subsolutions, a closer look at the right hand side of the Caccioppoli inequality \eqref{cacc} says that
the second and third terms should play the role of $\operatorname{Tail}_{+}$ on $u_{+}$ and $\operatorname{Tail}_{-}$ on $u_{-}$ respectively. Employing the De Giorgi iteration on these yields the local boundedness estimate (Proposition~\ref{local-boundedness-I}) which looks like
\[
   \operatorname*{ess\,sup}_{B_{\frac{r}{2}}(x_0)} u
   \;\lesssim\; \delta\bigl(1-\log\frac r2\bigr)^{-\frac1{p-1}}
   \Bigl[\operatorname{Tail}_{+}(u_+;x_0,\frac r2)+\operatorname{Tail}_{-}(u_-;x_0,\frac r2)\Bigr]
   + \cdots .
\]
As anticipated, for nonnegative supersolutions, the same duality reappears in the
logarithmic estimate (Lemma~\ref{log-estimate}) and in the growth lemma (Proposition~\ref{growth-lemma}) with the opposite pairing. However, the bridge between these two tail pairings is the tail estimate lemma in Proposition \ref{Tail} for solutions, where the tail pairings coming from the subsolutions are controlled by the tail pairings coming from the supersolutions, i.e., 
\begin{align*}
 \operatorname{Tail}_{+}(u_{+};x_0,R)+\operatorname{Tail}_{-}(u_{-};x_0,R) \lesssim C(r, R) \left[\operatorname{Tail}_{+}(u_-;x_0,R)+\operatorname{Tail}_{-}(u_+;x_0,R)\right]+ \cdots.
\end{align*}
Thus in the final form of Harnack inequality or H\"older estimate, we only see the pairing
\[
   \operatorname{Tail}_{+}(u_-;x_0,R)+\operatorname{Tail}_{-}(u_+;x_0,R).
\]
The mechanism behind the pairing can be understood in the following way. Far from the domain $\Omega,$ the contribution coming from the negative part of the solution is taken care of by the positive part of the kernel and therefore it produces the term $\operatorname{Tail}_{+}(u_-;x_0,R)$. On the other hand, the far field contribution coming from the positive part of the solution is taken care of by the negative part of the kernel that produces the term $\operatorname{Tail}_{-}(u_+;x_0,R)$. Thus, the natural quantity that controls the influence of the exterior error is precisely the sum of two tails, 
\[
   \operatorname{Tail}_{+}(u_-;x_0,R)+\operatorname{Tail}_{-}(u_+;x_0,R).
\]
The second term is the new feature. For the fractional $p$-Laplacian the kernel is positive everywhere, and hence only $\operatorname{Tail}_{+}(u_-, x_0, R)$ survives. Once the kernel changes sign, we need $\operatorname{Tail}_{-}(u_{+}; x_0, R)$ as compensation.

\smallskip
\textbf{$\bullet$ The $\log$-factor.} Each one of the estimates above carries a sharp factor $\bigl(1+\log\frac1R\bigr)^{-\frac1{p-1}}$ in front of the tails. The reason is not merely technical, but the exact rate at which the logarithmic growth of the kernel is absorbed. The source of this log factor is the fractional logarithmic Poincar\'e inequality (Proposition~\ref{poincare}). The extra $-\log r$ in $\K$ introduces a logarithmic growth into every local energy
that produces a factor $(1-\log r)$ at scale $r$ (see \eqref{int-1}). Proposition~\ref{poincare} supplies the matching decay. Since every estimates in $B_r\times B_r$ satisfies $|x-y|<2r<1$,
one has $-\log|x-y| > (-\log 2r)_+$ and inserting this ratio converts the Gagliardo seminorm into the $\Kp$ energy at the cost of a logarithm in the denominator that
 exactly balances the $(1-\log r)$ growth above. The quantity $g(r)=(1-\log r)/(-\log(2r))$ stays bounded for small radii, so the De Giorgi iterations can be run in a scale invariant fashion. The logarithmic growth also propagates into the estimate of the tails,
\begin{align*}
    [\operatorname{Tail}_{\mathrm{mod}}(u, x_0, r)]^{p-1}\lesssim\, \log \left(1+\frac{1}{r}\right)\omega(r),
\end{align*}
and at small scales, factors like $\bigl(1+\log\frac1R\bigr)^{-\Theta(p, \sigma)}$ for some $\Theta(p, \sigma)>0$ appear in the final results.

\smallskip
\textbf{$\bullet$ Effect of an unbounded potential.} Due to the presence of $f \in L_{\text{loc}}^q(\Omega)$ with $q>d/sp$, the statement of local boundedness (see \eqref{loc.bounded}) introduces a parameter $\sigma$, where the prescribed integrability of $f$ ensures the strict inequalities $\sigma>1$ and $p\sigma<p_s^*$, which are vital for the convergence in the iteration process (Lemma \ref{iteration}).  Indeed, if $f$ lies in the scaling critical space $L_{\text{loc}}^{d/sp}(\Omega)$, then $p\sigma = p_s^*$ causes the iteration lemma to fail, thereby precluding the derivation of local boundedness via this framework. The exponent $\sigma$ is again a key ingredient in establishing the expansion of positivity and H\"{o}lder regularity. The strict inequality $p\sigma <p_s^*$ is also required at these stages. 

\subsection{Outline} The article is organized in the following way. In Section \ref{sec:priliminaries}, we collect the necessary definitions and preliminary results that are used throughout the paper. Section \ref{sec:energy lemma and local bd} is focused on proving the energy estimate and local boundedness result. In Section \ref{sec:log lemma and growth}, we prove a logarithmic estimate and a growth type lemma. Section \ref{sec: harnack} and Section \ref{sec: Holder} are dedicated to completing the proof of Harnack inequality and H\"older regularity estimate, respectively. In Section \ref{sec: counterexample}, we give an example showing that if a solution is supported outside $B_R$, then the tail term in the Harnack inequality cannot be avoided.

\section{Preliminaries} \label{sec:priliminaries}
\noi \textbf{Notation:} We start by gathering the following notation and conventions. 
\begin{enumerate}[(i)]
    \item For $d>sp$, $p^*_s=\frac{dp}{d-sp}$ is the fractional critical exponent.
    \item For any $r>0$, the ball $B_{r}:= B_{r}(x_0)$ and $B_r^c := \rd \setminus B_{r}(x_0)$. 
    \item The positive and negative part of function $f$, are $f_+= \max\{f,0\}$ and $f_-=\max\{-f,0\}$.
    \item For brevity, $C=C(a,b,c,\cdots)$ denotes a generic positive constant that varies from line to line. 
\end{enumerate}

As mentioned before, in the regularity theory for \eqref{main_PDE}, the global behaviour of weak solutions plays a crucial role, which is captured by a special quantity called \textit{nonlocal tail}. Usually, the \textit{nonlocal tail} of a function $u$ is defined as 
\begin{align}\label{full tail}
  \operatorname{Tail}(u;x_0,R) = \left[ R^{sp} \int_{\rd \setminus B_R(x_0)} \K(|x_0-y|) |u(y)|^{p-1} \dy \right]^{\frac{1}{p-1}},
\end{align}
where $x_0\in\rd$ and $R>0$, with the kernel $\K$ given by 
\begin{align*}
   \K(r) := C_{d,s,p}\frac{B_{d,s,p}-p\log(r)}{r^{d+sp}}, \text{ for } r>0. 
 \end{align*} 
Set 
\begin{align*}
    \Kpm(r) := C_{d,s,p} \frac{(- \log (r))_{\pm}}{r^{d+sp}}, \text{ for } r>0. 
\end{align*}
and 
\begin{align*}
 \mathcal{K}^{s+\log}_{\mathrm{mod}}(r)=p\Kp(r) +p\Km(r).   
\end{align*}
Now the following decomposition holds
\begin{align}\label{split-2}
    \K(r) = C_{d,s,p}B_{d,s,p}r^{-d-sp} + p\Kp(r) - p\Km(r). 
\end{align}
As already discussed in the introduction, \eqref{full tail} is not useful in the analysis and also since $\K$ changes sign on $\rd \setminus B_R(x_0)$, the 
integral may be negative, so that the $(p-1)$-th root is not even well-defined. 
So we need to introduce the \textit{decomposed tails} in the following way:
\begin{align*}
   & \operatorname{Tail}_{\mathrm{mod}}(u;x_0,R) :=  \left[R^{sp} \int_{\rd \setminus B_R(x_0)} \left(\frac{1}{|x_0-y|^{d+sp}}+\mathcal{K}^{s+\log}_{\mathrm{mod}}(|x_0-y|)\right) |u(y)|^{p-1}\dy\right]^{\frac{1}{p-1}}, \\
   & \operatorname{Tail}_{+}(u;x_0,R) :=  \left[R^{sp} \int_{\rd \setminus B_R(x_0)} \left(\frac{1}{|x_0-y|^{d+sp}}+\Kp(|x_0-y|) \right)|u(y)|^{p-1}\dy\right]^{\frac{1}{p-1}},\\
   & \operatorname{Tail}_{-}(u;x_0,R) :=  \left[R^{sp} \int_{\rd \setminus B_R(x_0)} \left(\frac{1}{|x_0-y|^{d+sp}}+\Km(|x_0-y|) \right)|u(y)|^{p-1}\dy\right]^{\frac{1}{p-1}}.
\end{align*}
Next, we define the following tail space
\begin{align*}
   & L_{\log,p}(\rd) := \left\{ u \in L_{\text{loc}}^{p-1}(\rd) : \int_{\rd}  \frac{\left(1+\log(1+|y|)\right)}{(1+ \abs{y})^{d+sp}} |u(y)|^{p-1} \dy < \infty \right\}.
\end{align*}
\begin{remark}\label{tail remark}
In this remark, we see the behaviour of \textit{nonlocal tail} of a function which lies in tail spaces.
For $y \in \rd \setminus B_R(x_0)$, using 
\begin{align*}
    \frac{1+|y|}{|x_0-y|} \le 1 + \frac{1+|x_0|}{R}, \text{ and } \frac{|x_0-y|}{1+|y|}\leq 1+|x_0|, 
\end{align*}
it holds
\begin{align*}
    &\frac{(-\log(|x_0-y|))_+}{|x_0-y|^{d+sp}} \le \frac{C(x_0, R)}{(1+|y|)^{d+sp}}, \text{ and } \\
    & \frac{(-\log(|x_0-y|))_{-}}{|x_0-y|^{d+sp}} \le C(x_0, R) \frac{\left(1+\log(1+|y|)\right)}{(1+|y|)^{d+sp}},
\end{align*}
for some $C(x_0,R)>0$. In view of the above inequality, $\mathrm{Tail}_{+}(u;x_0,R)< \infty$ and $\operatorname{Tail}_{-}(u, x_0, R)< \infty$ for every $u \in L_{\log,p}(\rd).$
\end{remark}

\begin{remark}\label{support-positive-kernel}
This remark verifies that $\Km(|x_0-y|)$ is integrable over $\R^d\setminus B_r(x_0)$.
Using $\log (\tau)\leq \frac{1}{e \kappa}\tau^{\kappa}$ for any fixed $0<\kappa< sp,$ we observe that  
    \begin{align*}
    \int_{\R^d\setminus B_r(x_0)}\Km(|x_0-y|) \dy = C_{d,s,p} \int_{1}^{\infty}\frac{\tau^{d-1}\log (\tau)}{\tau ^{d+sp}} \dtau &\leq\frac{C_{d, s,p}}{e\kappa}\int_{1}^{\infty}\tau^{\kappa-1-sp} \dtau\\&=\frac{C_{d, s,p}}{e\kappa}\frac{\tau^{\kappa-sp}}{\kappa-sp}\Bigg|^{\infty}_1=\frac{C_{d,s,p}}{e\kappa(sp-\kappa)}=:\widetilde C \,\text{(say)}.
\end{align*} 
\end{remark}

\begin{remark}\label{well-definedness}
In order to define $ \logp u$ for $u \in \mathcal{C}_c^2(\rd)$, we first require the pointwise definition of $(-\Delta_p)^tu$. When $1<p<2$, one can take $u \in \C_{c}^{1, \alpha}(\rd)$ with $\al \in (0,1]$, and follow \cite[Lemma 2.11]{Iannizzotto2016}, to see
\begin{align*}
      \left|(-\Delta_p)^tu(x)\right| 
    &\leq\frac{\left||u(x)-u(x+z)|^{p-2}(u(x)-u(x+z)) + |u(x)-u(x-z)|^{p-2}(u(x)-u(x-z))\right|}{|z|^{d+tp}} \\
    &\le C |z|^{(\al+1)(p-1)-d-tp}. 
\end{align*}
Now note that for any $R>0$, $|z|^{(\al+1)(p-1)-d-tp} \in L^1(B_R(0))$ whenever $\al>\frac{1-p(1-t)}{p-1}$. To ensure $\al \le 1$, we must require
\begin{align*}
    \frac{1-p(1-t)}{p-1} <1 \Longleftrightarrow t<\frac{2(p-1)}{p}, \text{ where } \frac{2(p-1)}{p} <1 \Longleftrightarrow p<2. 
\end{align*}
So for $p<2$, $(-\Delta_p)^tu$ is pointwisely defined whenever $t<\frac{2(p-1)}{p}$.  
\end{remark}

Now, we consider the following function space, due to \cite{logp}:
\begin{align*}
    W^{s+\log,p}(\rd) := \left\{ u \in L^p(\rd) :[u]_{s,\log,p} < \infty \right\},
\end{align*}
where the seminorm $[\cdot]_{s,\log,p}$ is defined as 
\begin{align*}
    [u]_{s,\log,p} := \left(\iint_{\rd \times \rd} \Kp(|x-y|) |u(x) - u(y)|^p \dxy\right)^{\frac{1}{p}}.
\end{align*}
Let $U$ be a bounded open set in $\rd$. We consider the following closed subspace of $ W^{s+\log,p}(\rd)$:
\begin{align*}
    W_0^{s+\log,p}(U) = \left\{ u \in W^{s+\log,p}(\rd) : u =0 \text{ a.e. in } \rd \setminus U \right\}.
\end{align*}
Note that, the seminorm $[\cdot]_{s,\log,p}$ is an equivalent norm in $W_0^{s+\log,p}(U)$. Moreover, as $U$ is bounded, for $d> sp$, the embedding $W_0^{s+\log,p}(U)\hookrightarrow W_0^{s,p}(U) \hookrightarrow L^q(U)$ holds, for every $q \in [p, p^*_s]$ (for proofs, see \cite{logp}).
Next, we consider the following spaces
\begin{align*}
     W^{s+\log,p}(U) := \left\{ u \in L^p(U) :\iint_{U \times U} \Kp(|x-y|) |u(x) - u(y)|^p \dxy < \infty \right\}
\end{align*}
\text{ and }
\begin{align*}
    & W_{\text{loc}}^{s+\log,p}(U) := \left\{ u \in L^p(K) :\iint_{K \times K} \Kp(|x-y|) |u(x) - u(y)|^p \dxy < \infty \right\},
\end{align*}
for every compact set $K \subset U$. As
$$\iint_{U \times U} \frac{|u(x) - u(y)|^p }{|x-y|^{d+sp}}\dxy \leq \frac{1}{(-\log(\operatorname{diam}U))_+}\iint_{U \times U} \Kp(|x-y|) |u(x) - u(y)|^p \dxy, $$ the embedding $W^{s+\log,p}(U)\hookrightarrow W^{s,p}(U)$ holds.

\begin{remark}\label{C_c} We observe that if $u\in W_{\text{loc}}^{s+\log,p}(U)$ and $\phi\in \mathcal{C}^\infty_c(U)$, then $u\phi\in W_{0}^{s+\log,p}(U)$.
 It is clear that $u\phi\in L^p(U)$. Now take $U_2=\operatorname{supp}(\phi)$  and $U_1$ such that $U_2\subset \subset U_1 \subset\subset\Om$. We split 
\begin{align}\label{test-re-1}
    &\iint_{\rd \times \rd} \Kp(|x-y|) |u(x)\phi(x) - u(y)\phi(y)|^p \dxy\no\\=&\iint_{U_1 \times U_1} \Kp(|x-y|) |u(x) \phi(x)- u(y)\phi(y)|^p \dxy\no\\&+2\iint_{U_1 \times \rd \setminus U_1} \Kp(|x-y|) |u(x)\phi(x) - u(y)\phi(y)|^p \dxy.
\end{align}
Using $\abs{\phi(x)- \phi(y)} \le C \abs{x-y}$ for every $x,y \in U_1$, we see that 
\begin{align*}
    & \int_{U_1}\int_{ U_1} \Kp(|x-y|) |u(y)|^p |\phi(x)-\phi(y)|^p \dxy \\
    & \le C(d,s,p) \int_{ U_1} \left( \int_{U_1} \frac{(- \log (|x-y|))_{+}}{|x-y|^{d+sp-p}} \dx \right) |u(y)|^p \dy < \infty, 
\end{align*}
where the finiteness comes using $u \in L^p(U_1)$, and 
 \begin{align*}
     \int_0^1 \tau^{d-1} \frac{(- \log (\tau))_{+}}{\tau^{d+sp-p}} \dtau \le \int_0^1 \frac{|\log(\tau)|}{\tau^{sp-p+1}} \dtau \lesssim\int_0^1 \tau^{-\kappa-sp+p-1}\, \dtau =\frac{1}{p-(\kappa +sp)},
 \end{align*}
where $\kappa < p-sp$.
Moreover, since $u \in W_{\text{loc}}^{s+\log,p}(U)$,
\begin{align*}
    \int_{U_1}\int_{ U_1} \Kp(|x-y|) |u(x) - u(y)|^p \dxy < \infty.
\end{align*}
Hence, the first term on the RHS of \eqref{test-re-1} is finite. 
For the second term of \eqref{test-re-1}, we note that for every $x\in U_2$ and $y\in \rd \setminus U_1$, $|x-y|\geq \text{dist}(U_2, \rd \setminus U_1)=M>0$, which implies $(-\log|x-y|)_+ \le M_1$ for some $M_1\ge0$, and hence  
\begin{align*}
    &\int_{U_1}\int_{ \rd\setminus U_1} \Kp(|x-y|) |u(x)\phi(x) - u(y)\phi(y)|^p \dxy\\=&\int_{U_2}\int_{ \rd\setminus U_1} \Kp(|x-y|) |u(x)|^p |\phi(x)|^p \dxy \leq M_1 \int_{U_2} |u(x)|^p |\phi(x)|^p \left(\int_{ \rd\setminus U_1}\frac{\dy}{|x-y|^{d+sp}}\right)\dx.
\end{align*}
For every $x \in U_2$ and $y \in \rd \setminus U_1$,  
\begin{align*}
    \frac{1+|y|}{|x-y|} \le 1 + \frac{1+|x|}{M}, \text{ where } \int_{\rd} (1+ |y|)^{-d-sp} \dy < \infty.
\end{align*}
Therefore, 
\begin{align*}
    \int_{U_2} |u(x)|^p|\phi(x)|^p\left(\int_{ \rd\setminus U_1}\frac{\dy}{|x-y|^{d+sp}}\right)\dx \le C\norm{\phi}^p_{L^{\infty}(U_1)} \int_{U_2} |u(x)|^p \left( 1 + \frac{1+|x|}{M}\right)^{d+sp} \dx < \infty. 
\end{align*}
\end{remark}
Next, we observe a fractional logarithmic Poincar\'{e} inequality via scaling. 
\begin{proposition}\label{poincare}
Let $d>sp$, $u \in W^{s+\log,p}(B_r)$ and  $(u)_{B_r} = |B_r|^{-1} \int_{B_r} u(x) \dx$ with $r<1/2$. Then we have
\begin{align*}
  \fint_{B_r}|u-(u)_{B_r}|^{p}\dx &\leq   \left(\fint_{B_r}|u-(u)_{B_r}|^{p^*_s}\dx \right)^{\frac{p}{p^*_s}}\\&\leq C(d,s,p)\frac{r^{sp-d}}{(-\log(2r))_{+}}\iint_{B_r \times B_r}\mathcal{K}^{s+\log,}_+(|x-y|)|u(x)-u(y)|^p\dxy.
\end{align*}
\end{proposition}
\begin{proof}
Without loss of generality, we consider the center of the ball at the origin. We recall the following fractional Poincar\'{e} inequality \cite[Theorem 4.10]{SV13} over the ball $B_1(0):$
 \begin{align*}
     \left(\int_{B_1}|u-(u)_{B_1}|^{p^*_s}\dx\right)^{\frac{p}{p^*_s}} \leq C(d,s,p) \left(\iint_{B_1 \times B_1}\frac{|u(x)-u(y)|^{p}}{|x-y|^{d+sp}}\dxy\right).
 \end{align*}
 First we note that $(v)_{B_1(0)}=(u)_{B_r(0)}$ where $v(x):=u(rx).$ Indeed, let $x\in B_1(0),$ then $z=rx\in B_r(0).$ We can easily see that
 \begin{align*}
     (v)_{B_1}=\frac{1}{|B_1|}\int_{B_1}v(x)\dx=\frac{1}{|B_1|}\int_{B_1}u(rx)\dx=\frac{1}{\omega_{d-1}}\int_{B_r}u(z)\frac{\dz}{r^d}=\frac{1}{|B_r|}\int_{B_r}u(z)\, dz=(u)_{B_r}.
 \end{align*}
 Using this scaling, the left hand side of the above inequality becomes
 \begin{align*}
     \left(\int_{B_1}|v(x)-(v)_{B_1}|^{p^*_s}\dx\right)^{\frac{p}{p^*_s}}=\left(\omega_{d-1}\fint_{B_r}|u(z)-(u)_{B_r}|^{p^*_s}\dz\right)^{\frac{p}{p^*_s}}.
 \end{align*}
 Again, using the scaling, we compute the seminorm. Let $z_1=rx$ and $z_2=ry$. Then
 \begin{align*}
     \iint_{B_1 \times B_1}\frac{|v(x)-v(y)|^p}{|x-y|^{d+sp}}\dxy&=\iint_{B_r \times B_r}\frac{|u(z_1)-u(z_2)|^p}{r^{-d-sp}|z_1-z_2|^{d+sp}}\frac{\dz_1\dz_2}{r^{2d}}\\
     &=r^{sp-d}\iint_{B_r \times B_r}\frac{|u(z_1)-u(z_2)|^p}{|z_1-z_2|^{d+sp}}\dz_1\dz_2.
 \end{align*}
 Hence, we get the following inequality over $B_r(0)$:
 \begin{align*}
  \left(\fint_{B_r}|u-(u)_{B_r}|^{p^*_s}\dx\right)^{\frac{p}{p^*_s}}  \leq C(d,s,p) r^{sp-d}\iint_{B_r \times B_r}\frac{|u(x)-u(y)|^p}{|x-y|^{d+sp}}\dxy.
 \end{align*}
Since $|x-y|< 2r,$ we further have  $(-\log(|x-y|))_+ \geq (-\log(2r))_{+}$. 
 Therefore, 
 \begin{align*}
  \left(\fint_{B_r}|u-(u)_{B_r}|^{p^*_s}\dz\right)^{\frac{p}{p^*_s}} &\leq C(d,s,p) r^{sp-d}\iint_{B_r \times B_r}\frac{|u(x)-u(y)|^p}{|x-y|^{d+sp}}\dxy\\
  &\leq C(d,s,p)\frac{r^{sp-d}}{(-\log(2r))_{+}}\iint_{B_r \times B_r}\frac{(-\log|x-y|)_+|u(x)-u(y)|^p}{|x-y|^{d+sp}} \dxy,
 \end{align*}
 as required. 
\end{proof}

\section{Caccioppoli inequality and local boundedness}{\label{sec:energy lemma and local bd}}
This section contains the Caccioppoli inequality and the local boundedness of weak subsolutions to \eqref{main_PDE}. We begin with the following Caccioppoli estimate. In what follows, for $k>0$, we set $w_{\pm}(x) = (u(x) - k)_{\pm}$.

\begin{lemma}[Caccioppoli inequality]\label{energy estimate}
 Let $s \in (0,s_0)$ and $f \in L_{\text{loc}}^1(\Omega)$. Let $u$ be a weak subsolution to \eqref{main_PDE}, and let $0<r \le \frac{1}{2}$ be such that $B_r=B_r(x_0) \subset \Omega$. Then for any non-negative $\phi \in \C_c^{\infty}(B_r)$ the following estimate holds true 
 \begin{align}\label{cacc}
     & \iint_{B_r \times B_r} \K(|x-y|) |w_{+}(x)\phi(x) - w_{+}(y)\phi(y)|^p \dxy \no \\
     & \le C \int_{B_r} |f(x) |w_+(x) \phi(x)^{p} \dx +  C \iint_{B_r \times B_r} \K(|x-y|) \left(\max \{ w_{+}(x) , w_{+}(y) \}\right)^p |\phi(x) - \phi(y)|^p \dxy \no \\
     & \quad + C \left( \underset{x \in \text{supp}(\phi)}{\esssup}\, \int_{\rd \setminus B_{r}}       \left( \frac{1}{|x-y|^{d+sp}} + \Kp(|x-y|)\right) w_+(y)^{p-1}
     \dy \right) \int_{B_{r}} w_+(x)\phi(x)^{p} \dx \no \\
     & \quad + C \left( \underset{x \in \text{supp}(\phi)}{\esssup}\, \int_{\rd \setminus B_{r}}       \Km(|x-y|) w_-(y)^{p-1}\dy \right) \int_{B_{r}} w_+(x)\phi(x)^{p} \dx \no \\
     & \quad + C\left( \underset{x \in \text{supp}(\phi)}{\esssup}\, \int_{\rd \setminus B_{r}} \Km(|x-y|) 
       \dy \right) \int_{B_{r}} w_+(x)^p\phi(x)^{p} \dx,
 \end{align}
for some $C>0$. If $u$ is a weak supersolution of \eqref{main_PDE}, the estimate in \eqref{cacc} holds with $w_{+}$ replaced by $w_-$ and vice versa.
\end{lemma}

\begin{proof}
   Since $u$ is a subsolution to \eqref{main_PDE}, taking $v:= w_+ \phi^{p} \in W_{0}^{s+\log,p}(B_r)$ (see Remark \ref{C_c}) as a nonnegative test function, we write 
   \begin{align*}
        \iint_{\rd \times \rd} \K(|x-y|) |u(x)-u(y)|^{p-2}(u(x)-u(y)) (v(x)-v(y)) \dxy \leq \int_{B_r} f(x)v(x) \dx,
   \end{align*}
   where, using the symmetry of the kernel, we split 
   \begin{align*}
       &\iint_{\rd \times \rd} \K(|x-y|) |u(x)-u(y)|^{p-2} (u(x)-u(y)) (v(x)-v(y)) \dxy \\
       &=\left(\iint_{B_r \times B_r} + 2\iint_{\rd \setminus B_r \times B_r} \right)  \K(|x-y|) |u(x)-u(y)|^{p-2} (u(x)-u(y)) (v(x)-v(y)) \dxy.
   \end{align*}
    Note that, $\K(|x-y|) \ge 0$ for every $x, y \in B_r$. Now, following the same line of arguments as in \cite[Theorem 1.4]{Palatucci2016}, we get 
   \begin{align}\label{ener-0}
       & \iint_{B_r \times B_r} \K(|x-y|)  |u(x)-u(y)|^{p-2}(u(x)-u(y)) (v(x)-v(y)) \dxy \no \\ 
       & \ge  C\iint_{B_{r} \times B_{r}} \K(|x-y|) |w_+(x)\phi(x) - w_+(y)\phi(y)|^p \dxy \no \\
       &  - C \iint_{B_{r} \times B_{r}} \K(|x-y|) \max\{w_+(x), w_+(y)\}^{p} |\phi(x)-\phi(y)|^{p} \dxy.
   \end{align}
   Now, the kernel $\K(|x-y|)$ changes sign for $x \in B_r$ and $y \in \rd \setminus B_r$. Using \eqref{split-2}, we write 
   \begin{align}\label{ener-1}
       & \iint_{\rd \setminus B_r \times B_r} \K(|x-y|) |u(x)-u(y)|^{p-2}(u(x)-u(y))(v(x)-v(y)) \dxy\no \\
       & = \iint_{\rd \setminus B_r \times B_r} \left(  C_{d,s,p}B_{d,s,p} |x-y|^{-d-sp} + p\Kp(|x-y|) - p\Km(|x-y|) \right) \no \\
       & \quad\quad  |u(x)-u(y)|^{p-2}(u(x)-u(y)) (v(x)-v(y)) \dxy,
   \end{align}
  where using the inequalities given in \cite[Pg. 1285]{Palatucci2016}, it holds 
   \begin{align*}
       &\iint_{\rd \setminus B_r \times B_r} |u(x)-u(y)|^{p-2} (u(x)-u(y)) (v(x)-v(y)) \frac{\dxy}{|x-y|^{d+sp}} \\
       & \ge - \iint_{\rd \setminus B_r \times B_r} |x-y|^{-d-sp} w_+(y)^{p-1} w_+(x)\phi(x)^{p} \dxy, 
   \end{align*}
   and 
   \begin{align*}
       &\iint_{\rd \setminus B_r \times B_r} \Kp(|x-y|)  |u(x)-u(y)|^{p-2}(u(x)-u(y)) (v(x)-v(y)) \dxy \\
       & \ge - \iint_{\rd \setminus B_r \times B_r} \Kp(|x-y|) w_+(y)^{p-1} w_+(x)\phi(x)^{p} \dxy. 
   \end{align*}
   Further, 
   \begin{align*}
      & \iint_{\rd \setminus B_r \times B_r} -\Km(|x-y|)  |u(x)-u(y)|^{p-2}(u(x)-u(y)) (v(x)-v(y)) \dxy \\
      & = \iint_{\rd \setminus B_r \times B_r} -\Km(|x-y|) |u(x)-u(y)|^{p-2} (u(x)-u(y)) w_+(x) \phi(x)^{p} \dxy,
   \end{align*}
where using the fact that support of $w_+$ is $\{ u(x) \ge k\}$ and $-(u(y)-k) \le w_-(y)$, we get
\begin{align*}
    w_+(x)  |u(x)-u(y)|^{p-2}(u(x)-u(y)) &\leq   w_+(x) \left(u(x)-u(y)\right)_+^{p-1}\\&=  w_+(x) \left(u(x)-k+k-u(y)\right)_+^{p-1} \\&\leq C(p)(w_+(x)^p  + w_+(x)(k-u(y))_{+}^{p-1} )\\
    &= C(p)( w_+(x)^p + w_+(x)w_-(y)^{p-1}).
\end{align*}
 Also $\Km(|x-y|) \ge 0$ for every $x, y \in \rd$. Hence, 
\begin{align*}
    &\iint_{\rd \setminus B_r \times B_r} - \Km(|x-y|)|u(x)-u(y)|^{p-2}(u(x)-u(y)) (v(x)-v(y)) \dxy \no \\
    &\ge -C(p)\iint_{\rd \setminus B_r \times B_r}  \left( \Km(|x-y|)w_+(x)^p + \Km(|x-y|)w_+(x)w_-(y) ^{p-1}\right) \phi(x)^{p} \dxy.
\end{align*}
Hence using \eqref{ener-1}, we see that 
\begin{align}\label{ener-3}
    & \iint_{\rd \setminus B_r \times B_r} \K(|x-y|)|u(x)-u(y)|^{p-2}(u(x)-u(y))(v(x)-v(y)) \dxy \no \\
    & \ge - C(d,s,p) \Bigg( \iint_{\rd \setminus B_r \times B_r} \left( \frac{1}{|x-y|^{d+sp}} + \Kp(|x-y|)\right) w_+(y)^{p-1} w_+(x)\phi(x)^{p}  \dxy \no \\
    & \quad -C(p) \iint_{\rd \setminus B_r \times B_r} \Km(|x-y|) w_-(y)^{p-1}  w_+(x)\phi(x)^{p} \dxy \no \\
    & \quad - C(p)\iint_{\rd \setminus B_r \times B_r} \Km(|x-y|)w_+(x)^p \phi(x)^{p} \dxy \Bigg),
\end{align}
where 
\begin{align*}
    & -\iint_{\rd \setminus B_r \times B_r} \Km(|x-y|)w_+(x)^p \phi(x)^{p} \dxy \\
    & \ge - \left( \underset{x \in \text{supp}(\phi)}{\esssup}\, \int_{\rd \setminus B_{r}} \Km(|x-y|) 
       \dy \right) \int_{B_{r}} w_+(x)^p\phi(x)^{p} \dx, 
\end{align*}
and 
\begin{align*}
    & -\iint_{\rd \setminus B_r \times B_r} \Km(|x-y|) w_+(x)w_-(y)^{p-1}  \phi(x)^{p} \dxy \\
    & \ge -\left( \underset{x \in \text{supp}(\phi)}{\esssup}\, \int_{\rd \setminus B_{r}} \Km(|x-y|) w_-(y) ^{p-1} \dy \right) \int_{B_{r}} w_+(x)\phi(x)^{p} \dx. 
\end{align*}
Combining \eqref{ener-0}, \eqref{ener-1}, and \eqref{ener-3}, there exists $C=C(d,s,p)$ such that  
\begin{align}\label{ener-6}
     & \iint_{B_r \times B_r} \K(|x-y|) |w_{+}(x)\phi(x) - w_{+}(y)\phi(y)|^p \dxy \no \\
     & \le C \iint_{B_r \times B_r} \K(|x-y|) \left(\max \{ w_{+}(x) , w_{+}(y) \}\right)^p |\phi(x) - \phi(y)|^p \dxy \no \\
     & \quad + C \iint_{\rd \setminus B_r \times B_r} \left( \frac{1}{|x-y|^{d+sp}} + \Kp(|x-y|)\right) w_+(y)^{p-1}  w_+(x)\phi(x)^{p} \dxy \no \\
     & \quad + C \bigg(\underset{x \in \text{supp}(\phi)}{\esssup}\, \int_{\rd \setminus B_{r}} \Km(|x-y|) w_-(y) ^{p-1} \dy \bigg)\int_{B_{r}} w_+(x)\phi(x)^{p} \dx  \no \\
     & \quad + C\bigg( \underset{x \in \text{supp}(\phi)}{\esssup}\, \int_{\rd \setminus B_{r}} \Km(|x-y|) \dy \bigg) \int_{B_{r}} w_+(x)^p\phi(x)^{p} \dx.
\end{align}
Further, 
\begin{align*}
    & \iint_{\rd \setminus B_r \times B_r}  \left( \frac{1}{|x-y|^{d+sp}} + \Kp(|x-y|)\right) w_+(x)w_+(y)^{p-1}  \phi(x)^{p} \dxy \\
    & \le \left( \underset{x \in \text{supp}(\phi)}{\esssup}\, \int_{\rd \setminus B_{r}} \left( \frac{1}{|x-y|^{d+sp}} + \Kp(|x-y|)\right) w_+(y)^{p-1} 
       \dy \right) \int_{B_{r}} w_+(x)\phi(x)^{p} \dx. 
\end{align*} 
Therefore, from \eqref{ener-6} the required estimate holds. 
\end{proof}

Next, we prove the local boundedness of weak solutions to \eqref{main_PDE}. The following lemma is required for local boundedness. For the proof, see \cite[Lemma 4.1]{Dibe}.

\begin{lemma}[Iteration lemma]\label{iteration}
Let $(Y_j)_{j=0}^{\infty}$ be a sequence of positive real numbers such that
\begin{align*}
    Y_0 \leq \displaystyle c_{0}^{-\frac{1}{\delta}} b^{-\frac{1}{\delta^2}} \text{ and } Y_{j+1}\leq c_0 b^{j} Y_j^{1+\delta},
\end{align*}
$j=0,1,2,\dots$, for some constants $c_0,b>1$ and $\delta>0$. Then $Y_j \ra 0$ as $j \ra \infty$.
\end{lemma}


\begin{proposition}[Local boundedness I]\label{local-boundedness-I}
Let $s \in (0,s_0), d>sp,$ and $f$ be as given in \eqref{weight}. Assume that $u$ is a weak subsolution to \eqref{main_PDE}. Let $0<r\leq \frac{1}{8}$ be such that $B_r(x_0) \subset B_R(x_0)$. Then there exists $C=C(d,s,p)$ such that the following holds:
\begin{align}\label{loc.bounded}
    \underset{B_{\frac{r}{2}}(x_0)}{\esssup}\,u & \leq  \delta\left(1-\log(\frac{r}{2})\right)^{-\frac{1}{p-1}} \left( \mathrm{Tail}_+(u_+;x_0,\frac{r}{2}) + \mathrm{Tail}_-(u_-;x_0,\frac{r}{2}) \right)\no \\
    &\quad +C\delta^{-\frac{(p-1)p_s^*}{p(p_s^*-p\sig)}}\left(\fint_{B_r(x_0)}|u|^{p\sig}\dx\right)^{\frac1{p\sig}} \no \\
    &\quad + \delta^{\frac{\sigma(p-1)}{p\sigma-1}} \left(1-\log(\frac{r}{2})\right)^{-\frac{\sigma}{p\sigma-1}}  r^{\left( sp-d+\frac{d}{\sigma}\right)\frac{\sigma}{p\sigma-1}}\norm{f}_{L^q(B_r(x_0))}^{\frac{1}{p-1}},
\end{align}
where $\delta\in(0,1], \sigma=\frac{(dp-d+sp)q-d}{(dp-d+sp)q-dp}>1$, and $p\sigma<p_s^*$.
\end{proposition}

\begin{proof}
    For $r\in(0,1)$ and $j=0,1,2, \cdots$, define $$r_j=\frac{r}{2}(1+2^{-j}),\,\overline{r}_j=\frac{r_j+r_{j+1}}2,\,B_j=B_{r_j}(x_0),\, \text{ and } \,\overline{B}_j=B_{\overline{r}_j}(x_0).$$ Let $\{\phi_j\}_{j=1}^\infty\subset \C_c^{\infty}(\overline{B}_j)$ be a nonnegative sequence of cut-off functions such that $$0\leq \phi_j\leq 1\text{ in }\overline{B}_j,\,\phi_j\equiv1\text{ on }B_{j+1},\, |\nabla\phi_j|\leq \frac{2^{j+3}}{r}.$$
    For $k,\overline{k}\geq0$, we denote $$k_j=k+(1-2^{-j})\overline{k},\,\overline{k}_j=\frac{k_j+k_{j+1}}{2},\, w_j=(u-k_j)_+, \, \overline{w}_j=(u-\overline{k}_j)_+, \text{ and }  \underline{w}_j=(u-\overline{k}_j)_-. $$
Using the energy estimate (Lemma \ref{energy estimate}) and Remark \ref{support-positive-kernel}, we obtain
\begin{align}\label{caccioppoli-1}
    & \iint_{B_j \times B_j} \K(|x-y|) |\overline{w}_j(x)\phi_j(x) - \overline{w}_j(y)\phi_j(y) |^p \dxy \no \\
     & \le C \int_{B_j} |f(x)| \overline{w}_j(x) \phi_j(x)^{p} \dx \no \\
     & \quad +  C \iint_{B_j \times B_j} \K(|x-y|) \left(\max \{ \overline{w}_j(x) , \overline{w}_j(y) \}\right)^p |\phi_j(x) - \phi_j(y)|^p \dxy \no \\
     & \quad + C \left( \underset{x \in \text{supp}(\phi_j)}{\esssup}\, \int_{\rd \setminus B_j} \left( \frac{1}{|x-y|^{d+sp}} + \Kp(|x-y|)\right) \overline{w}_j(y)^{p-1}\dy \right) \int_{B_j} \overline{w}_j(x) \phi_j(x)^{p} \dx \no \\
     & \quad + C \left( \underset{x \in \text{supp}(\phi_j)}{\esssup}\, \int_{\rd \setminus B_j} \Km(|x-y|) \underline{w}_j(y)^{p-1}\dy \right) \int_{B_j} \overline{w}_j(x) \phi_j(x)^{p} \dx \no \\
     & \quad + C2^{j(d+sp)}(r_j^{-sp}+\widetilde{C})\int_{B_j} \overline{w}_j(x)^p\phi_j(x)^{p} \dx,
\end{align}
where $C=C(d,s,p)$.

We first explain the last term in the inequality \eqref{caccioppoli-1}. Observe that  
\begin{align*}
    \abs{y-x_0} \le 2^{j+4} \abs{y-x}, \text{ and } |y-x|\leq 2|y-x_0|,  
\end{align*}
for every $x\in \operatorname{supp}(\phi)_j$ and $y\in \rd\setminus B_j$. The second inequality yields 
\begin{align}\label{for log}
(-\log |y-x|)_-\leq (-\log 2-\log |y-x_0|)_- \leq \log 2+(-\log |y-x_0|)_-.    
\end{align}
Therefore, using \eqref{for log} and Remark \ref{support-positive-kernel}, 
\begin{align*}
\int_{\rd \setminus B_j} \Km(|x-y|) \dy \leq 2^{j(d+sp)} \int_{\rd \setminus B_j} \frac{\log 2+(-\log |y-x_0|)_-}{|y-x_0|^{d+sp}}\dy \le C2^{j(d+sp)}(r_j^{-sp}+\widetilde{C}).
\end{align*}

Next, we estimate the term involving the weight function. Applying H\"{o}lder's inequality, 
\begin{align*}
    \int_{B_j} |f(x)| \overline{w}_j(x) \phi_j(x)^{p} \dx \le \norm{f\phi_j}_{L^q(B_j)} \left\|\overline{w}_j\phi_j\right\|_{L^{\frac{q}{q-1}}(B_j)} \le \norm{f}_{L^q(B_j)} \left\|\overline{w}_j\phi_j\right\|_{L^{\frac{q}{q-1}}(B_j)}.
\end{align*}
Now we apply the generalized H\"{o}lder's inequality $\frac{1}{p_1} = \frac{\al}{p_2} + \frac{1-\al}{p_3}$ with the conjugate triplet $(p_1,p_2,p_3)$, where 
\begin{align*}
    p_1=\frac{q}{q-1}, p_2= p_s^*, p_3=1, \text{ and } \al=\frac{p_s^*}{(p_s^*-1)q},
\end{align*}
and $\al \in (0,1)$ by observing the fact that 
\begin{align*}
    q>\frac{d}{sp} \Longrightarrow \frac{q}{q-1}<p_s^* \Longrightarrow q > \frac{p_s^*}{p_s^*-1}.
\end{align*}
Hence,
\begin{align*}
    \left\|\overline{w}_j\phi_j \right\|_{L^{\frac{q}{q-1}}(B_j)}  \le \left\| \overline{w}_j \phi_j \right\|_{L^{p_s^*}(B_j)}^{\al} \left\|\overline{w}_j\phi_j\right\|_{L^1(B_j)}^{1-\al}.
\end{align*}
Now, for any $\widetilde{\varepsilon}>0$, applying the Young's inequality with coefficients $(\frac{p}{\al},\frac{p}{p-\al})$, we obtain 
\begin{align*}
\left\|\overline{w}_j\phi_j\right\|_{L^{p_s^*}(B_j)}^\al\,\left\|\overline{w}_j\phi_j\right\|_{L^{1}(B_j)}^{1-\al} \leq \widetilde{\varepsilon} \|\overline{w}_j\phi_j\|_{L^{p_s^*}(B_j)}^p+\widetilde{\varepsilon}^{-\frac{\alpha}{p-\alpha}}\|\overline{w}_j\phi_j\|_{L^{1}(B_j)}^{\frac{p(1-\al)}{p-\al}}.
\end{align*}
Define $\sig:=\frac{p-\al}{p(1-\al)}>1$. Then $p\sig\in(p,p_s^*)$. Using the relation 
\begin{align*}
   \frac{p(p-1)}{p- \alpha} = p - \frac{1}{\sigma}, \; 
    \overline{w}_j\leq(\overline{k}_j-k_j)^{1-p\sig}w_j^{p\sig}, 
\end{align*}
we get
\begin{align*}
     \|\overline{w}_j\phi_j\|_{L^1(B_j)}^{\frac1\sig}\leq\left(\frac{1}{\overline{k}_j-k_j}\right)^{p-\frac1\sig}\left(\int_{B_j}w_j^{p\sig} \dx \right)^{\frac1\sig} \le C\left(\frac{2^{j+2}}{\overline{k}}\right)^{p-\frac1\sig}\left(\int_{B_j}w_j^{p\sig} \dx \right)^{\frac1\sig}.
\end{align*}
Hence, we have the following estimate 
\begin{align}\label{lb-2}
    \int_{B_j} |f(x)| \overline{w}_j \phi_j^{p} \dx \le \norm{f}_{L^q(B_j)} \left( \widetilde{\varepsilon} \|\overline{w}_j\phi_j\|_{L^{p_s^*}(B_j)}^p + \widetilde{\varepsilon}^{-\frac{\alpha}{p-\alpha}} \left(\frac{2^{j+2}}{\overline{k}}\right)^{p-\frac1\sig}\left(\int_{B_j}w_j^{p\sig} \dx \right)^{\frac1\sig} \right),
\end{align}
where 
\begin{align*}
   \left( \int_{B_j} \abs{\overline{w}_j\phi_j}^{p^*_s} \dx \right)^{\frac{p}{p^*_s}} = r_j^{d-sp} \left( \fint_{B_j} \abs{\overline{w}_j\phi_j}^{p^*_s} \dx \right)^{\frac{p}{p^*_s}}, \text{ and } \left(\int_{B_j}w_j^{p\sig} \dx \right)^{\frac1\sig} = r_j^{\frac{d}{\sigma}} \left(\fint_{B_j}w_j^{p\sig} \dx \right)^{\frac1\sig}.
\end{align*}
From Remark \ref{C_c}, $\overline{w}_j\phi_j \in W_0^{s, \log}(B_j)$. Set
\begin{align*}
   W_j:= \overline{w}_j\phi_j - \left| \fint_{B_j} \overline{w}_j\phi_j \dx \right| \in W^{s,\log}(B_j). 
\end{align*}
From the fractional Poincaré type inequality (Proposition \ref{poincare}), we get
\begin{align}\label{caccioppoli-1.5}
    & \left( \fint_{B_j} \abs{W_j}^{p^*_s} \dx \right)^{\frac{p}{p^*_s}} \no \\
    & \le C(d,s,p) \left( \frac{r_j^{sp-d}}{(-\log(2r_j))_+} \iint_{B_j \times B_j}\K(|x-y|) \left|\overline{w}_j(x)\phi_j(x)-\overline{w}_j(y)\phi_j(y)\right|^p \dxy \right).
\end{align}
By the reverse Minkowski inequality, 
\begin{align*}
    \left| \left( \fint_{B_j} \abs{\overline{w}_j\phi_j}^{p^*_s} \dx \right)^{\frac{1}{p^*_s}} - \left|(\overline{w}_j\phi_j)_{B_j}\right|  \right|^p & = \left| \left( \fint_{B_j} \abs{\overline{w}_j\phi_j}^{p^*_s} \dx \right)^{\frac{1}{p^*_s}} - \left( \fint_{B_j} \abs{(\overline{w}_j\phi_j)_{B_j}}^{p^*_s} \dx \right)^{\frac{1}{p^*_s}}  \right|^p \\
    & \le \left( \fint_{B_j} \left| W_j \right|^{p^*_s} \dx \right)^{\frac{p}{p^*_s}}.
\end{align*}
Now we combine the above estimate, \eqref{caccioppoli-1.5}, the energy estimate \eqref{cacc}, and \eqref{lb-2}. Further, we choose $\widetilde{\varepsilon} = \left( 2C(d,s,p) \norm{f}_{L^q(B_j)} \right)^{-1}$. Hence
\begin{align}\label{caccioppoli-2}
    &\left( \fint_{B_j} |\overline{w}_j\phi_j|^{p^*_s}\dx \right)^{\frac{p}{p^*_s}} \no \\
    &\le C \frac{r_j^{sp-d+\frac{d}{\sigma}}}{(-\log(2r_j))_+} \norm{f}_{L^q(B_j)}^{\frac{p}{p-\alpha}} \left(\frac{2^{j+2}}{\overline{k}}\right)^{p-\frac1\sig}\left(\fint_{B_j}w_j^{p\sig} \dx \right)^{\frac1\sig}  \no \\
    &+ C \frac{r_j^{sp}}{(-\log(2r_j))_+} \fint_{B_j} \int_{B_j} \K(|x-y|) \left(\max \{ \overline{w}_j(x) , \overline{w}_j(y) \}\right)^p |\phi_j(x) - \phi_j(y)|^p \dxy \no \\
    &+ C\frac{r_j^{sp}}{(-\log(2r_j))_+} \left( \underset{x \in \text{supp}(\phi_j)}{\esssup}\, \int_{\rd \setminus B_j} \left( \frac{1}{|x-y|^{d+sp}} + \Kp(|x-y|)\right) \overline{w}_j(y)^{p-1}\dy \right) \fint_{B_j} \overline{w}_j \phi_j^{p} \dx \no \\
    & +C\frac{r_j^{sp}}{(-\log(2r_j))_+} \left( \underset{x \in \text{supp}(\phi_j)}{\esssup}\, \int_{\rd \setminus B_j} \Km(|x-y|) \underline{w}_j(y)^{p-1}\dy \right) \fint_{B_j} \overline{w}_j \phi_j^{p} \dx \no \\
     &+ (1+\widetilde{C}(d,s,p)) C 2^{j(d+sp)}\fint_{B_j} \overline{w}_j^p\phi_j^{p} \dx,
\end{align}
where $C=C(d,s,p)$. We now estimate the terms present on the right-hand side. 

\underline{Bound on the R.H.S. of \eqref{caccioppoli-2}:} Using $\abs{\phi(x)-\phi(y)} \le \frac{C}{r_j} |x-y|$ for every $x,y \in B_j$,  we estimate 
\begin{align*}
    &\frac{r_j^{sp}}{(-\log(2r_j))_+} \fint_{B_j} \int_{B_j} \K(|x-y|) \left(\max \{ \overline{w}_j(x) , \overline{w}_j(y) \}\right)^p |\phi_j(x) - \phi_j(y)|^p \dxy\\
   & \le C 2^{jp} \frac{r_j^{sp-p}}{(-\log(2r_j))_+} \fint_{B_j} w_j(y)^p \left( \int_{B_j} \frac{C_{d,s,p}B_{d,s,p}}{|x-y|^{d+sp-p}} + 2 \frac{(-\log(|x-y|))_+}{|x-y|^{d+sp-p}} \dx \right) \dy \\
   & \le C 2^{jp} \left( \frac{1}{p(1-s)} + I \right) \fint_{B_j} w_j(y)^p \dy,
\end{align*}
where $C=C(d,s,p)$. We calculate 
\begin{align*}
    I:= \frac{r_j^{sp-p}}{(-\log(2r_j))_+} \om_{d} \int_{0}^{2r_j} \frac{(-\log(\tau))_+}{\tau^{sp-p+1}} \dtau = \frac{r_j^{sp-p}}{-\log(2r_j)} \om_{d} \int_{0}^{2r_j} \frac{-\log(\tau)}{\tau^{sp-p+1}} \dtau.
\end{align*}
Applying the integration by parts formula,
\begin{align}\label{int-1}
    - \int_{0}^{2r_j} \frac{\log(\tau)}{\tau^{sp-p+1}} \dtau & = -\frac{\tau^{p-sp}\log(\tau)}{p-sp}\Bigg|_0^{2r_j} + \int_{0}^{2r_j} \frac{\tau^{p-sp}}{p-sp} \frac{\dtau}{\tau} \no \\
    & =\left(-\frac{\tau^{p-sp}\log(\tau)}{p-sp} + \frac{\tau^{p-sp}}{(p-sp)^2} \right) \Bigg|_0^{2r_j}\leq\frac{r_j^{p-sp}}{p-sp} \left( -\log(r_j) + \frac{1}{p-sp} \right).
\end{align}
Hence, 
\begin{align*}
    I \leq \frac{\omega_{d-1}}{(sp-p)\log(2r_j)} \left( -\log(r_j) + \frac{1}{p-sp} \right) \le \omega_{d-1} \left( \frac{2}{p(1-s)} + \frac{1}{p^2 \log(2)(1-s)^2} \right).
\end{align*}
The upper bound comes from the fact that $$g(r):= \frac{1}{(sp-p)\log(2r)} \left( -\log(r) + \frac{1}{p-sp} \right)$$ is increasing over $[0, \frac{1}{4}]$ (even any $r<\frac{1}{2} -r_0$, with $r_0>0$ will work). Also, observe that $g(r) \ra +\infty$, as $r \ra \frac{1}{2}$, which dictates us to choose $r \le \frac{1}{4}$.  
Therefore, there exists $C=C(d,s,p)$ such that 
\begin{align*}
    & \frac{r_j^{sp}}{(-\log(2r_j))_+} \fint_{B_j} \int_{B_j} \K(|x-y|) \left(\max \{ \overline{w}_j(x) ,\overline{w}_j(y) \}\right)^p |\phi_j(x) - \phi_j(y)|^p \dxy \\
    & \le C 2^{jp} \fint_{B_j} w_j(y)^p \dy.
\end{align*}
Next, we estimate  
\begin{align*}
    &\frac{r_j^{sp}}{(-\log(2r_j))_+} \left( \underset{x \in \text{supp}(\phi_j)}{\esssup}\, \int_{\rd \setminus B_j} \left( \frac{1}{|x-y|^{d+sp}} + \Kp(|x-y|)\right) \overline{w}_j(y)^{p-1}\dy \right) \\
    & \quad \left( \fint_{B_j} \overline{w}_j(x) \phi_j(x)^{p} \dx \right)\\
    & \le C \frac{2^{j(d+sp)} r_j^{sp}}{(-\log(2r_j))_+} \int_{\rd \setminus B_j} \left( \frac{1}{|x_0-y|^{d+sp}} + \Kp(|x_0-y|)\right) \overline{w}_j(y)^{p-1} \dy \\
    & \quad \left( \fint_{B_j} \frac{w_j(x)^p}{(\overline{k}_j-k_j)^{p-1}} \dx \right) \\
    & \le C \frac{2^{j(d+sp+p-1)}}{(-\log(2r_j))_+} (\overline{k})^{1-p} \Tail_{+}(u_{+}; x_0, \frac{r}{2})^{p-1} \left( \fint_{B_j} w_j(x)^p\dx \right),
\end{align*}
where $C=C(d,s,p)$. In the above estimate, we use $r_j \ge \frac{r}{2}$, $\overline{w}_j \le \frac{w_j^p}{(\overline{k}_j - k_j)^{p-1}}$, and $ \abs{y-x_0} \le 2^{j+4} \abs{y-x}$ for $x\in B_j$ and $y\in \rd\setminus B_j$. Observe that
\begin{align}\label{inq-1}
    \underline{w}_j = (u-k+k-k_j)_-\le (u-k)_- + \overline{k}(1-2^{-j}).
\end{align}
Thus, using $\abs{y-x_0} \le 2^{j+4} \abs{y-x}$ for $|x-y|^{-d-sp}$ and \eqref{for log} for $(-\log(|x-y|))_-$, and then using \eqref{inq-1}, we similarly estimate
\begin{align*}
    &\frac{r_j^{sp}}{(-\log(2r_j))_+} \left( \underset{x \in \text{supp}(\phi_j)}{\esssup}\, \int_{\rd \setminus B_j} \Km(|x-y|) \underline{w}_j(y)^{p-1}\dy \right) \fint_{B_j} \overline{w}_j(x) \phi_j(x)^{p} \dx \\
    &\le \frac{r_j^{sp}}{(-\log(2r_j))_+} \left( \underset{x \in \text{supp}(\phi_j)}{\esssup}\, \int_{\rd \setminus B_j} \left(\Km(|x-y|) + \frac{1}{|x-y|^{d+sp}} \right) \underline{w}_j(y)^{p-1}\dy \right) \\
    & \quad \left( \fint_{B_j} \overline{w}_j(x) \phi_j(x)^{p} \dx \right) \\
    & \le C 2^{j(d+sp)} \frac{r_j^{sp}}{(-\log(2r_j))_+} \int_{\rd \setminus B_j} \left( \Km(|x_0-y|) + \frac{1}{|x_0-y|^{d+sp}}  \right) \underline{w}_j(y)^{p-1}\dy  \\
    & \quad \left( \fint_{B_j} \frac{w_j(x)^p}{(\overline{k}_j-k_j)^{p-1}} \dx \right) \\
    & \le C \frac{2^{j(d+sp+p-1)}}{\overline{k}^{p-1}} ((-\log(2r_j))_+)^{-1} \fint_{B_j} w_j(x)^p\dx \Bigg( \mathrm{Tail}_-((u-k)_-;x_0,\frac{r}{2}) ^{p-1} \\
    &\quad + \overline{k}^{p-1}\int_{\rd \setminus B_{\frac{r}{2}}} \left(\Km(|x_0-y|) + \frac{1}{|x_0-y|^{d+sp}} \right) \dy \Bigg) \\
    & \le C \frac{2^{j(d+sp+p-1)}}{\overline{k}^{p-1}} ((-\log(2r_j))_+)^{-1} \mathrm{Tail}_-((u-k)_-;x_0,\frac{r}{2}) ^{p-1}\fint_{B_j} w_j(x)^p\dx \\
    & \quad + C (1+\widetilde{C}(d,s,p) )2^{j(d+sp+p-1)} \fint_{B_j} w_j(x)^p\dx,
\end{align*}
where $C=C(d,s,p)$. In the last inequality, we have used Remark \ref{support-positive-kernel}. 

Next, we observe that 
$$\fint_{B_j}w_j^p \dx \leq \left(\fint_{B_j}w_j^{p\sig} \dx \right)^{\frac1\sig}.$$
Further, using $r_j \le 2r$, we see that $(-\log(2r_j))_+ \ge (-\log(4r))_+ = -\log(4r)$ as $ r \le \frac{1}{4}$, and moreover, 
\begin{align*}
    \frac{1-\log(\frac{r}{2})}{-\log(4r)} \le C, \; \forall \, r \in [0, \frac{1}{8}]. 
\end{align*}
Therefore, from \eqref{caccioppoli-2} and using the above estimates, for some $C=C(d,s,p)$, we obtain 
\begin{align*}
    &\left( \fint_{B_j} |\overline{w}_j\phi_j|^{p^*_s}\dx \right)^{\frac{p}{p^*_s}} 
    \\ & \le C \left(\fint_{B_j}w_j^{p\sig} \dx \right)^{\frac1\sig} \Bigg( \frac{r^{sp-d+\frac{d}{\sigma}}}{(1-\log(\frac{r}{2}))} \norm{f}_{L^q(B_j)}^{\frac{p}{p-\alpha}} \left(\frac{2^{j+2}}{\overline{k}}\right)^{p-\frac1\sig} + 2^{jp} \\
    &\quad + \frac{2^{j(d+sp+p-1)}}{\overline{k}^{p-1}(1-\log(\frac{r}{2}))} \left[ \mathrm{Tail}_+(u_+;x_0,\frac{r}{2})^{p-1} + \mathrm{Tail}_-((u-k)_-;x_0,\frac{r}{2})^{p-1} \right]
    \\ & \quad + \widetilde{C}(d,s,p) 2^{j(d+sp+p-1)} +2^{j(d+sp)}\Bigg) \\
    &\le C 2^{j(d+sp+p-1)} \left(\fint_{B_j}w_j^{p\sig} \dx \right)^{\frac1\sig} \Bigg( \frac{r^{sp-d+\frac{d}{\sigma}}}{(1-\log(\frac{r}{2}))} \norm{f}_{L^q(B_j)}^{\frac{p}{p-\alpha}} \left(\frac{1}{\overline{k}}\right)^{p-\frac1\sig} \\
    &\quad+\frac{1}{\overline{k}^{p-1}(1-\log(\frac{r}{2}))} \left[\mathrm{Tail}_+(u_+;x_0,\frac{r}{2})^{p-1} + \mathrm{Tail}_-((u-k)_-;x_0,\frac{r}{2})^{p-1}\right] + 1 \Bigg).
\end{align*}
Using the definition of $\sigma$, we observe that $sp-d+\frac{d}{\sigma} \ge 0$. For $\delta\in(0,1]$, we choose $\overline{k}$ such that  
\begin{align*}
    & \frac{1}{\overline{k}^{p-1}(1-\log(\frac{r}{2}))} \left[\mathrm{Tail}_+(u_+;x_0,\frac{r}{2})+ \mathrm{Tail}_-((u-k)_-;x_0,\frac{r}{2})\right] ^{p-1} \le \delta^{1-p} \\
    &\Longleftrightarrow \overline{k} \ge \delta \left(1-\log(\frac{r}{2})\right)^{-\frac{1}{p-1}} \left[ \mathrm{Tail}_+(u_+;x_0,\frac{r}{2}) + \mathrm{Tail}_-((u-k)_-;x_0,\frac{r}{2})\right],
\end{align*}
and 
\begin{align*}
    &\frac{r^{sp-d+\frac{d}{\sigma}}}{(1-\log(\frac{r}{2}))} \norm{f}_{L^q(B_j)}^{\frac{p}{p-\alpha}} \left(\frac{1}{\overline{k}}\right)^{p-\frac1\sig} \le \delta^{1-p} \\
    &\Longleftrightarrow \overline{k} \ge \delta^{\frac{p-\alpha}{p}} r^{\left( sp-d+\frac{d}{\sigma}\right)\frac{p-\alpha}{p(p-1)}}\norm{f}_{L^q(B_j)}^{\frac{1}{p-1}}(1-\log(\frac{r}{2}))^{-\frac{\sigma}{p\sigma-1}}.
\end{align*}
In view of the above inequalities, we take 
\begin{align*}
    \overline{k} & \ge \delta \left(1-\log(\frac{r}{2})\right)^{-\frac{1}{p-1}} \left[ \mathrm{Tail}_+(u_+;x_0,\frac{r}{2}) + \mathrm{Tail}_-((u-k)_-;x_0,\frac{r}{2}) \right] \\
    &\quad + \delta^{\frac{p-\alpha}{p}} \left(1-\log(\frac{r}{2})\right)^{-\frac{\sigma}{p\sigma-1}} r^{\left( sp-d+\frac{d}{\sigma}\right)\frac{p-\alpha}{p(p-1)}}\norm{f}_{L^q(B_j)}^{\frac{1}{p-1}}, 
\end{align*}
so that 
\begin{align}\label{caccioppoli-3}
    \left( \fint_{B_j} |\overline{w}_j\phi_j|^{p^*_s}\dx \right)^{\frac{p}{p^*_s}} \le C 2^{j(d+sp+p-1)} \delta^{1-p} \left(\fint_{B_j}w_j^{p\sig} \dx \right)^{\frac1\sig}.
\end{align}
We also have the following lower bound
\begin{align}\label{caccioppoli-4}
   \left( \fint_{B_j} |\overline{w}_j\phi_j|^{p^*_s}\dx \right)^{\frac{p}{p^*_s}} \ge (k_{j+1}-\overline{k}_j)^{\frac{p(p_s^*-p\sig)}{p_s^*}} \left(\fint_{B_{j+1}}w_{j+1}^{p\sig}\right)^{\frac{p}{p_s^*}}.
\end{align}
Denote $Y_j:=\displaystyle \left(\fint_{B_j}w_j^{p\sig}\right)^{\frac1{p\sig}}.$ Combining \eqref{caccioppoli-3} and \eqref{caccioppoli-4}, there exists $C=C(d,s,p)>1$ such that 
\begin{align*}
    \left(\frac{\overline{k}}{2^{j}}\right)^{\frac{p(p_s^*-p\sig)}{p_s^*}}Y_{j+1}^{\frac{p^2\sig}{p_s^*}}\leq C 2^{j(d+sp+p-1)} \delta^{1-p} Y_j^p,
\end{align*} 
which implies
\begin{align*}
    \frac{Y_{j+1}}{\overline{k}}\leq C \delta^{\frac{(1-p)p^*_s}{p^2\sigma}} \hat{C}^j\left(\frac{Y_j}{\overline{k}}\right)^{1+\beta},
\end{align*}
where using $p\sig <p_s^*$, we see that $\beta:=\frac{p^*_s}{p\sig}-1>0$ and $\hat{C} = 2^{\left( \frac{d+sp+p-1}{p} + \frac{p^*_s-p\sigma}{p^*_s} \right)\frac{p^*_s}{p \sigma}}>1$. Finally, we choose 
\begin{align*}
    \overline{k} & = \delta\left(1-\log(\frac{r}{2})\right)^{-\frac{1}{p-1}} \left[ \mathrm{Tail}_+(u_+;x_0,\frac{r}{2}) + \mathrm{Tail}_-((u-k)_-;x_0,\frac{r}{2}) \right] \\
    &\quad + \delta^{\frac{p-\alpha}{p}} \left(1-\log(\frac{r}{2})\right)^{-\frac{\sigma}{p\sigma-1}} r^{\left( sp-d+\frac{d}{\sigma}\right)\frac{p-\alpha}{p(p-1)}}\norm{f}_{L^q(B_j)}^{\frac{1}{p-1}}+ \delta^{\frac{(1-p)p_s^*}{p(p_s^*-p\sig)}}C^{\frac{1}{\be}}\widetilde{C}^{\frac{1}{\be^2}}Y_0,
\end{align*}
so that 
$$\frac{Y_0}{\overline{k}}\leq \delta^{\frac{(p-1)p_s^*}{p^2\sig\be}}C^{-\frac{1}{\be}} \hat{C}^{-\frac{1}{\be^2}} = \left( C \delta^{\frac{(1-p)p_s^*}{p^2\sig}} \right)^{-\frac{1}{\beta}} \hat{C}^{-\frac{1}{\be^2}}. $$
Now we take $c_0 = C\delta^{\frac{(1-p)p^*_s}{p^2\sig}}$  and $b=\hat{C}$ in Lemma \ref{iteration}.
Therefore, applying the iteration lemma (Lemma \ref{iteration}), we get $Y_j\to0\text{ as }j\to\infty.$ Hence 
\begin{align*}
    &\sup_{B_{\frac{r}{2}}(x_0)}(u-k)^+\leq \overline{k} \\
    &=\delta\left(1-\log(\frac{r}{2})\right)^{-\frac{1}{p-1}} \left[ \mathrm{Tail}_+(u_+;x_0,\frac{r}{2}) + \mathrm{Tail}_-((u-k)_-;x_0,\frac{r}{2}) \right]\\
    &\quad + \delta^{\frac{p-\alpha}{p}} \left(1-\log(\frac{r}{2})\right)^{-\frac{\sigma}{p\sigma-1}}  r^{\left( sp-d+\frac{d}{\sigma}\right)\frac{p-\alpha}{p(p-1)}}\norm{f}_{L^q(B_j)}^{\frac{1}{p-1}}+ \delta^{\frac{(1-p)p_s^*}{p(p_s^*-p\sig)}}C^{\frac{1}{\be}}\widetilde{C}^{\frac{1}{\be^2}}Y_0,
\end{align*}
Observe that $C^{\frac{1}{\be}}\hat{C}^{\frac{1}{\be^2}}=C(d,s,p)$. Now, choosing $k=0$ yields \eqref{loc.bounded}.
\end{proof}

\section{Logarithmic estimate and growth lemma}\label{sec:log lemma and growth}

In the following lemma, we obtain a logarithmic energy estimate for weak supersolution to \eqref{main_PDE}.

\begin{lemma}[Logarithmic estimate]\label{log-estimate} 
 Let $s \in (0,s_0), d>sp,$ and $f$ be as given in \eqref{weight}. We consider $B_R(x_0)\subset \Omega$ with $\text{diam}(B_R(x_0))\leq 1$. Assume that $u$ is a weak supersolution to \eqref{main_PDE} satisfying $u \geq 0$ in $B_R(x_0)$. Let $r>0$ be such that $B_r (x_0)\subset B_{\frac{R}{2}}(x_0)$ and $t>0$. Then there exists $C=C(d,p, s)$ such that the following holds:
\begin{align}\label{log-main}
&\iint_{B_r(x_0) \times B_r(x_0)} \K(|x-y|) \left|\log \left(\frac{u(x)+t}{u(y)+t}\right)\right|^p \dxy \no \\
&\leq C r^{d-sp}\Bigg(t^{1-p} \left(\frac{r}{R}\right)^{sp} \left[\operatorname{Tail}_+\left(u_{-} ; x_0, R\right)^{p-1}+\operatorname{Tail}_-\left(u_{+} ; x_0, R\right)^{p-1}\right] + t^{1-p}r^{sp-\frac{d}{q}}\norm{f}_{L^q(B_{R}(x_0))} \no \\
& \quad + 1 - \log(r)\Bigg).
\end{align}
\end{lemma}
\begin{proof}
Consider a cut-off function $\phi \in \cc(B_{3r/2})$ satisfying $\phi =
1$ in $B_r, 0 \leq \phi \leq 1$ and $|\nabla \phi| \leq \frac{C}{r}$ in $B_{3r/2}$. Using $u\ge 0$ in $B_{3r/2}$ and $u \in W_{\text{loc}}^{s+\log}(B_{3r/2})$, we get $(u+t)^{1-p} \in W_{\text{loc}}^{s+\log}(B_{3r/2})$. From Remark \ref{C_c}, choosing $v=(u+t)^{1-p} \phi^p \in W_0^{s+\log}(B_{3r/2})$ as a nonnegative test function, we obtain 
\begin{align*}
    0 & \le \iint_{\rd \times \rd} \K(|x-y|) |u(x)-u(y)|^{p-2} (u(x)-u(y))(v(x)-v(y)) \dxy -\int_{B_{\frac{3r}{2}}} f(x)v(x) \dx \\
      & \le \iint_{\rd \times \rd} \K(|x-y|) |u(x)-u(y)|^{p-2} (u(x)-u(y))(v(x)-v(y)) \dxy + \int_{B_{\frac{3r}{2}}} \abs{f(x)} v(x) \dx. 
\end{align*}
By observing $\K(|x-y|) \ge 0$ for $x,y \in B_{2r}$, we use the same arguments as in the proof of \cite[Logarithmic Lemma 1.3]{Palatucci2016}, to get 
\begin{align}\label{log-0}
    &\iint_{B_{2r} \times B_{2r}} \K(|x-y|) |u(x)-u(y)|^{p-2} (u(x)-u(y))(v(x)-v(y)) \dxy \no \\
    & \le -\frac{1}{C} \iint_{B_{2r} \times B_{2r}} \K(|x-y|) \left|\log \left(\frac{u(x)+t}{u(y)+t}\right)\right|^p \phi(y)^p \dxy \no \\
    &\quad + C \iint_{B_{2r} \times B_{2r}} \K(|x-y|) |\phi(x)-\phi(y)|^p \dxy.
\end{align}
Now, using $\abs{\phi(x) - \phi(y)} \le C\frac{\abs{x-y}}{r}$ for $x,y \in B_{2r}$, and using \eqref{int-1} we obtain
\begin{align*}
    &\iint_{B_{2r} \times B_{2r}} \K(|x-y|) |\phi(x)-\phi(y)|^p \dxy \\
    &\le C(d,s,p) |B_{2r}| r^{-sp} \left( \frac{1}{p(1-s)} (1- \log(r)) + \frac{1}{p^2(1-s)^2} \right).
\end{align*}
Further, using $\abs{z}^{p-2}z_+ \le z_+^{p-1}$ for $z \in \rd$, 
\begin{align*}
    & \iint_{B_{R} \setminus B_{2r} \times B_{2r}} \K(|x-y|) |u(y)-u(x)|^{p-2} (u(y)-u(x)) v(y)\dxy \\
    & \le \iint_{B_{R} \setminus B_{2r} \times B_{2r}} \K(|x-y|) (u(y)-u(x))^{p-1}_+ \frac{\phi(y)^p}{(u(y)+t)^{p-1}} \dxy \\
    & \le \iint_{B_{R} \setminus B_{2r} \times B_{\frac{3r}{2}}} \K(|x-y|) \phi(y)^p \dxy,
\end{align*}
where we use the fact that $\K(|x-y|) \ge 0$ for $x,y \in B_{R}$ and $u \ge 0$ in $B_R$. Further, for $x \in B_R \setminus B_{2r}$ and $y \in B_{\frac{3r}{2}}$, using 
\begin{align}\label{bound-1}
    \frac{|x-x_0|}{|x-y|} \le 1 + \frac{|y-x_0|}{|y-x|} \le 1 + \frac{\frac{3r}{2}}{\frac{r}{2}} = 4,  
\end{align}
and the monotonicity property of $\log$ function, we estimate 
\begin{align}\label{log-1}
    &\iint_{B_{R} \setminus B_{2r} \times B_{\frac{3r}{2}}} \K(|x-y|) \phi(y)^p \dxy \no \\
    &\le C(d,s) |B_{\frac{3r}{2}}| \sup_{y \in B_{\frac{3r}{2}}} \int_{B_{R} \setminus B_{2r}} \frac{1 + (-\log(|x-y|))_+)}{|x-y|^{d+sp}}\dx \no \\
    & \le C(d,s,p) r^d \sup_{y \in B_{\frac{3r}{2}}} \int_{B_{R} \setminus B_{2r}} \frac{1 + (-\log(|x-x_0|))_+)}{|x-x_0|^{d+sp}}\dx \no \\
    &\le C(d,s,p) r^d (1-\log(r)) \sup_{y \in B_{\frac{3r}{2}}} \int_{\rd \setminus B_{2r}} \frac{\dx}{|x-x_0|^{d+sp}} = C(d,s,p) r^{d-sp} (1-\log(r)).
\end{align}
Next, we estimate $\K(|x-y|) (u(x)-u(y))(v(x)-v(y))$ over $\rd \setminus B_{R} \times B_{2r}$, where $\K(|x-y|)$ changes sign. 
Using $\abs{z}^{p-2}z_{\pm} \le z_{\pm}^{p-1}$ for $z \in \rd$, we write 
\begin{align*}
       & \iint_{\rd \setminus B_R \times B_{2r}} \K(|x-y|)|u(x)-u(y)|^{p-2}(u(x)-u(y))(v(x)-v(y)) \dxy\no \\
       & = \iint_{\rd \setminus B_R \times B_{2r}} \left(\frac{C_{d,s,p}B_{d,s,p}}{|x-y|^{d+sp}} + p\Kp(|x-y|) - p\Km(|x-y|) \right)\no\\
       & \quad \quad |u(x)-u(y)|^{p-2} (u(y)-u(x)) v(y) \dxy \no \\
       & \le C(d,s,p) \iint_{\rd \setminus B_R \times B_{2r}} \left(\frac{1}{|x-y|^{d+sp}} + \Kp(|x-y|) \right) (u(y)-u(x))^{p-1}_+ \frac{\phi(y)^p}{(u(y)+t)^{p-1}} \dxy \no \\
       & \quad +C(d,s,p)\iint_{\rd \setminus B_R \times B_{2r}} \left(\Km(|x-y|)\right) (u(y)-u(x))_-^{p-1} \frac{\phi(y)^p} {(u(y)+t)^{p-1}} \dxy := I_1 + I_2. 
   \end{align*} 
Further, using $(u(y)-u(x))_+ \le (u(y))_+ + (-u(x))_+= u(y) + (u(x))_-$,
\begin{align*}
    &I_1\le C(d,s,p) \iint_{\rd \setminus B_R \times B_{2r}} \left(\frac{1}{|x-y|^{d+sp}} + \Kp(|x-y|) \right) \phi(y)^p \dxy \\
    &\quad+ C(d,s,p) t^{1-p} \iint_{\rd \setminus B_R \times B_{2r}} \left(\frac{1}{|x-y|^{d+sp}} + \Kp(|x-y|) \right) (u(x)_-)^{p-1} \phi(y)^p\dxy := I_{1,1}+I_{1,2},
\end{align*}
where from the similar estimates as in \eqref{log-1},  
\begin{align*}
    I_{1,1} \le C(d,s,p) \iint_{\rd \setminus B_{r} \times B_{2r}} \frac{1 + (-\log(|x-y|))_+)}{|x-y|^{d+sp}} \phi(y)^p \dxy \le  C(d,s,p) r^{d-sp} (1-\log(r)).
\end{align*}
Moreover, using a similar estimate as in \eqref{bound-1}, 
\begin{align*}
    I_{1,2} & \le C(d,s,p) t^{1-p} r^{d} \int_{\rd \setminus B_R} \left(\frac{1}{|x-x_0|^{d+sp}}+\Kp(|x-x_0|) \right) (u(x)_-)^{p-1} \dx \\
    & \le C(d,s,p) t^{1-p} \frac{r^{d}}{R^{sp}} \mathrm{Tail}_{+}(u_-;x_0,R)^{p-1}.
\end{align*}
Similarly, using the fact that $(u(y)-u(x))_- \le (u(y))_- + (-u(x))_- = (u(x))_+$, 
\begin{align*}
    I_2 & \le C(d,s,p) t^{1-p}\iint_{\rd \setminus B_R \times B_{2r}} \frac{(-\log(|x-y|))_-}{|x-y|^{d+sp}} (u(x)_+)^{p-1} \dxy.
\end{align*}
Further for $x \in \rd \setminus B_R $ and $y \in B_{\frac{3r}{2}}(x_0)$,
\begin{align*}
    \frac{|x-y|}{|x-x_0|} \le 1 + \frac{|y-x_0|}{|x-x_0|} \le 1 + \frac{\frac{3r}{2}}{2r} = \frac{7}{4},  
\end{align*}
which implies 
\begin{align}\label{bound-2}
    (-\log(|x-y|))_- \le \log \frac{7}{4} + (-\log(|x-x_0|))_-.
\end{align}
Hence, using \eqref{bound-1} for $|x-y|^{-d-2s}$ and \eqref{bound-2} for $(-\log(|x-y|))_-$, we get 
\begin{align*}
    I_2 & \le  C(d,s,p) t^{1-p}\iint_{\rd \setminus B_R \times B_{2r}} \left( \frac{1}{|x-y|^{d+sp}}+\Km(|x-y|) \right) (u(x)_+)^{p-1} \dxy \\
    &\le C(d,s,p) t^{1-p} r^{d} \int_{\rd \setminus B_R} \left( \frac{1}{|x-x_0|^{d+sp}}+\Km(|x-x_0|) \right) (u(x)_+)^{p-1} \dx \\
    & \le C(d,s,p) t^{1-p} \frac{r^{d}}{R^{sp}} \mathrm{Tail}_{-}(u_+;x_0,R)^{p-1}.
\end{align*}
Therefore,
\begin{align}\label{log-4}
    & \iint_{\rd \setminus B_R \times B_{2r}} \K(|x-y|)|u(x)-u(y)|^{p-2}(u(x)-u(y))(v(x)-v(y)) \dxy \no \\
    & \le C(d,s,p) \left( r^{d-sp} (1-\log(r)) +  t^{1-p} \frac{r^{d}}{R^{sp}} \left(\mathrm{Tail}_{+}(u_-;x_0,R)^{p-1} + \mathrm{Tail}_{-}(u_+;x_0,R)^{p-1} \right) \right). 
\end{align}
Finally,  we estimate 
\begin{align*}
    \int_{B_{\frac{3 r}{2}}} |f(x)| \frac{\phi(x)^p}{(u(x)+t)^{p-1}} \dx& \leq t^{1-p}\int_{B_{\frac{3 r}{2}}}|f(x)| \phi(x)^p \dx  \leq t^{1-p}\left\|f\right\|_{L^{q}(B_{R})} |B_{\frac{3r}{2}}|^{\frac{sp}{d}-\frac1q} \left\| \phi \right\|_{L^{p_s^*}(B_{\frac{3r}{2}})}^p.
\end{align*}
Now we use the fractional Poincar\'e inequality, and the fact that $\abs{\phi(x)- \phi(y)}\le \frac{C}{r}|x-y|$ for every $x,y\in B_{2r}$,  
\begin{align*}
    \left( \int_{B_{2r}} \abs{\phi(x)}^{p^*_s}\dx \right)^{\frac{p}{p^*_s}} \le C \iint_{B_{2r} \times B_{2r}} \frac{|\phi(x)- \phi(y)|^p}{\abs{x-y}^{d+sp}} \dxy  \le Cr^{d-p} \frac{r^{p-sp}}{p-sp} \le C r^{d-sp} ,
\end{align*}
for some $C=C(d,s,p)$.
Hence, 
\begin{align*}
    \int_{B_{\frac{3 r}{2}}} |f(x)| \frac{\phi(x)^p}{(u(x)+t)^{p-1}} \dx \le C(d,s,p)  t^{1-p} r^{d-sp}  |B_{\frac{3r}{2}}|^{\frac{sp}{d}-\frac1q} \left\|f\right\|_{L^{q}(B_{\frac{3r}{2}})}.
\end{align*}
Combining \eqref{log-0}, \eqref{log-1}, \eqref{log-4}, and the above estimate, we get \eqref{log-main}.
\end{proof}

The following proposition proves a growth lemma for a weak supersolution to \eqref{main_PDE}.

\begin{proposition}[Growth lemma]\label{growth-lemma}
Let $s \in (0,s_0), d>sp,$ and $f$ be as given in \eqref{weight}. We consider $B_R(x_0)\subset \Omega$ with $\text{diam}(B_R(x_0))\leq 1$. Assume that $u$ is a weak supersolution to \eqref{main_PDE} satisfying $u \geq 0$ in $B_R(x_0) \subset \Omega$. Let $0<r\leq \frac{1}{24}$ be such that $B_r(x_0) \subset B_{\frac R{16}}(x_0)$. Assume $k \geq 0$ and there exists $\tau \in(0,1]$ such that
\begin{align}{\label{expan1}}
    \left|B_r(x_0) \cap\{u \geq k\}\right| \geq \tau\left|B_r(x_0\right)|.
\end{align} 
Then 
\begin{align}{\label{expan1.1}}
 \underset{B_{4 r}\left(x_0\right)}{{\essinf}}\, u & \geq \delta k-(1-\log r)^{-\frac{1}{p-1}}\Bigg(\left(\frac{r}{R}\right)^{\frac{sp}{p-1}} \left[\operatorname{Tail}_+\left(u_{-} ; x_0, R\right)+\operatorname{Tail}_-\left(u_{+} ; x_0, R\right)\right] \no \\
& \quad + r^{\frac{qsp-d}{q(p-1)}} \norm{f}_{L^q(B_R(x_0))}^{\frac{1}{p-1}} \Bigg),
\end{align}
where $\delta=\delta(d, s,p, \tau) \in \left(0, \frac{1}{4}\right)$.
\end{proposition}
\begin{proof}
Our proof includes two steps.\\
\textbf{Step 1.} Let $\var>0$. Under the assumption in \eqref{expan1}, we claim that there exists  ${C_1}={C_1}(d, p, s)$ such that
\begin{align}\label{expan2}
& \Bigg|B_{6 r} \cap \Bigg\{ u \leq 2 \delta k - (1-\log(r))^{-\frac{1}{p-1}} \Bigg(\frac{1}{2} \left(\frac{r}{R}\right)^{\frac{sp}{p-1}} \left[\operatorname{Tail}_+\left(u_{-} ; x_0, R\right)+\operatorname{Tail}_-\left(u_{+} ; x_0, R\right)\right] \no \\ 
& \quad +\frac{1}{2} r^{\frac{qsp-d}{q(p-1)}} \norm{f}_{L^q(B_R)}^{\frac{1}{p-1}}\Bigg)-\frac{\var}{2}  \Bigg\}\Bigg| \leq \frac{{C_1}}{\tau \log \frac{1}{2 \delta}}\left|B_{6r}\right|,
\end{align}
for every $\delta \in\left(0, \frac{1}{4}\right)$.
Consider a  cut-off function $\phi \in \C_c^{\infty}(B_{7 r})$ such that $0 \leq \phi \leq 1$ in $B_{7r}, \phi=1$ in $B_{6 r}$ and $|\nabla \phi| \leq \frac{8}{r}$ in $B_{7 r}$. We define $t_{\var}$ as follows:
\begin{align*}
t_\var & :=\frac{1}{2}(1-\log(r))^{-\frac{1}{p-1}}\Bigg(\left(\frac{r}{R}\right)^{\frac{sp}{p-1}} \left[\operatorname{Tail}_+\left(u_{-} ; x_0, R\right)+\operatorname{Tail}_-\left(u_{+} ; x_0, R\right)\right]\\
& \quad + r^{\frac{qsp-d}{q(p-1)}}\norm{f}_{L^q(B_R)}^{\frac{1}{p-1}}\Bigg)+\frac{\var}{2}. 
\end{align*}
The choice of $t_{\var}$ implies
\begin{align*}
   &t_{\varepsilon}^{1-p} \left( \left(\frac{r}{R}\right)^{sp} \left[\operatorname{Tail}_+\left(u_{-} ; x_0, R\right)^{p-1}+\operatorname{Tail}_-\left(u_{+} ; x_0, R\right)^{p-1}\right] + r^{\frac{qsp-d}{q}}\norm{f}_{L^q(B_{\frac{3r}{2}})} \right) \\
   &\quad \le C(1-\log(r)). 
\end{align*}
 Next, we define $v:=u+t_\var$. Since $u$ is a weak supersolution, we choose $v^{1-p} \phi^p$ as a test function and using Lemma \ref{log-estimate} to obtain
\begin{align}\label{expan3}
&\iint_{B_{6r} \times B_{6r}} \K(|x-y|) \left|\log \left(\frac{u(x)+t_\var}{u(y)+t_\var}\right)\right|^p \dxy \no \\
&\leq C r^{d-sp}\Bigg(t_{\varepsilon}^{1-p} \left(\frac{r}{R}\right)^{sp} \left[\operatorname{Tail}_+\left(u_{-} ; x_0, R\right)^{p-1}+\operatorname{Tail}_-\left(u_{+} ; x_0, R\right)^{p-1}\right] + t_{\varepsilon}^{1-p}r^{\frac{qsp-d}{q}}\norm{f}_{L^q(B_{\frac{3r}{2}})} \no \\
& \quad + 1 - \log(r)\Bigg)\no\\ & \leq C r^{d-sp}(1 - \log(r)),
\end{align}
where $C=C(d,s,p)$. For $\delta \in\left(0, \frac{1}{4}\right)$, we denote
$$
w=\left(\min \left\{\log \frac{1}{2 \delta}, \log \frac{k+t_\var}{v}\right\}\right)_{+} .
$$
Note that the kernel is positive inside $B_{6r}\times B_{6r}$. Since $w$ is a truncation of $\log(k+t_\var)-\log (v)$, the energy decreases and by \eqref{expan3}, we have
\begin{align}{\label{expan4}}
&\iint_{B_{6r} \times B_{6r}} \K(|x-y|) |w(x)-w(y)|^p \dxy \no\\ & \leq\iint_{B_{6r} \times B_{6r}} \K(|x-y|) \left|\log \left(\frac{u(x)+t_\var}{u(y)+t_\var}\right)\right|^p \dxy \leq C r^{d-sp}(1 - \log(r)).
\end{align}
Note that $0< |x-y| \leq 12r\leq \frac{1}{2}$ implies $\K(|x-y|)\geq\Kp(|x-y|)$ in $B_{6r}\times B_{6r}$. From \eqref{expan4}, H\"older's inequality and Poincaré inequality (see Proposition \ref{poincare}), we obtain
\begin{align}{\label{expan5}}
&\int_{B_{6r}}\left|w-(w)_{B_{6r}}\right| \dx \leq Cr^{\frac {d}{p^\prime}}\left(\int_{B_{6r}}\left|w-(w)_{B_{6r}}\right|^p \dx\right)^{\frac{1}{p}} \no\\&\leq C\frac{r^{s+\frac{d}{p^{\prime}}}}{(-\log(12r))^{\frac{1}{p}}}\left(\iint_{B_{6r} \times B_{6r}} \Kp(|x-y|) |w(x)-w(y)|^p \dxy \right)^{\frac{1}{p}}\no\\&\leq C\frac{r^{s+\frac{d}{p^{\prime}}}}{(-\log(12r))^{\frac{1}{p}}}\left(\iint_{B_{6r} \times B_{6r}} \K(|x-y|) |w(x)-w(y)|^p \dxy \right)^{\frac{1}{p}} \no\\&\leq C r^{d} \left(\frac{1 - \log(r)}{(-\log(12r))}\right)^{\frac{1}{p}} \leq C|B_{6r}|.
\end{align}
The last inequality follows from the boundedness of $g(r):= (1 - \log(r))(-\log(12r))^{-1}$ over $(0,\frac{1}{24}]$. Indeed, $g(r) \ra 1$ as $ r \ra 0$ and $g(r) \ra \infty$ as $r \ra \frac{1}{12}$. The blow up limit near $\frac{1}{12}$ dictates us to choose $r \le \frac{1}{24}$ (even any $r<\frac{1}{12} -r_0$, with $r_0>0$ will work).

Observe that $\{w=0\}=\left\{v \geq k+t_\var\right\}=\{u \geq k\}$. By the assumption \eqref{expan1}, it follows that
\begin{align}{\label{expan6}}
    \left|B_{6 r} \cap\{w=0\}\right| \geq \frac{\tau}{6^d}\left|B_{6r}\right|.
\end{align}
Following the proof of \cite[Lemma 3.1]{KuusiHarnack} and using \eqref{expan6}, we obtain
\begin{align}{\label{expan7}}
\log \frac{1}{2 \delta} & =\frac{1}{\left|B_{6 r} \cap\{w=0\}\right|} \int_{B_{6 r} \cap\{w=0\}}\left(\log \frac{1}{2 \delta}-w(x)\right) \dx  \leq \frac{6^d}{\tau}\left(\log \frac{1}{2 \delta}-(w)_{B_{6 r}}\right).
\end{align}
Now integrating \eqref{expan7} over the set $B_{6 r} \cap\left\{w=\log \frac{1}{2 \delta}\right\}$ and using \eqref{expan5}, there exists ${C_1}={C_1}(d, p, s)$ such that
$$
\left|\left\{w=\log \frac{1}{2 \delta}\right\} \cap B_{6 r}\right| \log \frac{1}{2 \delta} \leq \frac{6^d}{\tau} \int_{B_{6 r}}\left|w-(w)_{B_{6 r}}\right| \dx \leq \frac{{C_1}}{\tau}\left|B_{6 r}\right| .
$$
Hence, for any $\delta \in\left(0, \frac{1}{4}\right)$, we have
$$
\left|B_{6 r} \cap\left\{v \leq 2 \delta\left(k+t_\var\right)\right\}\right| \leq \frac{{C_1}}{\tau} \frac{1}{\log \frac{1}{2 \delta}}\left|B_{6r}\right| .
$$
This implies
$$
\left|B_{6 r} \cap\left\{u \leq 2 \delta k - {t_\var}\right\}\right| \leq \frac{{C_1}}{\tau} \frac{1}{\log \frac{1}{2 \delta}}\left|B_{6r}\right| .
$$
Thus, \eqref{expan2} holds.
\smallskip\\\textbf{Step 2.} We claim that, for every $\var>0$, there exists a constant $\delta=\delta(d, p, s,  \tau) \in\left(0, \frac{1}{4}\right)$ such that
\begin{align}{\label{expan8}}
& \underset{B_{4 r}}{\essinf}\, u \geq \delta k-(1-\log r)^{-\frac{1}{p-1}}\Bigg(\left(\frac{r}{R}\right)^{\frac{sp}{p-1}} \left[\operatorname{Tail}_+\left(u_{-} ; x_0, R\right)+\operatorname{Tail}_-\left(u_{+} ; x_0, R\right)\right] \no \\
& \quad + r^{\frac{qsp-d}{q(p-1)}} \norm{f}_{L^q(B_R)}^{\frac{1}{p-1}}\Bigg) - \var.
\end{align}
As a consequence of \eqref{expan8}, the estimate \eqref{expan1.1} follows.
To prove \eqref{expan8}, without loss of generality, we may assume that
\begin{align}{\label{expan9}}
&\delta k \geq(1-\log r)^{-\frac{1}{p-1}}\Bigg(\left(\frac{r}{R}\right)^{\frac{sp}{p-1}} \left[\operatorname{Tail}_+\left(u_{-} ; x_0, R\right)+\operatorname{Tail}_-\left(u_{+} ; x_0, R\right)\right] \no \\
& \quad + r^{\frac{qsp-d}{q(p-1)}} \norm{f}_{L^q(B_R)}^{\frac{1}{p-1}}\Bigg) + \var.
\end{align}
Otherwise \eqref{expan8} holds true, since $u \geq 0$ in $B_R$.
Let $\rho \in[r, 6 r]$ and $\phi \in \C_c^{\infty}(B_\rho(x_0))$ be a cut-off function such that $0 \leq \phi \leq 1$ in $B_\rho(x_0)$. For any $l \in(\delta k, 2 \delta k)$, we denote $w_{\pm}=(l-u)_{\pm}$. We estimate 
\begin{align}\label{f}
    \int_{B_{\rho}} |f(x)| w_+(x) \phi(x)^{p} \dx = \int_{B_{\rho} \cap \{u<l\}} |f(x)| w_+(x) \phi(x)^{p} \dx & \le l \int_{B_{\rho} \cap \{u<l\}} |f(x)| \dx \no \\
    &\le l \norm{f}_{L^q(B_{\rho})} |B_{\rho} \cap \{u<l\}|^{\frac{q-1}{q}}. 
\end{align}
To start the iteration process, for $j=0,1,2, \ldots$, we define
\begin{align}\label{notation-2}
l=k_j:=\delta k+2^{-j-1} \delta k, \; \rho=\rho_j:=4 r+2^{1-j} r, \text{ and } \hat{\rho}_j:=\frac{\rho_j+\rho_{j+1}}{2}.
\end{align}
Then $l \in(\delta k, 2 \delta k), \rho_j, \hat{\rho_j} \in(4 r, 6 r)$ and
\begin{align}{\label{obser}}
k_j-k_{j+1}=2^{-j-2} \delta k \geq 2^{-j-3} k_j.
\end{align}
Set $B_j=B_{\rho_j}\left(x_0\right), \hat{B}_j=B_{\hat{\rho}_j}\left(x_0\right)$. Denote $\overline{w}_j = \left(k_j-u\right)_{+}$ and $\underline{w}_j = \left(k_j-u\right)_{-}$.
We observe that
\begin{align*}
\overline{w}_j \geq 2^{-j-3} k_j \chi_{\left\{u<k_{j+1}\right\}}.
\end{align*}
We consider a sequence of cut-off functions $\left(\phi_j\right)_{j=0}^{\infty} \subset \C_c^{\infty}(\hat{B}_j)$ such that $0 \leq \phi_j \leq 1$ in $\hat{B}_j, \phi_j=1$ in $B_{j+1}$ and $\left|\nabla \phi_j\right| \leq \frac{2^{j+3}}{r}$. From the Caccioppoli inequality (Lemma \ref{energy estimate}) and \eqref{f}, for $C=C(d,p,s)$, we have
 \begin{align}\label{expan101}
     & \frac{\rho_j^{sp-d}}{-\log(2 \rho_j)} \iint_{B_j \times B_j} \K(|x-y|) |\overline{w}_j(x)\phi_j(x) - \overline{w}_j(y)\phi_j(y)|^p \dxy \no \\
     & \le C  k_j \frac{\rho_j^{sp-d}}{-\log(2 \rho_j)} \norm{f}_{L^q(B_j)} \left|B_j \cap \{u<k_j\}\right|^{\frac{q-1}{q}} \no \\
     &\quad +  C \frac{\rho_j^{sp-d}}{-\log(2 \rho_j)} \iint_{B_j \times B_j} \K(|x-y|) \left(\max \{ \overline{w}_j(x) , \overline{w}_j(y) \}\right)^p |\phi_j(x) - \phi_j(y)|^p \dxy \no \\
     & \quad + C \frac{\rho_j^{sp}}{-\log(2 \rho_j)} \left( \underset{x \in \text{supp}(\phi_j)}{\esssup}\, \int_{\rd \setminus B_j}   
     \left( \frac{1}{|x-y|^{d+sp}} + \Kp(|x-y|)\right) \overline{w}_j(y)^{p-1}
     \dy \right) \no \\
     &\quad \fint_{B_j} \overline{w}_j(x)\phi_j(x)^{p} \dx \no \\
     & \quad + C \frac{\rho_j^{sp}}{-\log(2 \rho_j)} \left( \underset{x \in \text{supp}(\phi_j)}{\esssup}\, \int_{\rd \setminus B_j}  \Km(|x-y|) \underline{w}_j(y)^{p-1}\dy \right) \fint_{B_{j}} \overline{w}_j(x)\phi_j(x)^{p} \dx \no \\
     & \quad + C \rho_j^{sp}\left( \underset{x \in \text{supp}(\phi_j)}{\esssup}\, \int_{\rd \setminus B_j} \Km(|x-y|) \dy \right) \fint_{B_j} \overline{w}_j(x)^p\phi_j(x)^{p} \dx \no \\
     &=:J_1+J_2+J_3+J_4+J_5.
   \end{align}
We estimate each term present on the right-hand side of \eqref{expan101}. For $J_1$, using $\rho_j \ge r$, we write 
\begin{align*}
      \rho_j^{-d} \norm{f}_{L^q(B_j)} \left|B_j \cap \{u<k_j\}\right|^{\frac{q-1}{q}} & =  \norm{f}_{L^q(B_j)} \rho_j^{-\frac{d}{q}} \left(\frac{\left|B_j \cap \{u<k_j\}\right|}{|B_j|}\right)^{\frac{q-1}{q}} \\
     & \le  \norm{f}_{L^q(B_R)} r^{-\frac{d}{q}} \left(\frac{\left|B_j \cap \{u<k_j\}\right|}{|B_j|}\right)^{\frac{q-1}{q}}.
\end{align*}
From the same estimate as in the proof of Proposition \ref{local-boundedness-I}, we get 
\begin{align*}
    J_2\le C(d,s,p) 2^{jp} \fint_{B_j} \overline{w}_j(x)^p \dx. 
\end{align*}
Further,
\begin{align*}
    J_3 & \le C 2^{j(d+sp)} \rho_j^{sp} \left(\int_{\rd \setminus B_j}   
     \left( \frac{1}{|x_0-y|^{d+sp}} + \Kp(|x_0-y|)\right) \overline{w}_j(y)^{p-1}
     \dy \right) \fint_{B_j} \overline{w}_j(x)\phi_j(x)^{p} \dx \\
     & \le C 2^{j(d+sp)} \rho_j^{sp} \Bigg( r^{-sp} \left(\frac{r}{R}\right)^{sp} \operatorname{Tail}_+\left(u_{-} ; x_0, R\right)^{p-1} + k_j^{p-1} \rho_j^{-sp} \\
     & \quad +  k_j^{p-1} \int_{B_1 \setminus B_j} \frac{(-\log(|x_0-y|))_+}{|x_0 -y|^{d+sp}} \dy \Bigg) \fint_{B_j} \overline{w}_j(x) \dx,
\end{align*}
where we use $(-\log(|x_0-y|))_+ = 0$ for $|x_0-y| \ge 1$. Further, using $\rho_j \ge r$, 
\begin{align*}
    \int_{B_1 \setminus B_j} \frac{(-\log(|x_0-y|))_+}{|x_0 -y|^{d+sp}} \dy \le (-\log \rho_j) \int_{\rd \setminus B_j} \frac{1}{|x_0 -y|^{d+sp}} \dy =  (-\log \rho_j) \frac{\rho_j^{-sp}}{sp} \le (-\log \rho_j) \frac{r^{-sp}}{sp}. 
\end{align*}
Therefore, using the definition of $k_j$ in \eqref{notation-2}, and using $r\le \rho_j \le 6r$,
\begin{align*}
    J_3 & \le  C 2^{j(d+sp)} \frac{\rho_j^{sp}}{-\log(2 \rho_j)} r^{-sp} \left(  \left(\frac{r}{R}\right)^{sp} \operatorname{Tail}_+\left(u_{-} ; x_0, R\right)^{p-1} + k_j^{p-1} (1-\log(\rho_j)) \right) \fint_{B_j} \overline{w}_j(x) \dx \\
    & \le  C 2^{j(d+sp)} \frac{\rho_j^{sp}}{-\log(2 \rho_j)} k_j^{p-1} (1-\log(\rho_j)) r^{-sp}  \fint_{B_j} \overline{w}_j(x) \dx \\
    & \le C 2^{j(d+sp)} \frac{(1-\log(r))}{-\log(12r)} \rho_j^{sp} k_j^{p-1}r^{-sp}  \fint_{B_j} \overline{w}_j(x) \dx \le  C 2^{j(d+sp)} \rho_j^{sp} k_j^{p-1}r^{-sp}  \fint_{B_j} \overline{w}_j(x), 
\end{align*}
since for $r \in [0, \frac{1}{24}]$, the function $g(r) = (1-\log(r))(-\log(12r))^{-1}$ is bounded. 
To estimate $J_4$, we note that the nonlocal integral survives over $\rd\setminus B_R$. Using the fact that $(k_j-u)_-=(u-k_j)_+\leq u_+$, we estimate
\begin{align*}
    J_4 &=\frac{\rho_j^{sp}}{-\log(2 \rho_j)} \left( \underset{x \in \text{supp}(\phi_j)}{\esssup}\, \int_{\rd \setminus B_R}  \Km(|x-y|) \underline{w}_j(y)^{p-1}\dy \right) \fint_{B_{j}} \overline{w}_j(x)\phi_j(x)^{p} \dx \no\\
     & \le C 2^{j(d+sp)} \frac{\rho_j^{sp}}{-\log(2 \rho_j)} \left(\int_{\rd \setminus B_R}   
     \left( \frac{1}{|x_0-y|^{d+sp}} + \Km(|x_0-y|)\right) (k_j-u)_-^{p-1}
     \dy \right) \\
     & \quad \fint_{B_j} \overline{w}_j(x) \dx \\
     & \le C 2^{j(d+sp)} \frac{\rho_j^{sp}}{-\log(2 \rho_j)}  r^{-sp} \left(\frac{r}{R}\right)^{sp} \operatorname{Tail}_-\left(u_{+} ; x_0, R\right)^{p-1} \fint_{B_j} \overline{w}_j(x) \dx.
\end{align*}
Therefore, again using the definition of $k_j$ in \eqref{notation-2} and using the boundedness of $g(r)$, 
\begin{align*}
    J_4 \le C  2^{j(d+sp)}\rho_j^{sp} k_j^{p-1} g(r) r^{-sp}  \fint_{B_j} \overline{w}_j(x) \dx\leq C 2^{j(d+sp)} \rho_j^{sp} k_j^{p-1} r^{-sp}  \fint_{B_j} \overline{w}_j(x) \dx.
\end{align*}
Finally, similarly like the homonym term in the proof of Proposition \ref{local-boundedness-I}, and using Remark \ref{support-positive-kernel}, we estimate 
\begin{align*}
    J_5 \le \rho_j^{sp} 2^{j{(d+sp)}}(\rho_j^{-sp}+C_1)\fint_{B_j} \overline{w}_j(x)^p \dx. 
\end{align*}
Accumulating all the estimates, from \eqref{expan101}, we obtain 
\begin{align*}
     & \frac{\rho_j^{sp-d}}{-\log(2 \rho_j)} \iint_{B_j \times B_j} \K(|x-y|) |\overline{w}_j(x)\phi_j(x) - \overline{w}_j(y)\phi_j(y)|^p \dxy  \\
     & \le k_j \frac{\rho_j^{sp}}{-\log(2 \rho_j)} \norm{f}_{L^q(B_R)} r^{-\frac{d}{q}} \left(\frac{\left|B_j \cap \{u<k_j\}\right|}{|B_j|}\right)^{\frac{q-1}{q}} + C2^{j(p+sp+d)}   \fint_{B_j} \overline{w}_j(x)^p \dx \\
     & \quad + C 2^{j(p+sp+d)} \rho_j^{sp} k_j^{p-1} r^{-sp}  \fint_{B_j} \overline{w}_j(x) \dx.
\end{align*}
Further, for any $a_1, a_2 \ge 1$, we observe that 
\begin{align*}
    \fint_{B_j} |\overline{w}_j(x)|^{a_1} \dx \le \left( \fint_{B_j} |\overline{w}_j(x)|^{a_1 a_2} \dx \right)^{\frac{1}{a_2}} \le k_j^{a_1} \left(\frac{|B_j \cap \{ u < k_j\}|}{|B_j|} \right)^{\frac{1}{a_2}}.
\end{align*}
Therefore, 
\begin{align*}
   & \frac{\rho_j^{sp-d}}{-\log(2 \rho_j)} \iint_{B_j \times B_j} \K(|x-y|) |\overline{w}_j(x)\phi_j(x) - \overline{w}_j(y)\phi_j(y)|^p \dxy  \\
   & \le k_j \frac{\rho_j^{sp}}{-\log(2 \rho_j)} \norm{f}_{L^q(B_R)} r^{-\frac{d}{q}} \left(\frac{\left|B_j \cap \{u<k_j\}\right|}{|B_j|}\right)^{\frac{q-1}{q}} + C2^{j(p+sp+d)}  k_j^p \left(\frac{\left|B_j \cap \{u<k_j\}\right|}{|B_j|}\right)^{\frac{q-1}{q}} \\
   & \quad + C 2^{j(p+sp+d)} \rho_j^{sp} k_j^p r^{-sp} \left(\frac{\left|B_j \cap \{u<k_j\}\right|}{|B_j|}\right)^{\frac{q-1}{q}}.
\end{align*}
From the definition of $k_j$ we note that 
\begin{align*}
    k_j \ge \delta k \ge (1-\log(r))^{-\frac{1}{p-1}}r^{\frac{qsp-d}{q(p-1)}} \norm{f}_{L^q(B_R)}^{\frac{1}{p-1}},
\end{align*}
which implies 
\begin{align*}
    k_j^{p-1}(1-\log(r)) r^{\frac{d-qsp}{q}} \ge \norm{f}_{L^q(B_R)}.
\end{align*}
With the help of this inequality, we obtain
\begin{align*}
    & \frac{\rho_j^{sp-d}}{-\log(2 \rho_j)} \iint_{B_j \times B_j} \K(|x-y|) |\overline{w}_j(x)\phi_j(x) - \overline{w}_j(y)\phi_j(y)|^p \dxy  \\
   & \le k_j^{p} \rho_j^{sp} r^{-sp} \frac{(1-\log(r))}{-\log(12r)}\left(\frac{\left|B_j \cap \{u<k_j\}\right|}{|B_j|}\right)^{\frac{q-1}{q}} + C2^{j(p+sp+d)} k_j^p  \left(\frac{\left|B_j \cap \{u<k_j\}\right|}{|B_j|}\right)^{\frac{q-1}{q}} \\
   & \le C 2^{j(p+sp+d)} k_j^p \rho_j^{sp} r^{-sp} \left(\frac{\left|B_j \cap \{u<k_j\}\right|}{|B_j|}\right)^{\frac{q-1}{q}} \\
   & \le C 2^{j(p+sp+d)} k_j^p \left(\frac{\left|B_j \cap \{u<k_j\}\right|}{|B_j|}\right)^{\frac{q-1}{q}},
\end{align*}
where in the third and last inequalities we use the boundedness of $g$ and $\rho_j \le 6r$ respectively.
Further, using the fractional Poincare inequality (Proposition \ref{poincare}), for some $C=C(d,p,s)$, we have  
\begin{align*}
  & C \frac{\rho_j^{sp-d}}{-\log(2 \rho_j)} \iint_{B_j \times B_j} \K(|x-y|) |\overline{w}_j(x)\phi_j(x) - \overline{w}_j(y)\phi_j(y)|^p \dxy  \\
  & \ge C \left( \fint_{B_j} |\overline{w}_j \phi_j|^{p^*_s} \dx \right)^{\frac{p}{p^*_s}} \ge C \left( \fint_{B_{j+1}} |\overline{w}_j \phi_j|^{p^*_s} \dx \right)^{\frac{p}{p^*_s}} \ge (k_j-k_{j+1})^p \left( \frac{|B_{j+1} \cap \{ u < k_{j+1}\}|}{|B_{j+1}|} \right)^{\frac{p}{p^*_s}}. 
\end{align*}
Therefore, 
\begin{align*}
    (k_j-k_{j+1})^p \left( \frac{|B_{j+1} \cap \{ u < k_{j+1}\}|}{|B_{j+1}|} \right)^{\frac{d-sp}{d}} \le C 2^{j(p+sp+d)} k_j^p \left(\frac{\left|B_j \cap \{u<k_j\}\right|}{|B_j|}\right)^{\frac{q-1}{q}}.
\end{align*}
Let 
\begin{align*}
    Y_j := \left( \frac{|B_j \cap \{ u<k_j\}|}{|B_j|} \right)^{\frac{q-1}{pq}}.
\end{align*}
From \eqref{obser} there exists $C=C(d,s,p)>1$ such that
\begin{align*}
&\left(Y_{j+1}\right)^{p\frac{d-sp}{d}\frac{q}{q-1}} \leq C 2^{j(d+2 p+p s)} Y_j^p \implies Y_{j+1}\leq C\hat{C}^j\left(Y_j\right)^{1+\beta},
\end{align*}
where using the fact that $q>\frac{d}{sp}$, we get $\beta := \frac{d}{d-sp}\frac{q-1}{q}-1>0$, and $\hat{C}:= 2^{\frac{d(d+2p+ps)(q-1)}{pq(d-sp)}}>1$. We choose $c_0 = C$ and $b = \hat{C}$ in Lemma \ref{iteration}. Using \eqref{expan9}, 
\begin{align*}
k_0  = \frac{3}{2} \delta k & \le 2\delta k - \frac 12(1-\log(r))^{-\frac{1}{p-1}}\Bigg(\left(\frac{r}{R}\right)^{\frac{sp}{p-1}} \left(\operatorname{Tail}_+\left(u_{-} ; x_0, R\right)+\operatorname{Tail}_-\left(u_{+} ; x_0, R\right)\right)\\
& \quad + r^{\frac{qsp-d}{q(p-1)}}\norm{f}_{L^q(B_R)}^{\frac{1}{p-1}}\Bigg)-\frac{\var}{2}.
\end{align*}
Further, using \eqref{expan2}, for some $C=C(d,s,p)$ and for every $\delta \in (0, \frac{1}{4})$, we have 
\begin{align*}
    Y_0 \le \left( \frac{C}{\tau \log \frac{1}{2 \delta}} \right)^{\frac{q-1}{pq}}.
\end{align*}
We now choose $\delta=\delta(d,s,p,\tau)$ as
\begin{align*}
0<\delta:=\frac{1}{4} \exp \left(-\frac{C_1 c_0^{\frac{pq}{\beta(q-1)}} b^{\frac{pq}{\beta^2(q-1)}}}{\tau}\right)<\frac{1}{4},
\end{align*}
so that the estimate $Y_0 \leq c_0^{-\frac{1}{\beta}} b^{-\frac{1}{\beta^2}}$ holds in Lemma \ref{iteration}. Therefore, applying Lemma \ref{iteration}, we infer that $Y_j \ra 0$ as $j \ra \infty$, which immediately gives
\begin{align*}
\underset{B_{4 r}(x_0)}{\essinf } \,u \geq \delta k,
\end{align*}
and hence \eqref{expan8} holds. This completes the proof.
\end{proof}

We conclude this section with the following tail estimate for a weak solution to \eqref{main_PDE}.

\begin{proposition}[Tail estimate]\label{Tail}
Let $s \in (0,s_0), d>sp,$ and $f$ be as given in \eqref{weight}. We consider $B_R(x_0)\subset \Omega$ with $\text{diam}(B_R(x_0))\leq 1$. Assume that $u$ is a weak solution to \eqref{main_PDE} satisfying $u \geq 0$ in $B_R(x_0)$. Let $r>0$ be such that $B_r (x_0)\subset B_{R}(x_0)$. Then we have
\begin{align*}
&\operatorname{Tail}_+(u_+,x_0,r)+\operatorname{Tail}_-(u_-,x_0,r)\no\\ 
&\leq C(d,s,p)\bigg((1-\log r)^{\frac{1}{p-1}}\underset{B_{r}(x_0)}{\esssup}\,u+\left(\frac{r}{R}\right)^{\frac{sp}{p-1}}\left[\operatorname{Tail}_{+}(u_-;x_0,R)+\operatorname{Tail}_{-}(u_+;x_0,R)\right]\no\\
& \quad+r^{\frac{qsp-d}{q(p-1)}}\norm{f}^{\frac{1}{p-1}}_{L^q(B_R(x_0))}\bigg).
\end{align*}
\end{proposition}

\begin{proof}
 Let $M=\esssup_{B_r} (u)$ and $\phi\in \C_c^{\infty}(B_{\frac{r}{2}})$ be a cut-off function such that
\begin{align*}
    0\leq\phi\leq 1 \text{ in } B_{\frac{r}{2}}, \; \phi=1 \text{ in } B_{\frac{r}{4}}, \text{ and } |\nabla\phi|\leq\frac{8}{r} \text{ in } B_{\frac{r}{2}}.
\end{align*}
For $w=u-2M$, taking $w\phi^p$ as a test function,
\begin{align}\label{tailtest}
&\int_{B_r} f(x)w\phi^p\dx \no \\ 
&=\iint_{B_r \times B_r}\K(|x-y|) |u(x)-u(y)|^{p-2}(u(x)-u(y)) (w(x)\phi(x)^p-w(y)\phi(y)^p) \dxy\no\\
&\quad+2\iint_{\mathbb{R}^d\setminus B_r\times B_r} \K(|x-y|) |u(x)-u(y)|^{p-2}(u(y)-u(x)) w(y)\phi(y)^p\dxy\no\\
&:=I_1+I_2.
\end{align}
Note that, $\K(|x-y|) \ge 0$ for every $x, y \in B_r$. In view of $(4.11)$ in \cite[Pages 1827--1828]{KuusiHarnack},
\begin{align}\label{tailestI2}
I_1&\geq-C(d,s,p)M^p\iint_{B_r \times B_r}\K(|x-y|) |\phi(x)-\phi(y)|^p \dxy.
\end{align}
Now, using $\abs{\phi(x) - \phi(y)} \le C\frac{\abs{x-y}}{r}$ for $x,y \in B_{r}$, and using \eqref{int-1},
\begin{align*}
\iint_{B_{r} \times B_{r}} \K(|x-y|) |\phi(x)-\phi(y)|^p \dxy \le C(d,s,p)r^{d-sp}(1- \log(r)).
\end{align*}
Thus, \eqref{tailestI2} yields
\begin{align}{\label{I_2Fin}}
    I_1\geq -C(d,s,p)M^pr^{d-sp}(1- \log(r)).
\end{align}
Further, 
\begin{align}\label{tailestI4}
        \int_{B_r} |f(x)||u-2M|\phi^p\dx \le M \int_{B_r} |f(x)| \dx \le M \norm{f}_{L^q(B_R)} |B_r|^{\frac{q-1}{q}},
\end{align} 
For the term $I_2$, the kernel $\K(|x-y|)$ can change sign. We therefore split
\begin{align}{\label{tail22}}
    I_2&=\iint_{B_R\setminus B_r \times B_r}\K(|x-y|) |u(x)-u(y)|^{p-2}(u(y)-u(x)) w(y)\phi(y)^p\dxy\no\\&\quad +\iint_{\mathbb{R}^d\setminus B_R \times B_r}\K(|x-y|) |u(x)-u(y)|^{p-2}(u(y)-u(x)) w(y)\phi(y)^p\dxy=:J+L.
\end{align}
For the first term, the kernel $\K(|x-y|)$ is positive for every $x\in B_R\setminus B_r$ and $y\in B_r$. Moreover, for every $x\in\rd \setminus B_r$ and $y\in B_r$, 
\begin{align}\label{bound-4}
    |u(x)-u(y)|^{p-2}(u(y)-u(x)) w(y) \ge M(u(x)-M)_{+}^{p-1} - 2M\chi_{\{u(x)<M\}}(u(y)-u(x))_{+}^{p-1}.
\end{align}
Hence
\begin{align*}
J\geq & \iint_{B_R \backslash B_r \times B_r} \K(|x-y|)  M(u(x)-M)_{+}^{p-1} \phi(y)^p \dxy \no\\
& -\iint_{B_R \backslash B_r \times B_r} 2 \K(|x-y|) M\chi_{\{u(x)<M\}}(u(y)-u(x))_{+}^{p-1}\phi(y)^p \dxy =:  J_1-J_2.
\end{align*}
Note that for $x\in \rd\setminus B_{r}$ and $y\in \operatorname{supp}(\phi)=B_{\frac{r}{2}}$, we have 
\begin{align}\label{bound-3}
\frac{|x-x_0|}{|y-x|}\leq 1+\frac{|y-x_0|}{|y-x|}\leq 1+\frac{\frac{r}{2}}{\frac{r}{2}}=2, \text{ and } \frac{|y-x|}{|x-x_0|}\leq 1+\frac{|y-x_0|}{|x-x_0|}\leq 1+\frac{\frac{r}{2}}{r}=\frac{3}{2}.
\end{align}
Also, $(-\log(|x-y|))_-= 0$ in $B_R\times B_r$ as $R\leq \frac 12$. Using \eqref{bound-3} and $-\log(|y-x|) \ge - \log(\frac{3}{2}|x-x_0|)$, we have the following lower estimate:
\begin{align*}
J_1 & \geq CM \iint_{B_R \backslash B_r \times B_r}  \K(|x-y|) u(x)_+^{p-1}\phi(y)^p\dxy\\&\quad-CM^p \iint_{B_R\backslash B_r \times B_r}  \K(|x-y|) \phi(y)^p \dxy \\
& \geq CM \iint_{B_R \backslash B_r \times B_{r}} \frac{1+ (-\log(\frac{3}{2}|x-x_0|)_+)}{|x-x_0|^{d+sp}} u(x)_+^{p-1} \phi(y)^p\dxy \\
&\quad- CM^p \iint_{B_R\backslash B_r \times B_r}  \frac{1+ (-\log(|x-x_0|)_+)}{|x-x_0|^{d+sp}} \phi(y)^p\dxy \\
& \ge CM r^d \int_{B_R \backslash B_r} \frac{1+ (-\log(\frac{3}{2}|x-x_0|)_+)}{|x-x_0|^{d+sp}} u(x)_+^{p-1} \dx- CM^p r^{d-sp}(1-\log(r)),
\end{align*}
where in the last inequality we use the derivation of the integral over $\rd \setminus B_r$ (see \eqref{log-1}). 
Also,
\begin{align*}
J_2 & \leq 2 M^p \iint_{B_R \backslash B_r \times B_r} \K(|x-y|)\phi(y)^p\dxy \leq CM^pr^{d-sp}(1-\log(r)).
\end{align*}
Therefore, 
\begin{align}\label{tail-final-1}
    J & \ge CM r^d \int_{B_R \backslash B_r} \frac{1+ (-\log(\frac{3}{2}|x-x_0|)_+)}{|x-x_0|^{d+sp}} u(x)_+^{p-1} \dx\no \\ & \quad- CM r^d\int_{B_R \backslash B_r } \frac{(-\log(|x-x_0|)_-)}{|x-x_0|^{d+sp}} u(x)_+^{p-1} \dx - C(d,s,p)M^p r^{d-sp}(1-\log(r)),
\end{align}
where $C=C(d,s,p)$. Here we have also used $(-\log(|x-y|))_-= 0$ in $B_R$. To estimate $L$, we split 
\begin{align*}
    L= &\iint_{\mathbb{R}^d\setminus B_R \times B_r} \left(\frac{C_{d,s,p}B_{d,s,p}}{|x-y|^{d+sp}} + p\Kp(|x-y|)\right)  |u(x)-u(y)|^{p-2}(u(y)-u(x)) w(y)\phi(y)^p\dxy \\
    \quad & - p\iint_{\mathbb{R}^d\setminus B_R \times B_r} \Km(|x-y|)|u(x)-u(y)|^{p-2}(u(y)-u(x)) w(y)\phi(y)^p\dxy:=L_1 - L_2.
\end{align*}
Since the kernel presents in $L_1$ is non-negative, using \eqref{bound-4}, 
\begin{align*}
L_1 & \geq \iint_{\mathbb{R}^d\setminus B_R \times B_r} \left(\frac{C_{d,s,p}B_{d,s,p}}{|x-y|^{d+sp}} + p\Kp(|x-y|)\right)  M(u(x)-M)_{+}^{p-1} \phi(y)^p \dxy \no\\
& \quad -\iint_{\mathbb{R}^d\setminus B_R \times B_r} 2 \left(\frac{C_{d,s,p}B_{d,s,p}}{|x-y|^{d+sp}} + p\Kp(|x-y|)\right) M\chi_{\{u(x)<M\}}(u(y)-u(x))_{+}^{p-1}\phi(y)^p \dxy \no\\
&=: L_{1,1}-L_{1,2},
\end{align*}
where we get similar to $J_1$,
\begin{align*}
    L_{1,1} \ge C(d,s,p)M r^d \int_{\rd \backslash B_R} \frac{1+ (-\log(\frac{3}{2}|x-x_0|)_+)}{|x-x_0|^{d+sp}} u(x)_+^{p-1} \dx - CM^p r^{d-sp}(1-\log(r)).
\end{align*}
Further, note that for $y \in B_r$, $(u(y)-u(x))_+ \le (u(y)-M)_+ + (M-u(x))_+ = (M-u(x))_+ \leq M+u(x)_-$. Hence using \eqref{bound-3},
\begin{align*}
    L_{1,2} & \le C \iint_{\mathbb{R}^d\setminus B_r \times B_r} \frac{1+ (-\log(|x-y|)_+)}{|x-y|^{d+sp}} M\chi_{\{u(x)<M\}}(u(y)-u(x))_{+}^{p-1}\phi(y)^p \dxy \\
    & \le  C M \iint_{\mathbb{R}^d\setminus B_R \times B_r}  \frac{1+ (-\log(|x-x_0|)_+)}{|x-x_0|^{d+sp}} (M + u(x)_-)^{p-1} \phi(y)^p \dxy \\
    & \le C M |B_r| \int_{\mathbb{R}^d\setminus B_R} \frac{1+ (-\log(|x-x_0|)_+)}{|x-x_0|^{d+sp}} (M + u(x)_-)^{p-1} \dx \\
    & \le C |B_r|\left( M^p (R^{-sp} + \widetilde{C}) + M  \int_{\mathbb{R}^d\setminus B_R} \frac{1+ (-\log(|x-x_0|)_+)}{|x-x_0|^{d+sp}} u(x)_-^{p-1} \dx \right),
\end{align*}
where $C=C(d,s,p)$. The last inequality uses Remark \ref{support-positive-kernel}. 
Further, for small $r$, $$|B_r|(R^{-sp} + \widetilde{C}) \le C(d,s,p) \left( r^{d-sp} + r^d \right) \le C(d,s,p)  r^{d-sp}(1-\log(r)).$$
Therefore, 
\begin{align}\label{L1fin}
    L_1 & \ge C(d,s,p)M r^d \int_{\rd \backslash B_R} \frac{1+ (-\log(\frac 32|x-x_0|)_+)}{|x-x_0|^{d+sp}} u(x)_+^{p-1} \dx \no \\
    & \quad - C(d,s,p) M r^d R^{-sp} \operatorname{Tail}_{+}(u_-;x_0,R)^{p-1} - C(d,s,p)M^p r^{d-sp}(1-\log(r)).
\end{align}
Finally, we derive a lower estimate of $L_2$. For that purpose, we require the following upper estimate compared to \eqref{bound-4}:

For every $x\in \rd\setminus B_R$ and $y\in B_r$, 
\begin{align}\label{bound4.1}
|u(x)-u(y)|^{p-2}(u(y)-u(x))w(y)\leq 2M u(x)_+^{p-1}-Mu(x)_-^{p-1}.
\end{align}
Indeed when $u(x)>u(y)$, we have $u(x)>0$ as $u(y)>0$ in $B_r$. Further, $$|u(x)-u(y)|^{p-2}(u(y)-u(x)) w(y)=(u(x)-u(y))^{p-1}(2M-u(y))\leq 2Mu(x)^{p-1}=2Mu(x)_+^{p-1}.$$ 
When $u(x)<u(y)$, we have $$|u(x)-u(y)|^{p-2}(u(y)-u(x))=(u(y)-u(x))_+^{p-1}\geq (-u(x))_+^{p-1}=u(x)_-^{p-1}.$$ 
This gives $$|u(x)-u(y)|^{p-2}(u(y)-u(x)) w(y)\leq u(x)_-^{p-1}(u(y)-2M)\leq -Mu(x)_-^{p-1}.$$ 
Thus \eqref{bound4.1} holds. 

Using \eqref{bound4.1}, we estimate $L_2$ as
\begin{align*}
L_2&=p\iint_{\mathbb{R}^d\setminus B_R \times B_r} \Km(|x-y|)|u(x)-u(y)|^{p-2}(u(y)-u(x)) w(y)\phi(y)^p\dxy\no\\ & \leq p\iint_{\mathbb{R}^d\setminus B_R \times B_r} \Km(|x-y|)(2M u(x)_+^{p-1}-Mu(x)_-^{p-1})\phi(y)^p\dxy\no\\ & \leq  C M \iint_{\mathbb{R}^d\setminus B_R \times B_r}  \frac{1+ (-\log(|x-x_0|)_-)}{|x-x_0|^{d+sp}} u(x)_+^{p-1}\phi(y)^p \dxy \no\\&-CM\iint_{\mathbb{R}^d\setminus B_R \times B_r} \frac{(-\log(\frac 12|x-x_0|)_-)}{|x-x_0|^{d+sp}}u(x)_-^{p-1}\phi(y)^p\dxy\no\\ & \leq C M r^d\int_{\mathbb{R}^d\setminus B_R }\frac{1+ (-\log(|x-x_0|)_-)}{|x-x_0|^{d+sp}} u(x)_+^{p-1}\dxy\no\\ & \quad -CM\iint_{\mathbb{R}^d\setminus B_R \times B_r} \frac{(-\log(\frac 12|x-x_0|)_-)}{|x-x_0|^{d+sp}}u(x)_-^{p-1}\phi(y)^p\dxy,
\end{align*}
for some $C=C(d,s,p)$. 
Therefore noting \eqref{L1fin}, for some $C=C(d,s,p)$, we have 
\begin{align}\label{Lfin}
 L& \ge CM r^d \int_{\rd \backslash B_R} \frac{1+ (-\log(\frac 32|x-x_0|)_+)}{|x-x_0|^{d+sp}} u(x)_+^{p-1} \dx \no\\&\quad +CM\iint_{\mathbb{R}^d\setminus B_R \times B_r} \frac{1+ (-\log(\frac 12|x-x_0|)_-)}{|x-x_0|^{d+sp}}u(x)_-^{p-1}\phi(y)^p\dxy\no \\
    & \quad - C M r^d R^{-sp} \operatorname{Tail}_{+}(u_-;x_0,R)^{p-1} - C M^p r^{d-sp}(1-\log(r))\no\\&\quad-CM r^d\int_{\mathbb{R}^d\setminus B_R }  \frac{1+ (-\log(|x-x_0|)_-)}{|x-x_0|^{d+sp}} u(x)_+^{p-1}\dx\no\\& \ge CM r^d \int_{\rd \backslash B_R} \frac{1+ (-\log(\frac 32|x-x_0|)_+)}{|x-x_0|^{d+sp}} u(x)_+^{p-1} \dx \no\\&\quad +CMr^d\int_{\mathbb{R}^d\setminus B_R } \frac{1+ (-\log(\frac 12|x-x_0|)_-)}{|x-x_0|^{d+sp}}u(x)_-^{p-1}\dx\no \\
    & \quad - C M r^d R^{-sp} \operatorname{Tail}_{+}(u_-;x_0,R)^{p-1}- C M r^d R^{-sp} \operatorname{Tail}_{-}(u_+;x_0,R)^{p-1}\no\\&\quad - CM^p r^{d-sp}(1-\log(r)).
\end{align}
Combining the estimates \eqref{I_2Fin}, \eqref{tailestI4}, \eqref{tail22}, \eqref{tail-final-1}, \eqref{Lfin} and using them in \eqref{tailtest}, we obtain
\begin{align}\label{tailestfinal}
&\int_{\rd \backslash B_r} \frac{1+ (-\log(\frac 32|x-x_0|)_+)}{|x-x_0|^{d+sp}} u(x)_+^{p-1} \dx+\int_{\mathbb{R}^d\setminus B_R } \frac{1+ (-\log(\frac 12|x-x_0|)_-)}{|x-x_0|^{d+sp}}u(x)_-^{p-1}\dx\no\\
&\leq C(d,s,p)\Bigg( R^{-sp} \operatorname{Tail}_{-}(u_+;x_0,R)^{p-1}+ R^{-sp} \operatorname{Tail}_{+}(u_-;x_0,R)^{p-1}\no\\
&\quad + M^{p-1} r^{-sp}(1-\log(r))+ \norm{f}_{L^q(B_R)} |B_r|^{-\frac{1}{q}}\Bigg).
\end{align}
We now eliminate the fractions appearing inside the logarithmic functions on the left-hand side of \eqref{tailestfinal}. For brevity, we carry out the argument only for the kernel $\Kp$. Note that $$ (-\log(|x-x_0|))_+\leq \log\frac 32+(-\log \frac32 -\log(|x-x_0|))_+.$$
Therefore,
\begin{multline}\label{note}
\int_{\rd \backslash B_r} \frac{1+ (-\log(|x-x_0|)_+)}{|x-x_0|^{d+sp}} u(x)_+^{p-1} \dx \\ \leq  \int_{\rd \backslash B_r} \frac{1+\log \frac 32+ (-\log(\frac 32|x-x_0|)_+)}{|x-x_0|^{d+sp}} u(x)_+^{p-1} \dx.
\end{multline}
Denoting the right hand side of \eqref{tailestfinal} by $A$, and noting that each term present in the left are non-negative, we get 
$$ \int_{\rd \backslash B_r} \frac{1+\log \frac 32}{|x-x_0|^{d+sp}} u(x)_+^{p-1} \dx\leq  (1+\log\frac 32)A, \text{ and } \int_{\rd \backslash B_r} \frac{(-\log(\frac 32|x-x_0|)_+)}{|x-x_0|^{d+sp}} u(x)_+^{p-1} \dx\leq  A.
$$
So by \eqref{note},
$$
\int_{\rd \backslash B_r} \frac{1+ (-\log(|x-x_0|)_+)}{|x-x_0|^{d+sp}} u(x)_+^{p-1} \dx\leq (2+\log\frac{3}{2})A,
$$
 and similarly
$$
\int_{\rd \backslash B_R} \frac{1+ (-\log(|x-x_0|)_-)}{|x-x_0|^{d+sp}} u(x)_-^{p-1} \dx\leq (2-\log\frac{1}{2})A.
$$
Finally, since $u_-\equiv 0$ in $B_R\setminus B_r$, the second integral present on the left-hand side of \eqref{tailestfinal} can be extended to $\rd \setminus B_r$. Thus from \eqref{tailestfinal} we conclude: 
 \begin{align*}
     &\quad r^{-sp}\left(\operatorname{Tail}_+(u_+,x_0,r)^{p-1}+\operatorname{Tail}_-(u_-,x_0,r)^{p-1}\right)\\&\leq C(d,s,p)\bigg(R^{-sp} \operatorname{Tail}_{+}(u_-;x_0,R)^{p-1} +R^{-sp}\operatorname{Tail}_{-}(u_+;x_0,R)^{p-1}\\&\quad+ M^{p-1} r^{-sp}(1-\log(r))+ \norm{f}_{L^q(B_R)} |B_r|^{-\frac{1}{q}}\bigg),
 \end{align*}
 which implies
 \begin{align*}
    & \operatorname{Tail}_+(u_+,x_0,r)+\operatorname{Tail}_-(u_-,x_0,r)\\&\quad\leq C(d,s,p)\bigg((1-\log r)^{\frac{1}{p-1}}\underset{B_{r}(x_0)}{\esssup}\,u+\left(\frac{r}{R}\right)^{\frac{sp}{p-1}}\operatorname{Tail}_{+}(u_-;x_0,R)\\&\quad+\left(\frac{r}{R}\right)^{\frac{sp}{p-1}}\operatorname{Tail}_{-}(u_+;x_0,R)+r^{\frac{sp-\frac dq}{p-1}}\norm{f}^{\frac{1}{p-1}}_{L^q(B_R)}\bigg).
 \end{align*}
 This completes the proof. 
\end{proof} 

\section{Harnack and weak Harnack inequalities}\label{sec: harnack}

From the application of the Krylov-Safonov covering argument (\cite[Lemma 2.5]{KuusiHarnack}) with the expansion of positivity (Proposition \ref{growth-lemma}), we get the following preliminary version of the weak Harnack inequality. The proof follows standard arguments; see, for example, \cite[Lemma 4.1]{KuusiHarnack}.

\begin{proposition}[Weak Harnack inequality]\label{weakharnack-1}
Let $s \in (0,s_0), d>sp,$ and $f$ be as given in \eqref{weight}. We consider $B_R(x_0)\subset \Omega$ with $\text{diam}(B_R(x_0))\leq 1$. Assume that $u$ is a weak supersolution to \eqref{main_PDE} satisfying $u \geq 0$ in $B_R(x_0) \subset \Omega$. Let $0<r\leq \frac{1}{24}$ be such that $B_r(x_0) \subset B_{\frac R{16}}(x_0)$. Then there exists $\var \in (0,1)$ and $C=C(d,s,p)>1$ such that 
\begin{align}\label{wh-1.1}
 \left( \fint_{B_r(x_0)} u^{\var} \dx \right)^{\frac{1}{\var}} &  \le C \underset{B_{r}\left(x_0\right)}{{\essinf}}\, u +  C (1-\log r)^{-\frac{1}{p-1}}\Bigg(\left(\frac{r}{R}\right)^{\frac{sp}{p-1}} \left[\operatorname{Tail}_+\left(u_{-} ; x_0, R\right)+\operatorname{Tail}_-\left(u_{+} ; x_0, R\right)\right] \no \\
&\quad + r^{\frac{qsp-d}{q(p-1)}} \norm{f}_{L^q(B_R(x_0))}^{\frac{1}{p-1}} \Bigg).
\end{align}
\end{proposition}

From the local boundedness and the tail estimate, we now obtain the other half of the Harnack inequality. For that, we require the following iteration lemma from \cite[Lemma 1.1]{acta}.  
\begin{lemma}\label{iteration1}
Let  $0\leq T_0\leq \widetilde{t} \leq T_1$ and assume that $f:[T_0,T_1]\to[0,\infty)$ is a nonnegative bounded function. Suppose that for $T_0\leq \widetilde{t} <\hat{t} \leq T_1$, 
\begin{equation*}
f(\widetilde{t})\leq A(\hat{t}-\widetilde{t})^{-\alpha} +B +\theta f(\hat{t}),
\end{equation*}
where $A,B,\alpha,\theta$ are nonnegative constants and $\theta<1$. 
Then there exists $C=C(\alpha,\theta)$ such that for every $\rho,R$ and $T_0\leq\rho<R\leq T_1$,
\begin{equation*}
f(\rho)\leq C(A(R-\rho)^{-\alpha}+B).
\end{equation*}
\end{lemma}

\begin{proposition}[Local boundedness II]\label{weakharnack-2}
    Let $s \in (0,s_0), d>sp,$ and $f$ be as given in \eqref{weight}. We consider $B_R(x_0)\subset \Omega$ with $\text{diam}(B_R(x_0))\leq 1$.  Assume that $u$ is a weak solution to \eqref{main_PDE} satisfying $u \geq 0$ in $B_R(x_0)$. Let $0<r\leq \frac{1}{4}$ be such that $B_r(x_0) \subset B_R(x_0)$. Then there exists $C=C(d,s,p)$ such that the following holds:
    \begin{align}\label{wh-2.1}
    \underset{B_{\frac{r}{2}}(x_0)}{\esssup}\,u & \leq C \Bigg( (1-\log(r))^{-\frac{1}{p-1}}\left(\frac{r}{R}\right)^{\frac{sp}{p-1}}\left[\operatorname{Tail}_{+}(u_-;x_0,R)+\operatorname{Tail}_{-}(u_+;x_0,R)\right]\no \\
    & \quad + (1-\log(r))^{-\frac{\sigma}{p \sigma -1}} r^{\frac{spq-d}{q(p-1)}}\norm{f}_{L^q(B_R(x_0))}^{\frac{1}{p-1}} + \left( \fint_{B_r(x_0)} u^t \dx \right)^{\frac{1}{t}} \Bigg),
    \end{align}
    for every $t \in (0, p\sigma)$. 
\end{proposition}

\begin{proof}
    For $\sigma$ given in Proposition \ref{local-boundedness-I}, we first observe that 
    \begin{align*}
        \left(sp-d+\frac{d}{\sigma}\right)\frac{\sigma}{p\sigma-1} := \frac{spq-d}{q(p-1)}.
    \end{align*}
   Let $\rho/2 \in (0,r)$. For a weak solution $u$, applying Proposition \ref{local-boundedness-I}, 
    \begin{align*}
    \underset{B_{\frac{\rho}{2}}(x_0)}{\esssup}\,u & \leq  \delta\left(1-\log(\frac{\rho}{2})\right)^{-\frac{1}{p-1}}\left( \mathrm{Tail}_+(u_+;x_0,\frac{\rho}{2}) + \mathrm{Tail}_-(u_-;x_0,\frac{\rho}{2}) \right) \\
    &\quad +C(d,s,p)\delta^{-\frac{(p-1)p_s^*}{p(p_s^*-p\sig)}}   \left(\fint_{B_{\rho}(x_0)}|u|^{p\sig}\dx\right)^{\frac1{p\sig}} \\
    & \quad + \left(1-\log(\frac{\rho}{2})\right)^{-\frac{\sigma}{p \sigma -1}} \rho^{\frac{spq-d}{q(p-1)}}\norm{f}_{L^q(B_{\rho}(x_0))}^{\frac{1}{p-1}},
\end{align*}
where $\delta\in(0,1]$. 
Now we observe that, since $1-\log(\frac{\rho}{2}) \ge 1$,
\begin{align*}
    \sigma >1 \Longleftrightarrow \frac{\sigma}{p\sigma -1} < \frac{1}{p-1} \Longrightarrow \left(1-\log(\frac{\rho}{2})\right)^{-\frac{\sigma}{p\sigma-1}} \ge \left(1-\log(\frac{\rho}{2})\right)^{-\frac{1}{p-1}}. 
\end{align*}
Hence the tail estimate (Proposition \ref{Tail}) yields
\begin{align*}
    &\underset{B_{\frac{\rho}{2}}}{\esssup}\,u\\
    & \leq  \delta\left(\underset{B_{\rho}}{\esssup}\,u+\left(1-\log(\frac{\rho}{2})\right)^{-\frac{1}{p-1}}\left(\frac{\rho}{R}\right)^{\frac{sp}{p-1}}\left[\operatorname{Tail}_{+}(u_-;x_0,R)+\operatorname{Tail}_{-}(u_+;x_0,R)\right] \right) \\
    &\quad + C\delta^{-\frac{(p-1)p_s^*}{p(p_s^*-p\sig)}}\left(\fint_{B_{\rho}}u^{p\sig}\dx\right)^{\frac1{p\sig}} + C \left[\left(1-\log(\frac{\rho}{2})\right)^{-\frac{1}{p-1}}+ \left(1-\log(\frac{\rho}{2})\right)^{-\frac{\sigma}{p \sigma -1}}\right] r^{\frac{spq-d}{q(p-1)}}\norm{f}_{L^q(B_{\rho})}^{\frac{1}{p-1}},\\
    &\leq \delta\left(\underset{B_{\rho}}{\esssup}\,u+(1-\log(r))^{-\frac{1}{p-1}}\left(\frac{\rho}{R}\right)^{\frac{sp}{p-1}}\left[\operatorname{Tail}_{+}(u_-;x_0,R)+\operatorname{Tail}_{-}(u_+;x_0,R)\right] \right) \\
    &\quad +C\delta^{-\frac{(p-1)p_s^*}{p(p_s^*-p\sig)}}\left(\fint_{B_{\rho}}u^{p\sig}\dx\right)^{\frac1{p\sig}} + C (1-\log(r))^{-\frac{\sigma}{p \sigma -1}} r^{\frac{spq-d}{q(p-1)}}\norm{f}_{L^q(B_{\rho})}^{\frac{1}{p-1}},
\end{align*}
where $C=C(d,s,p)$.   
We now set $\rho=(\eta-\eta')r,$ with $\frac12\leq\eta'<\eta\leq1$. By a covering argument, we obtain 
\begin{align*}
    &\underset{B_{\eta' r}}{\esssup}\,u \le C \Bigg( \frac{\delta^{-\frac{(p-1)p_s^*}{p(p_s^*-p\sig)}}}{(\eta-\eta')^{\frac{d}{\sigma}}} \left(\fint_{B_{\eta r}}u^{p\sig}\dx\right)^{\frac1{p\sig}} + \delta \,\underset{B_{\eta r}}{\esssup}\,u \\
    & \quad + \delta (1-\log(r))^{-\frac{1}{p-1}}\left(\frac{r}{R}\right)^{\frac{sp}{p-1}}\left[\operatorname{Tail}_{+}(u_-;x_0,R)+\operatorname{Tail}_{-}(u_+;x_0,R)\right] \\
    & \quad + (1-\log(r))^{-\frac{\sigma}{p \sigma -1}} r^{\frac{spq-d}{q(p-1)}}\norm{f}_{L^q(B_R)}^{\frac{1}{p-1}} \Bigg).
\end{align*}
where $C=C(d,s,p)$. For $t \in (0, p \sigma)$,
   \begin{align*}
       \left( \fint_{B_{\eta r}} u^{p \sigma} \dx \right)^{\frac{1}{p \sigma}} \le \left( \underset{B_{\eta r}}{\esssup}\,u \right)^{\frac{p \sigma -t}{p \sigma}} \left( \fint_{B_{\eta r}} u^{t} \dx \right)^{\frac{1}{p \sigma}}.
   \end{align*}
Choosing $\delta=\frac1{4C(d,s,p)}$ and applying Young's inequality with the conjugate pair $(\frac{p \sigma}{p \sigma -t}, \frac{p \sigma}{t})$, we get
\begin{align*}
    & \underset{B_{\eta'r}}{\esssup}\,u\leq \frac12\,\underset{B_{\eta r}}{\esssup}\,u+\frac{C}{(\eta-\eta')^{\frac{d}{t}}}\left(\fint_{B_{\eta r}}u^{t}\right)^{\frac1{t}} \\
    & + C(1-\log(r))^{-\frac{1}{p-1}} \left(\frac{r}{R}\right)^{\frac{sp}{p-1}}\left[\operatorname{Tail}_{+}(u_-;x_0,R)+\operatorname{Tail}_{-}(u_+;x_0,R)\right] \\
    & + C (1-\log(r))^{-\frac{\sigma}{p \sigma -1}} r^{\frac{spq-d}{q(p-1)}}\norm{f}_{L^q(B_R)}^{\frac{1}{p-1}},
\end{align*}
where $C=C(d,p,s)$. Now applying Lemma \ref{iteration1}, with $f(z) = \esssup_{B_{zr}} u, \hat{t} = \eta , \widetilde{t} = \eta', \alpha = \frac{d}{t}$, we obtain
\begin{align*}
    \underset{B_{\frac{r}{2}}}{\esssup}\,u & \leq C(d,s,p) \Bigg( (1-\log(r))^{-\frac{1}{p-1}}\left(\frac{r}{R}\right)^{\frac{sp}{p-1}}\left[\operatorname{Tail}_{+}(u_-;x_0,R)+\operatorname{Tail}_{-}(u_+;x_0,R)\right]\\
    & \quad + (1-\log(r))^{-\frac{\sigma}{p \sigma -1}} r^{\frac{spq-d}{q(p-1)}}\norm{f}_{L^q(B_R)}^{\frac{1}{p-1}} + \left( \fint_{B_r} u^t \right)^{\frac{1}{t}} \Bigg),
\end{align*}
for every $t \in (0, p \sigma)$, as required. 
\end{proof}

We now complete the proof of the Harnack inequality.

\noi \textbf{Proof of Theorem \ref{harnack_intro}:} For $\var$ as given in Proposition \ref{weakharnack-1}, we take $t =\var$ in Proposition \ref{weakharnack-2}. Then the required Harnack inequality is obtained by combining \eqref{wh-2.1} and \eqref{wh-1.1}. \qed

\section{H\"older continuity}\label{sec: Holder}
In this section, we complete the proof of H\"older continuity. Before stepping into the proof of H\"older estimate, we record a direct consequence of the logarithmic estimate (Lemma \ref{log-estimate}).
\begin{lemma}\label{lem: consequence of log estimate}
Let $0<R<\frac{1}{2}$ be such that $B_R(x_0)\subset \Omega$. 
Assume that $u$ is a weak solution to \eqref{main_PDE} satisfying $u \geq 0$ in $B_R(x_0)$. Let $a, t>0, b>1,$ and define
\begin{align*}
    v:=\min\left\{(\log(a+t)-\log(u+t))_{+}, \log (b)\right\}.
\end{align*}
Then for any $B_r(x_0)\subset B_{\frac{R}{2}}(x_0),$ there exists $C=C(d,p, s)$ such that the following holds:    
\begin{align*}
    &\fint_{B_r} |v-(v)_{B_r}|^p \dx\\
    &\leq \frac{C}{(-\log(2r))}\Big(t^{1-p} \left(\frac{r}{R}\right)^{sp} \left(\operatorname{Tail}_+\left(u_{-} ; x_0, R\right)^{p-1}+\operatorname{Tail}_-\left(u_{+} ; x_0, R\right)^{p-1}\right) \\
    &+ t^{1-p} r^{sp-\frac{d}{q}}\norm{f}_{L^{q}(B_{\frac{3r}{2}})}
    + 1 - \log(r)\Big).
\end{align*}
\end{lemma}
\begin{proof}
From Proposition \ref{poincare}, we have
\begin{align*}
  \fint_{B_r}|v-(v)_{B_r}|^{p}\dx \leq C(d,s,p)\frac{r^{sp-d}}{(-\log(2r))}\iint_{B_r \times B_r}\mathcal{K}^{s+\log,}_+(|x-y|)|v(x)-v(y)|^p\dxy.   
\end{align*}
By definition of $v$ and using $\Kp\leq \K$ in $B_r,$ we also have
\begin{align*} 
\iint_{B_r \times B_r}\mathcal{K}^{s+\log,}_+(|x-y|)|v(x)-v(y)|^p\dxy \leq \iint_{B_r \times B_r} \K(|x-y|)\left|\log\left(\frac{u(y)+t}{u(x)+t}\right)\right|^p\dxy.
\end{align*}
Now, the desired inequality follows applying Lemma \ref{log-estimate}.
\end{proof}
Now, let us fix some notation. For any $j \in \N,$ let $0<r<\frac{R}{2},$ for some $R>0$ such that $B_R(x_0)\subset \Omega.$  Also, let
\begin{align*}
    r_j:=\gamma^{j} \frac{r}{2}, \quad \gamma, r \in \left(0, 1/4\right], \quad B_j:=B_{r_j}(x_0).
\end{align*}
Furthermore, we define
\begin{align*}
    \frac{1}{2}\omega(r_0)=\frac{1}{2}\omega\left(\frac{r}{2}\right):=C(1-\log r)^{-\frac{1}{p-1}}\Big[\operatorname{Tail}_{+}(u; x_0, \frac{r}{2})&+\operatorname{Tail}_{-}(u; x_0, \frac{r}{2})\Big]+C\left(\fint_{B_r(x_0)}|u|^{p\sigma}\dx\right)^{\frac{1}{p\sigma}}\\
    &+C(1-\log r)^{-\frac{\sigma}{p\sigma-1}}r^{\frac{1}{p-1}(sp-\frac{d}{q})}\norm{f}_{L^q(B_{\frac{3r}{2}}(x_0))}^{\frac{1}{p-1}},
\end{align*}
and 
\begin{align*}
    \omega(r_j):=\left(\frac{r_j}{r_0}\right)^{\alpha}\omega\left(r_0\right),  \text{ for some } \alpha<\frac{sp-\frac{d}{q}}{p-1}.
\end{align*}
Note that $\omega(r_0)$ is precisely the right hand side of the local boundedness estimate obtained in Proposition \ref{local-boundedness-I} and the substitution of $\sigma$ gives the exact power of $r.$ 

Thus, we have 
\begin{align}\label{Eq 5.1}
\frac{1}{2}\omega(r_0)\geq C\frac{\operatorname{Tail}_{\mathrm{mod}}(u; x_0,\frac{r}{2})}{ (1-\log r)^{\frac{1}{p-1}}}+C\left(\fint_{B_r(x_0)}|u|^{p\sigma}\dx\right)^{\frac{1}{p\sigma}}+C\frac{r^{\frac{1}{p-1}(sp-\frac{d}{q})}\norm{f}_{L^q(B_{\frac{3r}{2}}(x_0))}^{\frac{1}{p-1}}}{(1-\log r)^{\frac{\sigma}{p\sigma-1}}}.
\end{align}
The following theorem is key to proving the H\"older estimate.
\begin{theorem}[Oscillation estimate]
Grant the above setting and let $u$ be the weak solution to  \eqref{main_PDE}. Then we have the following oscillation control:
\begin{align}\label{osc estimate}
\mathop{\mathrm{osc}}\limits_{B_j} u\equiv \sup_{B_j} u-\inf_{B_j} u \leq \omega (r_j), \text{ for all } j=0, 1, 2....    
\end{align}
\end{theorem}
\begin{proof}
We use an induction process. Note that, for $j=0$, the estimate \eqref{osc estimate} holds trivially by the local boundedness estimate \eqref{loc.bounded} with $\delta=1.$ Now, we assume that \eqref{osc estimate} holds for all $i \in \{0, 1, ..., j\}$ for some $j \geq 0,$ and then show that \eqref{osc estimate} holds for $j+1.$ We have the following two alternatives.
\begin{align}\label{1st alternative}
    \frac{\left|2B_{j+1}\cap \{u \geq \inf_{B_j} u+\omega(r_j)/2\}\right|}{|2B_{j+1}|}\geq \frac{1}{2}
\end{align}
or
\begin{align}\label{2nd alternative}
    \frac{\left|2B_{j+1}\cap \{u \leq \inf_{B_j} u+\omega(r_j)/2\}\right|}{|2B_{j+1}|}\geq \frac{1}{2}
\end{align}
If \eqref{1st alternative} holds, we set $u_j:=u- \inf_{B_j} u$ and if \eqref{2nd alternative} holds, we set $u_j:=\omega(r_j)-(u-\inf_{B_j}u).$ In both cases, $u_j \geq 0$ in $B_j$ and 
\begin{align*}
  \frac{\left|2B_{j+1}\cap \{u_j \geq \omega(r_j)/2\}\right|}{|2B_{j+1}|}\geq \frac{1}{2}   
\end{align*}
holds. Also, $u_j$ is a weak solution (of course the source term $f$ may change to $-f$, but the steps all involves $\norm{f}$ and hence it would not affect) with $\sup_{B_i}|u_j|\leq 2 \omega(r_i)$ for all $i \in \{0, 1, ...., j\}.$ We show that 
\begin{align*}
    \left[\operatorname{Tail_{\mathrm{mod}}}(u_j; x_0, r_j)\right]^{p-1}\leq C\gamma^{-\alpha (p-1)}\log \left(1+\frac{1}{r_j}\right)[\omega(r_j)]^{p-1}.
\end{align*}
Since $B_j \subset B_{j-1}....\subset B_1 \subset B_0,$ we split,
\begin{align}\label{eq: tail estimate-4 terms}
\left[\operatorname{Tail_{\mathrm{mod}}}(u_j; x_0, r_j)\right]^{p-1}&= r^{sp}_j \sum_{i=1}^j\int_{B_{i-1}\setminus B_i}\left(\frac{1}{|x-x_0|^{d+sp}}+ \mathcal{K}^{s+\log}_{+}(|x-x_0|)\right)|u_j(x)|^{p-1}\dx\no\\
 &+r^{sp}_j \int_{\rd\setminus B_0}\left(\frac{1}{|x-x_0|^{d+sp}}+ \mathcal{K}^{s+\log}_{+}(|x-x_0|)\right)|u_j(x)|^{p-1}\dx\no\\
 &+r^{sp}_j \sum_{i=1}^j\int_{B_{i-1}\setminus B_i}\mathcal{K}^{s+\log}_{-}(|x-x_0|)|u_j(x)|^{p-1}\dx\no\\
 &+r^{sp}_j\int_{\rd\setminus B_0}\mathcal{K}^{s+\log}_{-}(|x-x_0|)|u_j(x)|^{p-1}\dx.
\end{align}
We note that the third integral of \eqref{eq: tail estimate-4 terms} does not survive by the construction of $B_i$ and thus
\begin{align*}
r^{sp}_j \sum_{i=1}^j\int_{B_{i-1}\setminus B_i}\mathcal{K}^{s+\log}_{-}(|x-x_0|)|u_j(x)|^{p-1}\dx=0.    
\end{align*}
Let $\beta = sp-\alpha (p-1)>0.$ We estimate the first two integrals of \eqref{eq: tail estimate-4 terms}.
\begin{align*}
&r^{sp}_j \sum_{i=1}^j\int_{B_{i-1}\setminus B_i}\left(\frac{1}{|x-x_0|^{d+sp}}+ \mathcal{K}^{s+\log}_{+}(|x-x_0|)\right)|u_j(x)|^{p-1}\dx\\
&\quad+r^{sp}_j\int_{\rd\setminus B_0}\left(\frac{1}{|x-x_0|^{d+sp}}+ \mathcal{K}^{s+\log}_{+}(|x-x_0|)\right)|u_j(x)|^{p-1}\dx \\
&\leq r^{sp}_j\sum_{i=1}^j \left(\sup_{B_{i-1}}|u_j|\right)^{p-1}\big[1+(-\log r_j)_{+}\big]\int_{\rd\setminus B_i} \frac{\dx}{|x-x_0|^{d+sp}} \\
&\quad+C\,\left(\frac{r_j}{r_0}\right)^{sp}\left[\operatorname{Tail_{\mathrm{mod}}}(u;x_0,r_0)\right]^{p-1}\\
&\quad+C\,\big[1+(-\log r_0)_{+}\big]\left(\frac{r_j}{r_0}\right)^{sp}\Big(\sup_{B_0}|u|^{p-1}+[\omega(r_0)]^{p-1}\Big),
\end{align*}
where on $\rd\setminus B_0,$ we used $|u_j(x)|^{p-1}\le C\big(|u(x)|^{p-1}+\sup_{B_0}|u|^{p-1}+[\omega(r_0)]^{p-1}\big)$, the bound 
\begin{align*}
 \int_{\rd\setminus B_0}\left(\frac{1}{|x-x_0|^{d+sp}}+\mathcal{K}^{s+\log}_{+}(|x-x_0|)\right)|u(x)|^{p-1}\dx\leq r_0^{-sp}[\operatorname{Tail_{\mathrm{mod}}}(u;x_0,r_0)]^{p-1}
\end{align*}
and the estimate
\begin{align*}
\int_{\rd\setminus B_0}\left(\frac{1}{|x-x_0|^{d+sp}}+\mathcal{K}^{s+\log}_{+}(|x-x_0|)\right)\dx
&= \omega_{d-1}\int_1^\infty \tau^{-1-sp}\dtau + \omega_{d-1}\int_{r_0}^1 \tau^{-1-sp}\log\frac 1\tau\,\dtau\\
&\leq \frac{\omega_{d-1}}{sp}\,r_0^{-sp}\big(1+\log\frac{1}{r_0}\big).
\end{align*}
Now we estimate the fourth integral. Since $\mathcal{K}^{s+\log}_{-}$ is supported in $\{|x-x_0|>1\}\subset \rd\setminus B_0$ (because $r_0\le1$), we have
\begin{align*}
\int_{\rd\setminus B_0}\mathcal{K}^{s+\log}_{-}(|x-x_0|)\dx
=\omega_{d-1}\int_1^\infty \tau^{-1-sp}\log\tau\,\dtau
=\frac{\omega_{d-1}}{(sp)^2}\le C\,r_0^{-sp},
\end{align*}
and using the bound of $|u_j|^{p-1}$ as above, we obtain
\begin{align*}
r^{sp}_j\int_{\rd\setminus B_0}\mathcal{K}^{s+\log}_{-}(|x-x_0|)|u_j(x)|^{p-1}\dx
&\le C\left(\frac{r_j}{r_0}\right)^{sp}\left[\operatorname{Tail_{\mathrm{mod}}}(u;x_0,r_0)\right]^{p-1}\\
&\quad+C\left(\frac{r_j}{r_0}\right)^{sp}\Big(\sup_{B_0}|u|^{p-1}+[\omega(r_0)]^{p-1}\Big).
\end{align*}
 Thus, using \eqref{Eq 5.1}, we estimate the integrals over $\rd \setminus B_0$ as
\begin{align*}
&r^{sp}_j \int_{\rd\setminus B_0} \left(\frac{1}{|x-x_0|}+\mathcal{K}^{s+\log}_{\mathrm{mod}}(|x-x_0|)\right)|u_j(x)|^{p-1}\dx\\
&\leq C \left(\frac{r_j}{r_0}\right)^{sp}[\operatorname{Tail}_{\mathrm{mod}}(u; x_0, r_0)]^{p-1}+C\,\big[1+(-\log r_0)_{+}\big]\left(\frac{r_j}{r_0}\right)^{sp}\Big(\sup_{B_0}|u|^{p-1}+[\omega(r_0)]^{p-1}\Big)\\
&\leq C\,\big[1+(-\log r_0)_{+}\big]\left(\frac{r_j}{r_0}\right)^{sp}[\omega(r_0)]^{p-1}\leq C[1+(-\log r_j)_{+}]\left(\frac{r_j}{r_1}\right)[w(r_0)]^{p-1},
\end{align*}
where in the last line above, we use $r_0>r_1\ge r_j$, we have $(r_j/r_0)^{sp}\le (r_j/r_1)^{sp}$ and $1+(-\log r_0)_+\le 1+(-\log r_j)_+.$  Combining the above estimates and using 
\begin{align*}
\int_{\rd\setminus B_i}|x-x_0|^{-d-sp}\dx=\frac{\omega_{d-1}}{sp}\,r_i^{-sp}, \quad \text{and }\quad \sup_{B_0}|u_j|\leq 2\omega(r_0), 
\end{align*}
we arrive at
\begin{align*}
\left[\operatorname{Tail_{\mathrm{mod}}}(u_j;x_0,r_j)\right]^{p-1}
\le C\,\big[1+(-\log r_j)_{+}\big]\sum_{i=1}^j \left(\frac{r_j}{r_i}\right)^{sp}[\omega(r_{i-1})]^{p-1}.
\end{align*}
It remains to sum the geometric series. Inserting $[\omega(r_{i-1})]^{p-1}=[\omega(r_j)]^{p-1}(r_{i-1}/r_j)^{\alpha(p-1)}$ and $r_j/r_i=\gamma^{j-i}$,
\begin{align*}
\sum_{i=1}^j \left(\frac{r_j}{r_i}\right)^{sp}[\omega(r_{i-1})]^{p-1}
&=[\omega(r_j)]^{p-1}\sum_{i=1}^j \left(\frac{r_{i-1}}{r_i}\right)^{\alpha(p-1)}\left(\frac{r_j}{r_i}\right)^{sp-\alpha(p-1)}\\
&=\gamma^{-\alpha(p-1)}[\omega(r_j)]^{p-1}\sum_{i=1}^j \gamma^{(j-i)\beta}\\
&\le \frac{\gamma^{-\alpha(p-1)}}{1-\gamma^{\beta}}[\omega(r_j)]^{p-1}
\le \frac{4^{\beta}}{\beta\log 4}\,\gamma^{-\alpha(p-1)}[\omega(r_j)]^{p-1},
\end{align*}
where in the last step we used $\gamma\le\frac14$ and $4^\beta-1\ge\beta\log4$, so that $\frac{1}{1-\gamma^\beta}\le\frac{1}{1-4^{-\beta}}\le\frac{4^\beta}{\beta\log4}$. Hence,
\begin{align*}
\left[\operatorname{Tail_{\mathrm{mod}}}(u_j;x_0,r_j)\right]^{p-1}
\le C\,\gamma^{-\alpha(p-1)}\big[1+(-\log r_j)_{+}\big][\omega(r_j)]^{p-1},
\end{align*}
and since $r_j\le1$ gives $1+(-\log r_j)_+ = 1+\log\frac1{r_j}\le 2\log\!\big(1+\frac1{r_j}\big)$, we conclude
\begin{align*}
\left[\operatorname{Tail_{\mathrm{mod}}}(u_j; x_0, r_j)\right]^{p-1}\leq C\,\gamma^{-\alpha (p-1)}\log\left(1+\frac{1}{r_j}\right)[\omega(r_j)]^{p-1},
\end{align*}
with $C=C(d,p,s, \beta)$, independent of $\gamma$ and $j$.

For some $k >0,$ we consider the function $v$ defined as
\begin{align*}
    v(x):=\min\left\{\left[\log(\frac{\omega(r_j)}{2+t})-\log(u_j(x)+t)\right]_{+}, k \right\}
\end{align*}
Applying Lemma \ref{lem: consequence of log estimate} with this function $v,$ we get
\begin{align}\label{Eq 5.6}
    &\fint_{2B_{j+1}}|v-(v)_{2B_{j+1}}|^p \dx \no \\
    &\leq \frac{C}{(-\log(4r_{j+1}))_{+}}\Big(t^{1-p} \left(\frac{r_{j+1}}{r_j}\right)^{sp} \left(\operatorname{Tail}_+\left((u_j)_{-} ; x_0, r_j\right)^{p-1}+\operatorname{Tail}_-\left((u_j)_{+} ; x_0, r_j\right)^{p-1}\right) \no \\
    &+ t^{1-p}r_{j+1}^{sp-\frac{d}{q}}\norm{f}_{L^{q}(B_{3r_{j+1}})} + 1 - \log(r_{j+1})\Big)\no \\
    &\leq C \left(t^{1-p}\left(\frac{r_{j+1}}{r_j}\right)^{sp} \frac{[\operatorname{Tail}_{\mathrm{mod}}(u_j; x_0, r_j)]^{p-1}}{(-\log(4r_{j+1}))_{+}}+\frac{t^{1-p}r_{j+1}^{sp-\frac{d}{q}}\norm{f}_{L^{q}(B_{3r_{j+1}})}}{(-\log(4r_{j+1}))_{+}}+\frac{1-\log(r_{j+1})}{(-\log(4r_{j+1}))_{+}}\right)\no \\
 &\leq C \left(t^{1-p}\left(\frac{r_{j+1}}{r_j}\right)^{sp} \gamma^{-\alpha (p-1)} [\omega(r_{j})]^{p-1}+{\gamma^{\,sp-\frac dq}}\;t^{1-p}[\omega(r_j)]^{p-1}+1+\frac{\log (4e)}{\log 8} \right),
\end{align}
where in the last step, we used, for $r_{j+1}\leq \frac{1}{32},$
\begin{align*} 
\frac{1-\log(r_{j+1})}{-\log(4r_{j+1})}=\frac{1-\log(r_{j+1})}{-\log(r_{j+1})-2\log 2}=1+\frac{1+2\log 2}{\log \frac{1}{4r_{j+1}}} \leq 1+\frac{1+2\log 2}{3\log 2}=1+\frac{\log (4e)}{\log 8}
\end{align*}
 and, since $r_j\leq \frac18$ implies 
$\log\big(1+\frac{1}{r_j}\big)\leq \log\frac{2}{r_j}
=\log 2+\log\frac{1}{r_j}\leq \frac{4}{3}\log\frac{1}{r_j}$, 
while $\gamma\leq\frac14$ gives 
$-\log(4r_{j+1})=\log\frac{1}{r_j}+\log\frac{1}{4\gamma}\geq \log\frac{1}{r_j}$,
\begin{align*}
\frac{\left[\operatorname{Tail}_{\mathrm{mod}}(u_j; x_0, r_j)\right]^{p-1}}{(-\log(4r_{j+1}))_{+}}
&\leq C\,\gamma^{-\alpha (p-1)}\,
\frac{\log\big(1+\frac{1}{r_j}\big)}{(-\log(4r_{j+1}))_{+}}\,[\omega(r_j)]^{p-1}\\
&\leq C\,\gamma^{-\alpha (p-1)}\,
\frac{\frac{4}{3}\log\frac{1}{r_j}}{\log\frac{1}{r_j}+\log\frac{1}{4\gamma}}\,[\omega(r_j)]^{p-1}
\;\leq\; C\,\gamma^{-\alpha (p-1)}\,[\omega(r_j)]^{p-1}.
\end{align*}
Moreover, Since $r_{j+1}\le \frac14\cdot\frac r2$, we have $3 r_{j+1}=\frac{3}{2}\gamma^{j+1}r\le\frac{3r}{8}<\frac{3r}{2}$, so $B_{3r_{j+1}}\subset B_{3\frac{r}{2}}$ and hence
\begin{align*}
\norm{f}_{L^q(B_{3r_{j+1}})}\le\norm{f}_{L^q(B_{3\frac{r}{2}})}.
\end{align*}
From \eqref{Eq 5.1}, we have $(1-\log r)^{\frac{\sigma}{p\sigma-1}}\omega(r_0)\ge C\,r^{\frac{sp-\frac{d}{q}}{p-1}}\norm{f}_{L^q(B_{3\frac{r}{2}})}^{\frac{1}{p-1}}$ which gives,
\begin{align*}
\norm{f}_{L^q(B_{3\frac{r}{2}})}\le C\,r^{-\left(sp-\frac dq\right)}\,(1-\log r)^{\frac{\sigma(p-1)}{p\sigma -1}}[\omega(r_0)]^{p-1} \leq C\,r^{-\left(sp-\frac dq\right)}\,(1-\log r)[\omega(r_0)]^{p-1}
\end{align*}
where in the last inequality we use $\frac{\sigma(p-1)}{p\sigma-1}<1.$
Therefore, using $[\omega(r_0)]^{p-1}=\gamma^{-j\alpha(p-1)}[\omega(r_j)]^{p-1}$ and $\left(\frac{r_{j+1}}{r}\right)^{sp-\frac dq}=\left(\frac{\gamma^{j+1}}{2}\right)^{sp-\frac dq}$, we get
\begin{align*}
\frac{t^{1-p}r_{j+1}^{sp-\frac dq}\norm{f}_{L^q(B_{3r_{j+1}})}}{(-\log(4r_{j+1}))_+}
&\le C\frac{t^{1-p}\left(\frac{r_{j+1}}{r}\right)^{sp-\frac dq}(1-\log r)[\omega(r_0)]^{p-1}}{(-\log(4r_{j+1}))_+}\\
&=\frac{C}{2^{sp-\frac dq}}\frac{t^{1-p}[\omega(r_j)]^{p-1}\,\gamma^{(j+1)\left(sp-\frac dq\right)-j\alpha(p-1)}(1-\log r)}{(-\log(4r_{j+1}))_+}\\
&=\frac{C}{2^{sp-\frac dq}}\,\frac{t^{1-p}\,[\omega(r_j)]^{p-1}\,\gamma^{\left(sp-\frac dq\right)+j\left[\left(sp-\frac dq\right)-\alpha(p-1)\right]}(1-\log r)}{(-\log(4r_{j+1}))_+}.
\end{align*}
Since, $\alpha< \frac{sp-\frac{d}{q}}{p-1}$ and $(1-\log r )/ (-\log(4r_{j+1}))$ is bounded, we conclude
\begin{align}\label{eq: use in recursive inequality}
\frac{t^{1-p}\,r_{j+1}^{sp-\frac dq}\norm{f}_{L^q(B_{3r_{j+1}})}}{(-\log(4r_{j+1}))_+}
\le C\gamma^{\,sp-\frac dq}t^{1-p}[\omega(r_j)]^{p-1}.
\end{align}
Now, choosing $t=\gamma^{a}\omega(r_j)$ with $0< a<\frac{1}{p-1}\min\left\{\beta,\;sp-\frac dq\right\}$ in \eqref{Eq 5.6}, we get
\begin{align*}
\fint_{2B_{j+1}}|v-(v)_{2B_{j+1}}|^p\dx &\le C \left(\gamma^{\beta-a(p-1)}+\gamma^{\,sp-\frac dq-a(p-1)}+1+\frac{\log (4e)}{\log 8} \right)\\
&\leq C\left(3+\frac{\log(4e)}{\log8}\right)=:C_0
\end{align*}
where $C_0$ is independent of $\gamma$ and $j.$

Next, we denote $\widetilde B\equiv 2B_{j+1}$ and follow the steps of \cite[pp. 1296]{Palatucci2016} verbatim to arrive at
\begin{align*}
    \frac{\left|\widetilde B \cap \left\{u_j \leq 2 \varepsilon \omega (r_j)\right\}\right|}{|\widetilde B|} \leq \frac{\widetilde{C_1}}{\log (1/\gamma)},
\end{align*}
where we set $\varepsilon := \gamma^a.$

We are now in a position to start the iteration procedure, which will end up showing the oscillation estimate given in \eqref{osc estimate}. For any $i=0,1,2,...$, we define
\begin{align*} 
\varrho_i=(1+2^{-i})r_{j+1},\quad \widetilde{\varrho}_i := \frac{\varrho_{i} +\varrho_{i+1}}{2}, \quad  B^i=B_{\varrho_i}, \quad \widetilde B^i=B_{\widetilde \varrho_i}
\end{align*}
and corresponding cut-off functions;
\begin{align*}
\phi_i\in \mathcal{C}_0^\infty(\widetilde B^i), \,\,\, 0\leq \phi_i\leq 1,\,\,\,  \phi_i\equiv 1\text{ on } B^{i+1}, \,\,\,\text{and} \,\,\, |D\phi_i|<C\varrho_i^{-1}.
\end{align*}
Also,  we set
\begin{align*} 
k_i=(1+2^{-i})\varepsilon\,\omega(r_j),\quad \text{and}\quad  w_i=k_i-u_j,\quad  \overline{w}_i := (k_i-u_j)_+,\quad \underline{w}_i := (k_i-u_j)_-
\end{align*}
and
\begin{align*}     
A_i=\frac{|B^{i}\cap \{u_j \leq k_i\}|}{|B^{i}|} = \frac{|B^{i}\cap \{w_i > 0\}|}{|B^{i}|}.
\end{align*}
First, we denote the energy as
\begin{align*}
 \mathcal E_{i} (\overline{w}_i, \phi_i )= \iint_{B^i \times B^i} \K(|x-y|) |\overline{w}_{i}(x)\phi_i(x) - \overline{w}_{i}(y)\phi_i(y)|^p \dxy,
\end{align*}
and recall the Caccioppoli inequality from Lemma \ref{energy estimate},
\begin{align}\label{cacc-2}
\mathcal E_i (\overline{w}_i, \phi_i) &\le C \int_{B^i} |f(x)| \overline{w}_i(x) \phi_i(x)^{p} \dx\no \\
     &+  C \iint_{B^i \times B^i} \K(|x-y|) \left(\max \{ \overline{w}_{i}(x) , \overline{w}_{i}(y) \}\right)^p |\phi_i(x) - \phi_i(y)|^p \dxy \no \\
     & + C \left( \underset{x \in \text{supp}(\phi)}{\esssup}\, \int_{\rd \setminus B^i}       \left( \frac{1}{|x-y|^{d+sp}} + \Kp(|x-y|)\right) \overline{w}_i(y)^{p-1}
     \dy \right) \int_{B^i} \overline{w}_i(x)\phi_i(x)^{p} \dx \no \\
     & + C \left( \underset{x \in \text{supp}(\phi)}{\esssup}\, \int_{\rd \setminus B^i}       \Km(|x-y|) \underline{w}_i(y)^{p-1}\dy \right) \int_{B^i} \overline{w}_i(x)\phi_i(x)^{p} \dx \no \\
     & + C\left( \underset{x \in \text{supp}(\phi)}{\esssup}\, \int_{\rd \setminus B^i} \Km(|x-y|) 
       \dy \right) \int_{B^i} \overline{w}_i(x)^p\phi_i(x)^{p} \dx\no\\
       &=\mathrm{I}+\mathrm{II}+\mathrm{III}+\mathrm{IV}+\mathrm{V}.
 \end{align}
 We divide our proof into four steps. Before we begin, let us collect some observations. Recalling that $t:=\varepsilon\,\omega(r_j)$, we have $k_i\in(t,2t]$, $k_i-k_{i+1}=2^{-(i+1)}t$, and $\varrho_i\in(r_{j+1},2r_{j+1}]$. Note that
\begin{align}\label{eq:ball inclusions}
2r_{j+1}=2\gamma\,r_j\le \frac12 r_j<r_j,\qquad\text{and hence}\qquad B_{r_{j+1}}\subset B^i\subset B_{2r_{j+1}}\subset B_{r_j}\subset B_R .
\end{align}
Since $u_j\ge0$ on each $B^j$, hence $u_j\geq 0$ on each $B^i,$ and 
\begin{equation}\label{eq:wbound}
0\le w_i\le k_i\le 2t\quad\text{on } B^i\cap \{u_j \leq k_i\},\qquad\text{and}\qquad
w_i\le 2t+|u_j|\quad\text{on }\rd .
\end{equation}
 Note that on $B^i\cap \{u_j \leq k_i\}$, we have $\overline{w}_i=w_i.$ One can also see that, with $i=0$ one has $B^0=B_{2r_{j+1}}=2B_{j+1}$ and $k_0=2t$, and this gives
\begin{equation}\label{eq:A0}
A_0\ \le\ \frac{\widetilde{C_1}}{\log(1/\gamma)}.
\end{equation}

\noindent \textbf{Step 1} {(lower bound of the energy):}  In this step, we show a lower bound of the energy in terms of $A_{i+1}.$ Since $u_j$ is a weak solution, and hence $\overline{w}_i\phi_i \in W^{s+\log, p}_0(\widetilde B^i)\subset W^{s+\log, p}(B^{i}).$ Hence, using Proposition \ref{poincare} together with $2\varrho_i<1,$ we have
\begin{align*}
   & \left(\fint_{B^i} |\overline{w}_i \phi_i(x)|^{p^*_s}\dx\right)^{\frac{p}{p^*_s}}\leq C\left(\fint_{B^{i}}|\overline{w}_i\phi_i(x)-(\overline{w}_i \phi_i)_{B^i}|^{p^*_s}\dx\right)^{\frac{p}{p^*_s}}+C|(\overline{w}_i \phi_i)_{B^i}|^p\\
    &\leq C\frac{\varrho^{sp-d}_i }{(-\log(2\varrho_i))_{+}}\iint_{B^i\times B^i}\Kp(|x-y|)|\overline{w}_i\phi_i(x)-\overline{w}_i\phi_i(y)|^p\dx\dy+C|(\overline{w}_i \phi_i)_{B^i}|^p\\
    &=C \frac{\varrho^{sp-d}_i}{-\log(2\varrho_i)} \mathcal{E}_i(\overline{w}_i, \phi_i)+C|(\overline{w}_i \phi_i)_{B^i}|^p.
\end{align*}
Next, we estimate the mean. Since $\overline{w}_i\phi_i\leq 2t$ on 
$B^i\cap\{u_j\leq k_i\}$ and $\overline{w}_i\phi_i=0$ on 
$B^i\cap\{u_j> k_i\}$, we have
\begin{align*}
\big|(\overline{w}_i\phi_i)_{B^i}\big|^p
=\left(\fint_{B^i}\overline{w}_i\phi_i(x)\dx\right)^p
=\left(\frac{1}{|B^i|}\int_{B^i\cap\{u_j\leq k_i\}}\overline{w}_i\phi_i(x)\dx\right)^p
\leq (2t)^p A_i^p \leq (2t)^p A_i.
\end{align*}

On the other hand, we note that $\overline{w}_i=(k_i-u_j)_{+}\geq k_i-k_{i+1}=2^{-(i+1)}t$ in $B^{i+1}\cap\{u_j \leq k_{i+1}\}.$ Hence, we have
\begin{align*}
    \frac{1}{|B^i|}\int_{B^i} |\overline{w}_i\phi_i(x)|^{p^*_s}\dx & \geq \frac{1}{|B^i|}\int_{B^{i+1}\cap\{u_j \leq k_{i+1}\}}\overline{w}_i(x)^{p^*_s}\dx \\
    &\geq \frac{\left|B^{i+1}\cap\{u_j \leq k_{i+1}\}\right|}{|B^{i+1}|}\cdot\frac{|B^{i+1}|}{|B^i|}(k_i-k_{i+1})^{p^*_s},
\end{align*}
and
\begin{align*}
\left(\fint_{B^i} |\overline{w}_i \phi_i(x)|^{p^*_s}\dx\right)^{\frac{p}{p^*_s}} & \geq (2^{-(i+1)}t)^p \left(\frac{\left|B^{i+1}\cap\{u_j 
\leq k_{i+1}\}\right|}{|B^{i+1}|}\right)^{\frac{p}{p^*_s}}\left(\frac{1}{2^d}\right)^{\frac{p}{p^*_s}}
\\&=(2^{-(i+1)}t)^p \left(A_{i+1}\right)^{\frac{p}{p^*_s}}\left(\frac{1}{2^d}\right)^{\frac{p}{p^*_s}}.
\end{align*}
Combining the above two inequalities, we get
\begin{align}\label{energy lower bound}
(2^{-(i+1)}t)^p \left(A_{i+1}\right)^{\frac{p}{p^*_s}}\frac{1}{2^d} \leq  C \frac{\varrho^{sp-d}_i}{-\log(2\varrho_i)} \mathcal{E}_i(\overline{w}_i, \phi_i)+C(2t)^p A_i.   
\end{align}
\\
\textbf{Step 2} {(upper bound of the energy):} In this step, we obtain an upper bound of the energy. To be more precise,  we estimate the right-hand side of the above Caccioppoli inequality \eqref{cacc-2} to derive an upper estimate for 
$\left(-\varrho^{sp-d}_i/\log(2\varrho_i)\right) \mathcal{E}_i(\overline{w}_i, \phi_i).$ which appears on the right-hand side of \eqref{energy lower bound}.

\noindent \textbf{Estimate of $\mathrm{I}:$} We estimate the term coming from the right-hand side of the equation simply using H\"older's inequality
\begin{align}\label{eq: 1st estimate of holder}
    \frac{\varrho^{sp-d}_i}{-\log(2\varrho_i)} \int_{B^i} |f(x)|\overline{w}_i(x)\phi_i(x)^p\dx&\leq C\frac{2t (r_{j+1})^{sp-d}}{-\log(4r_{j+1})} \int_{B^i\cap \{u_j\leq k_i\}}|f|\dx\no\\
    &\leq C\frac{2t (2r_{j+1})^{sp-d}}{-\log(4r_{j+1})}\norm{f}_{L^q(B^i)}\left(|B^i|A_{i}\right)^{\frac{q-1}{q}}\no\\
    &=C\frac{2t r^{sp-\frac{d}{q}}_{j+1}}{-\log(4r_{j+1})}\norm{f}_{L^q(B^i)}A^{\frac{q-1}{q}}_i.
\end{align}\\
\textbf{Estimate of $\mathrm{II} :$} We follow the estimate already derived in Proposition \ref{local-boundedness-I} with $|\phi_i(x)-\phi_i(y)|\leq C\frac{2^i}{r_{j+1}}|x-y|$ and obtain
\begin{align}\label{eq:2nd estimate of holder}
&\frac{\varrho^{sp-d}_i}{-\log(2\varrho_i)}\iint_{B^i \times B^i} \K(|x-y|) \left(\max \{ \overline{w}_{i}(x) , \overline{w}_{i}(y) \}\right)^p |\phi_i(x) - \phi_i(y)|^p \dxy\no \\
&\leq C(d, s, p)\frac{r^{sp}_{j+1}}{-\log(4r_{j+1})}\fint_{B^i}\int_{B^i} \K(|x-y|) \left(\max \{ \overline{w}_{i}(x) , \overline{w}_{i}(y) \}\right)^p |\phi_i(x) - \phi_i(y)|^p \dxy\no\\
&\leq C(d, s, p) 2^{ip}\fint_{B^i}w_i(y)^p\dy\leq C(d, s, p)2^{ip}\frac{1}{|B^i|}\int_{B^i\cap\{u_j\leq k_i\}} w_i(y)^p\dy \leq C(d, s, p)2^{ip}(2t)^p A_i,
\end{align}
where in the last step we used the first part of \eqref{eq:wbound}.\\
\textbf{Estimate of $\mathrm{III} :$} First, we note that, for $x\in\supp\phi_i\subset\widetilde B^i$ and $y\in\rd\setminus B^i$,
\begin{align}\label{eq:geom}
|x-y|\ \ge\ \varrho_i-\widetilde\varrho_i=\frac{\varrho_i-\varrho_{i+1}}2=2^{-(i+2)}r_{j+1},
\quad \text{and}\quad 
\frac{|x_0-y|}{|x-y|}\le 1+\frac{|x-x_0|}{|x-y|}\le 2^{\,i+4}.
\end{align}
We also note that
\begin{align}\label{eq:tail estimate-2}
    \fint_{B^i}\overline{w}_i(x)\phi_i(x)^p\dx \leq (2t) \frac{|B^i \cap \{u_j \leq k_i\}|}{|B^i|}=2t A_{i}.
\end{align}
Now we estimate
\begin{align}\label{eq:tail estimate-0}
 &\frac{r^{sp}_{j+1}}{-\log(4r_{j+1})}\left( \underset{x\in \widetilde B^i}{\esssup}\, \int_{\rd \setminus B^i}       \left( \frac{1}{|x-y|^{d+sp}} + \Kp(|x-y|)\right) \overline{w}_i(y)^{p-1}
     \dy \right) \fint_{B^i} \overline{w}_i(x)\phi_i(x)^{p} \dx\no\\
      &\leq 2t A_i \frac{r^{sp}_{j+1}}{-\log(4r_{j+1})}\left( \underset{x\in \widetilde B^i}{\esssup}\, \int_{\rd \setminus B^i}       \left( \frac{1}{|x-y|^{d+sp}} + \Kp(|x-y|)\right) \overline{w}_i(y)^{p-1}
     \dy \right)\no\\
     &\leq C\frac{2t A_i}{-\log(4r_{j+1})} (1+i)2^{i(d+sp)} \left[\operatorname{Tail}_{\mathrm{+}}(\overline{w}_i; x_0, r_{j+1})\right]^{p-1}.
\end{align}
In the last estimate above, we used the ball inclusions from \eqref{eq:ball inclusions} and the fact from \eqref{eq:geom} that all $y \in \rd\setminus B^i,$ we have
\begin{align*}
\frac{1}{|x-y|^{d+sp}}\leq \frac{2^{(i+4)(d+sp)}}{|x_0-y|^{d+sp}}, \quad \text{and}\quad (-\log|x-y|)_+\le (i+4)\log2+(-\log|x_0-y|)_+.
\end{align*}
Further, we split the domain $\rd\setminus B_{j+1}=(B_j\setminus B_{j+1})\cup(\rd\setminus B_j)$, and write
\begin{align}\label{eq:tail estimate-1}
&[\operatorname{Tail}_{+}(\overline{w}_i;x_0,r_{j+1})]^{p-1}\no\\
&= r_{j+1}^{sp}\int_{B_j\setminus B_{j+1}}\left(\frac{1}{|x_0-y|^{d+sp}}+\Kp(|x_0-y|)\right)\overline{w}_i^{\,p-1}(y)\dy\no\\
&+\Big(\frac{r_{j+1}}{r_j}\Big)^{sp}[\operatorname{Tail}_{+}(\overline{w}_i;x_0,r_j)]^{p-1}\no\\
&\leq C \left[(2t)^{p-1}(1-\log (r_{j+1}))+\left(\frac{r_{j+1}}{r_j}\right)^{sp}\left([\operatorname{Tail}_{\mathrm{mod}}(u_j; x_0, r_j)]^{p-1}+(2t)^{p-1}(1-\log r_j)\right)\right].
\end{align}
We elaborate on the estimate in \eqref{eq:tail estimate-1} for the reader's convenience. On $B_j\setminus B_{j+1}\subset B_j$ one has $\overline{w}_i\le 2t$. Using $(-\log|x_0-y|)_+\le-\log r_{j+1}$ for
$|x_0-y|\ge r_{j+1}$, we get 
\begin{align*}
&r_{j+1}^{sp}\int_{B_j\setminus B_{j+1}}\left(\frac{1}{|x_0-y|^{d+sp}}+\Kp(|x_0-y|)\right)\overline{w}_i^{\,p-1}(y)\dy\\
&\leq r^{sp}_{j+1}(2t)^{p-1} (1-\log(r_{j+1}))\int_{\rd\setminus B_{j+1}}\frac{\dy}{|x_0-y|^{d+sp}}=C(d, s, p)(2t)^{p-1}(1-\log (r_{j+1})).
\end{align*}
For the second part of the estimate in \eqref{eq:tail estimate-1}, we use \eqref{eq:wbound}, in particular, $\overline{w}_i\leq |u_j|+2t.$ Indeed, we get
\begin{align*} 
[\operatorname{Tail}_{+}(\overline{w}_i; x_0, r_j)]^{p-1} &\lesssim \operatorname{Tail}_{+}(u_j; x_0, r_j)]^{p-1}+ (2t)^{p-1}r^{sp}_j\int_{\rd\setminus B_j} \left(\frac{1}{|x_0-y|^{d+sp}}+\Kp(|x_0-y|)\right)\dy\\
&\lesssim \operatorname{Tail}_{\mathrm{mod}}(u_j; x_0, r_j)]^{p-1}+(2t)^{p-1}(1-\log(r_j)).
\end{align*}
Next, we use 
\begin{align}\label{eq:estimate for tail mod}
\left[\operatorname{Tail_{\mathrm{mod}}}(u_j; x_0, r_j)\right]^{p-1}\leq C\,\gamma^{-\alpha (p-1)}\log\left(1+\frac{1}{r_j}\right)[\omega(r_j)]^{p-1},
\end{align}
to have further control on the second part of \eqref{eq:tail estimate-1}. Indeed, we have
\begin{align*}
    &\left(\frac{r_{j+1}}{r_j}\right)^{sp}\left([\operatorname{Tail}_{\mathrm{mod}}(u_j; x_0, r_j)]^{p-1}+(2t)^{p-1}(1-\log r_j)\right)\\
    &\lesssim \gamma^{sp-\alpha(p-1)}\log\left(1+\frac{1}{r_j}\right)[\omega(r_j)]^{p-1}+\gamma^{sp}(2t)^{p-1}(1-\log r_j)\\
    &=\gamma^{\beta-a(p-1)}\log\Big(1+\frac1{r_j}\Big)t^{p-1}+\gamma^{sp}(2t)^{p-1}(1-\log r_j)\\
    &\leq 2 (2t)^{p-1}(1-\log r_j).
\end{align*}
In the last line of the above inequality, we used $\beta>a(p-1),$ $\gamma<1$ and $\frac{\log (1+1/r_j)}{1+\log (1/r_j)}\leq 2.$
Finally, substituting this estimate in \eqref{eq:tail estimate-1} and using $r_{j+1}\leq r_j,$ we obtain
\begin{align*}
 [\operatorname{Tail}_{+}(\overline{w}_i;x_0,r_{j+1})]^{p-1} \leq C (2t)^{p-1} (1-\log r_{j+1}).   
\end{align*}
Thus, from \eqref{eq:tail estimate-0}, we obtain
\begin{align}\label{eq:3rd estimate of holder}
&\frac{r^{sp}_{j+1}}{-\log(4r_{j+1})}\left( \underset{x\in \widetilde B^i}{\esssup}\, \int_{\rd \setminus B^i}       \left( \frac{1}{|x-y|^{d+sp}} + \Kp(|x-y|)\right) \overline{w}_i(y)^{p-1}
     \dy \right) \fint_{B^i} \overline{w}_i(x)\phi_i(x)^{p} \dx\no \\
     &\leq C\frac{2t A_i}{-\log(4r_{j+1})} (1+i)2^{i(d+sp)} (2t)^{p-1}(1-\log r_{j+1})=C(2t)^p 2^{i(d+sp+1)}A_i.
\end{align}
This completes the estimate of $\mathrm{III}.$

\noindent\textbf{Estimate of $\mathrm{IV} :$} Now we estimate the term, 
\begin{align*}  
\frac{\varrho^{sp-d}_i}{-\log(2\varrho_i)}\left( \underset{x \in \widetilde{B}^i}{\esssup}\, \int_{\rd \setminus B^i} \Km(|x-y|) \underline{w}_i(y)^{p-1}\dy \right) \int_{B^i} \overline{w}_i(x)\phi_i(x)^{p} \dx.
\end{align*}
Note that, in this case $\Km (|x-y|)\neq 0$ when $|x-y|\geq 1.$ Thus, using \eqref{eq:tail estimate-2}, we get
\begin{align*}
   & \frac{r_{j+1}^{sp-d}}{-\log(4r_{j+1})}\left( \underset{x \in \widetilde{B}^i}{\esssup}\, \int_{\rd \setminus B_{1/2}}       \Km(|x-y|) \underline{w}_i(y)^{p-1}\dy \right) \int_{B^i} \overline{w}_i(x)\phi_i(x)^{p} \dx\\
    &\leq \frac{r_{j+1}^{sp}2t A_i}{-\log(4r_{j+1})}\left( \underset{x \in \widetilde{B}^i}{\esssup}\, \int_{\rd \setminus B_{1/2}}       \Km(|x-y|) (2t+|u_j|)^{p-1}\dy \right)\\
    &\leq C\frac{r^{sp}_{j+1}(2t)^p A_i}{-\log(4r_{j+1})}\left(\underset{x \in \widetilde{B}^i}{\esssup}\int_{\rd\setminus B_{1/2}}\Km(|x-y|)\dy\right)\\
    &+ \frac{r_{j+1}^{sp}2t A_i}{-\log(4r_{j+1})}\left( \underset{x \in \widetilde{B}^i}{\esssup}\, \int_{\rd \setminus B_{1/2}}       \Km(|x-y|)|u_j(y)|^{p-1}\dy \right)\\
    &\leq C \widetilde{C}\frac{(2t)^p A_i}{-\log(4r_{j+1})}+C\frac{r^{sp}_{j+1}r^{-sp}_j2t A_i}{-\log(4r_{j+1})} [\operatorname{Tail}_{\mathrm{mod}}(u_j; x_0, r_j)]^{p-1},
\end{align*}
where $\widetilde{C}$ stands for the finiteness of the integral (see Remark \ref{support-positive-kernel}). In addition, one needs to use the inequality $\frac{1}{C}|x_0-y|\leq |x-y|\leq C|x_0-y|$ for some positive constant $C$, independent of $i$. Now we use the tail estimate \eqref{eq:estimate for tail mod} and $r_{j+1}\leq r_j$ to obtain
\begin{align} \label{eq:4th estimate for holder}
&\frac{\varrho^{sp-d}_i}{-\log(2\varrho_i)}\left( \underset{x \in \widetilde{B}^i}{\esssup}\, \int_{\rd \setminus B^i}       \Km(|x-y|) \underline{w}_i(y)^{p-1}\dy \right) \int_{B^i} \overline{w}_i(x)\phi_i(x)^{p} \dx\no\\
&\leq C \widetilde{C}\frac{(2t)^p A_i}{-\log(4r_{j+1})}+C\frac{2t A_i}{-\log(4r_{j})} \gamma^{sp-\alpha (p-1)}\log\left(1+\frac{1}{r_j}\right)[\omega(r_j)]^{p-1}\no\\
&\leq C \widetilde{C}\frac{(2t)^p A_i}{-\log(4r_{j+1})}+ C (2t)^{p}\gamma^{\beta-a(p-1)}\frac{\log(1+\frac{1}{r_j})}{\log \frac{1}{4r_j}}  A_i\leq C(d, s, p) (2t)^p A_i.
\end{align}

\noi \textbf{Estimate of $\mathrm{V}:$} From the estimates of the previous step, we get
\begin{align*}
\left( \underset{x \in \text{supp}(\phi)}{\esssup}\, \int_{\rd \setminus B^i} \Km(|x-y|) \dy \right) \lesssim R^{-sp},   
\end{align*}
and from \eqref{eq:tail estimate-2}, we also have
\begin{align*}
    \fint_{B^i}\overline{w}^p_i(x)\phi_i(x)^p\dx \leq (2t )^pA_{i}.
\end{align*}
Therefore, we estimate
\begin{align}\label{eq:5th estimate for holder}
   &\frac{\varrho^{sp}_i}{-\log(2\varrho_i)} \left( \underset{x \in \text{supp}(\phi)}{\esssup}\, \int_{\rd \setminus B^i} \Km(|x-y|) 
       \dy \right) \fint_{B^i} \overline{w}_i(x)^p\phi_i(x)^{p} \dx\no\\
       &\lesssim R^{sp}R^{-sp} (2t)^p A_i.
\end{align}

\noi \textbf{Step 3} {(recursive inequality)}: Combining the lower bound \eqref{energy lower bound} of the energy from Step 1 and the upper bounds \eqref{eq: 1st estimate of holder}, \eqref{eq:2nd estimate of holder}, \eqref{eq:3rd estimate of holder}, \eqref{eq:4th estimate for holder}, \eqref{eq:5th estimate for holder} of the energy from Step 2, we arrive at the recursive inequality,
\begin{align*}
(2^{-(i+1)}t)^p \left(A_{i+1}\right)^{\frac{p}{p^*_s}}\frac{1}{2^d} \leq C 2^{ip}2^{i(d+sp+1)}(2t)^{p}A_i+ C\frac{2t r^{sp-\frac{d}{q}}_{j+1}}{-\log(4r_{j+1})}\norm{f}_{L^q(B^i)}A^{\frac{q-1}{q}}_i.
\end{align*}
Now diving by $t^p$, using \eqref{eq: use in recursive inequality} and $A_i\leq 1,$ we get
\begin{align*} 
 \left(A_{i+1}\right)^{\frac{p}{p^*_s}} &\leq C 2^{i(d+sp+2p+1)}A_i+ C\frac{2t^{1-p} r^{sp-\frac{d}{q}}_{j+1}}{-\log(4r_{j+1})}\norm{f}_{L^q(B^i)}A^{\frac{q-1}{q}}_i\\
 &\leq C 2^{i(d+sp+2p+1)}A_i+ C{\gamma^{\,sp-\frac dq}}t^{1-p}[\omega(r_j)]^{p-1} A^{1-1/q}_i\\
& =C 2^{i(d+sp+2p+1)}A_i+ C\gamma^{sp-\frac dq-a(p-1)} A^{1-1/q}_i\leq C 2^{i(d+sp+2p)}A^{1-1/q}_{i}.
\end{align*}
This gives
\begin{align}\label{eq:recursive inequality}
    A_{i+1} \leq C^{\frac{p^*_s}{p}}2^{i(d+sp+2p+1)\frac{p^*_s}{p}}A_i^{\frac{(q-1)p^*_s}{qp}}.
\end{align}
The above recursive inequality \eqref{eq:recursive inequality} satisfies the statement of Lemma \ref{iteration} with
\begin{align*}
c_0= C^{\frac{p^*_s}{p}}\quad b=2^{\frac{(d+sp+2p+1)p^*_s}{p}}>1, \quad \delta= \frac{qsp-d}{q(d-sp)}>0, \quad\text{since}\quad sp>\frac{d}{q}.
\end{align*}
Also, from \eqref{eq:A0} we have
\begin{align*}
    A_0 \leq \frac{\widetilde{C_1}}{\log (1/\gamma)} \leq c^{-\frac{1}{\delta}}_0 b^{-\frac{1}{\delta^2}}=\nu_*.
\end{align*}
Thus, choosing $\gamma =\min \{1/4, e^{-\widetilde{C_1}/\nu_*}\},$ we conclude that $A_i\to 0$ as $i \to \infty.$

\noindent \textbf{Step 4} {(reduction of oscillation)}: Since $A_i \to 0,$ by definition
\begin{align*}
|B^{i}\cap \{u_j \leq k_i\}|\to |B_{r_{j+1}}(x_0)\cap \{u_j \leq t\}|=0.    
\end{align*}
Thus, $u_j(x) \geq \varepsilon w(r_j)$ for almost every $x \in B_{r_{j+1}}(x_0).$ If \eqref{1st alternative} holds, then $u_j=u-\inf_{B_j}u.$ In this case, using $B_{j+1}\subset B_j$ and the induction hypothesis $\mathop{\mathrm{osc}}\limits_{B_j} u\leq \omega(r_j),$ we have
\begin{align*}
\mathop{\mathrm{osc}}\limits_{B_{j+1}}u=\sup_{B_{j+1}}u-\inf_{B_{j+1}}u\ \le\ \sup_{B_j}u-\inf_{B_j}u-\varepsilon\omega(r_j)\ \le\ (1-\varepsilon)\,\omega(r_j).    
\end{align*}

On the other hand, if \eqref{2nd alternative} holds, then $u_j=\omega(r_j)-u+\inf_{B_j} u.$ In this case
\begin{align*}
\mathop{\mathrm{osc}}\limits_{B_{j+1}}u=\sup_{B_{j+1}}u-\inf_{B_{j+1}}u \leq \omega(r_j)-\varepsilon \omega(r_j)+\inf_{B_j} u -\inf_{B_{j+1}}u \leq (1-\varepsilon) \omega(r_j).   
\end{align*}
Hence, in any case, we have shown 
\begin{align*}
    \mathop{\mathrm{osc}}\limits_{B_{j+1}} u \leq (1-\varepsilon)\omega(r_j)=(1-\varepsilon)\left(\frac{r_j}{r_{j+1}}\right)^{\alpha}\omega(r_{j+1})=(1-\gamma^a)\gamma^{-\alpha}\omega(r_{j+1})\leq \omega(r_{j+1}),
\end{align*}
when $(1-\gamma^a)\gamma^{-\alpha}\leq 1.$ This gives that, for sufficiently small fixed $a$ and $\gamma,$ $\alpha$ needs to satisfy the relation and 
this finishes the proof.
\end{proof}

We now complete the proof of H\"older regularity.

\noi \textbf{Proof of Theorem \ref{thm: holder estimate}:} Fix $B_{r}(x_0)\subset \Omega$ for some $0<r<\frac{1}{4}$ as before. By Theorem \ref{osc estimate}, we get
\begin{align*}
    \mathop{\mathrm{osc}}\limits_{B_j} u \leq \omega(r_j)=2^{\alpha}\left(\frac{r_j} {r}\right)^{\alpha}\omega\left(\frac{r}{2}\right).
\end{align*}
Thus, using $\operatorname{Tail}_{+}(u; x_0, \frac{r}{2})+\operatorname{Tail}_{-}(u; x_0, \frac{r}{2}) \leq C \operatorname{Tail}_{\mathrm{mod}}(u; x_0, \frac{r}{2}),$ for any $\rho \in (r_{j+1}, r_j),$ we get,
\begin{align*} 
  &\mathop{\mathrm{osc}}\limits_{B_{\rho}(x_0)} u \leq \mathop{\mathrm{osc}}\limits_{B_j} u \no\\
  &\leq C \left(\frac{\rho}{r}\right)^{\alpha} \Bigg[(1-\log r)^{-\frac{1}{p-1}}\Big[\operatorname{Tail}_{\mathrm{mod}}(u; x_0, \frac{r}{2})\Big]+C\left(\fint_{B_r(x_0)}|u|^{p\sigma}\dx\right)^{\frac{1}{p\sigma}}\no\\
    &+(1-\log r)^{-\frac{\sigma}{p\sigma-1}}r^{\frac{1}{p-1}(sp-\frac{d}{q})}\norm{f}_{L^q(B_{r}(x_0))}^{\frac{1}{p-1}} \Bigg], 
\end{align*}
which completes the proof.
\qed
\section{A counterexample}\label{sec: counterexample}
In this section, we assume $p=2$, $f\equiv 0$ and $d\geq 2.$  Moreover, we assume $u \in \mathcal{C}^{\alpha}_{\mathrm{loc}}(B_R(0))\cap \mathcal{C}(\overline{B_R(0)}).$ We want to mention that the continuity up to the boundary has not been proved yet. One needs to use barrier arguments and comparison principle similar to \cite{RosOtonSerra2014, Iannizzotto2016} and also see \cite[Section 5]{ChenWeth} for boundary continuity of logarithmic Laplacian. Since this does not fall under the scope of this paper, we do not pursue it here.  

Inspired by Kassmann \cite{Kassmann2007}, the aim of this section is to show that if a solution is supported outside of $B_R(0),$ then the tail term appearing in the Harnack inequality is unavoidable. We begin with a simple but useful observation.
\begin{lemma}\label{lem: 1st section 7}
For any $x\in B_R(0),$ with $\text{diam}(B_R(0))\le \min\{\widetilde{R}(d, s), 1\},$ we have
\begin{align}\label{comparison}
   - (-\Delta)^{s+\log}\mathbb{1}_{B^c_R(0)}(x)\geq - (-\Delta)^{s+\log}\mathbb{1}_{B^c_R(0)}(0).
\end{align}
\end{lemma}
\begin{proof} It is easy to check that for every $x\in B_R(0)$, $ - (-\Delta)^{s+\log}\mathbb{1}_{B^c_R(0)}(x)$ is finite i.e., it exists pointwisely. First, we claim that 
\begin{align}\label{eq:decresing}
    r\mapsto \frac{B_{d,s}-2\log r}{r^{d+2s}}\quad  \text{is strictly decreasing for $0<r\leq 1$.}
\end{align}
Indeed, differentiating the map, we get
\begin{align*}
    \frac{{\rm d}}{\dr}\left[\frac{B_{d,s}-2\log r}{r^{d+2s}}\right]=-r^{-(d+2s)-1}\left[2+(d+2s)(B_{d,s}-2\log r)\right]<0,
\end{align*}
for any $r \in (0, 1].$

From the definition of the fractional logarithmic Laplacian, we have
    \begin{align*}
       - (-\Delta)^{s+\log}\mathbb{1}_{B^c_R}(x)=C_{d, s}\int_{B^c_R(0)}\frac{B_{d,s}-2\log(|x-y|)}{|x-y|^{d+2s}}\dy=C_{d, s}\int_{B^c_R(x)}\frac{B_{d,s}-2\log(|z|)}{|z|^{d+2s}}\dz,
    \end{align*}
    where we use the substitution $x-y=z.$ Now
\begin{align*}
   - (-\Delta)^{s+\log}\mathbb{1}_{B^c_R}(0)=C_{d, s}\int_{B^c_R(0)}\frac{B_{d,s}-2\log|z|}{|z|^{d+2s}}\dz.
\end{align*}
Note that we split the domain of integration as:
\begin{align*}
    &B^c_R(0):=\left(B_R(x)\setminus B_R(0)\right)\cup \left(B_R(x)\cup B_R(0)\right)^c, \text{ and }\\
    &B^c_R(x):=\left(B_R(0)\setminus B_R(x)\right)\cup \left(B_R(x)\cup B_R(0)\right)^c.
\end{align*}
where $\cup$ denotes the disjoint union.
Therefore, 
\begin{align}\label{Eq7.1}
    &-\left((-\Delta)^{s+\log}\mathbb{1}_{B^c_R}(x)-(-\Delta)^{s+\log}\mathbb{1}_{B^c_R}(0)\right)\no\\
    &=C_{d, s}\left[\int_{B^c_R(x)}\frac{B_{d,s}-2\log(|z|)}{|z|^{d+2s}}\dz-\int_{B^c_R(0)}\frac{B_{d,s}-2\log|z|}{|z|^{d+2s}}\dz\right]\no\\
    &=C_{d, s}\left[\int_{B_R(0)\setminus B_R(x)}\frac{B_{d,s}-2\log(|z|)}{|z|^{d+2s}}\dz-\int_{B_R(x)\setminus B_R(0)}\frac{B_{d,s}-2\log(|z|)}{|z|^{d+2s}}\dz\right].
\end{align}
We define 
\begin{align*}
    \sigma: B_R(x)\setminus B_R(0)\to B_R(0)\setminus B_R(x) , \quad \sigma (\xi)=x-\xi.
\end{align*}
Since $|\text{det} D\sigma|=1,$ we use the change of variable formula to get
\begin{align*}
\int_{B_R(x)\setminus B_R(0)}\frac{B_{d,s}-2\log(|z|)}{|z|^{d+2s}}\dz&=\int_{B_R(0)\setminus B_{R}(x)}\frac{B_{d,s}-2\log(|\sigma(z)|)}{|\sigma(z)|^{d+2s}}\dz\\
&=\int_{B_R(0)\setminus B_{R}(x)}\frac{B_{d,s}-2\log(|z-x|)}{|z-x|^{d+2s}}\dz.
\end{align*}
Now, from \eqref{Eq7.1}, we have
\begin{align}\label{Eq7.2}
&-\left((-\Delta)^{s+\log}\mathbb{1}_{B^c_R}(x)-(-\Delta)^{s+\log}\mathbb{1}_{B^c_R}(0)\right)\no\\
&=C_{d, s}\int_{B_R(0)\setminus B_R(x)}\left[\frac{B_{d,s}-2\log(|z|)}{|z|^{d+2s}}-\frac{B_{d,s}-2\log(|z-x|)}{|z-x|^{d+2s}}\right]\dz.
\end{align}
Since, $z\not\in B_R(x),$ $|z-x|\geq R,$ and $|z|\leq R.$ Also, we note that $|z-x|\leq |x|+|z|\leq 2R\leq 1.$ By decreasing property \eqref{eq:decresing}, we get \eqref{Eq7.2} is nonnegative. This completes the proof.
\end{proof}
\begin{lemma}\label{example-1}
Let $B_R(0)\subset B_{R_{e}}(0)$ with $\text{diam}(B_R)\leq \min\{\widetilde{R}(d, s), 1\}$ and $R_e:=R+e^{\frac{B_{d,s}}{2}}+1.$ Also, assume that $s<\frac{1}{2}$ and $u \in W^{s+\log,2}(B_{R}(0)) \cap L_{\log, 2}(\rd)$ is a weak solution to
\begin{align}\label{countereq}
    (-\Delta)^{s+\log}u=0\quad \text{in} \quad B_{R}(0),
\no\\
    u=\mathbb{1}_{B^c_{R_e}}\quad \text{in} \quad B^c_{R}(0).
    \end{align}
Then $\inf_{B_{R}(0)}u<0.$
\end{lemma}
\begin{center}
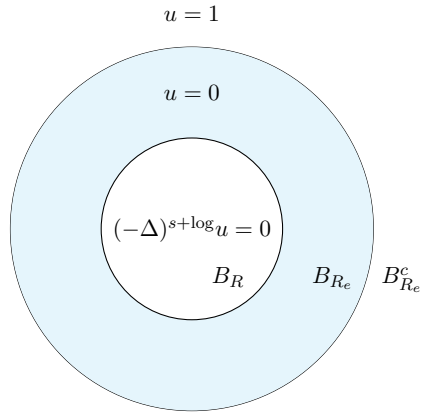

\begin{tikzpicture}[scale=0.8, transform shape]

  \def\R{1.5}    
  \def\Re{3.0}   
  \def\L{4.6}    


  \draw (0,0) circle (\Re);
  \draw (0,0) circle (\R);

     \fill[cyan!9] (0,0) circle (\Re);
    \draw[fill=white] (0,0) circle (\R);

  \node at (0,0) {$(-\Delta)^{s+\log}u= 0$};
  \node at (0,{(\R+\Re)/2}) {$u = 0$};
  \node at (0,{(\Re+\L-0.5)/2}) {$u = 1$};

  \node at (\R-0.5,-0.5) [below left] {$B_R$};
 \node at (\Re-0.2,-0.5) [below left] {$B_{R_e}$};
 \node at (\Re,-0.5) [below right] {$B^c_{R_e}$};
\end{tikzpicture}

\captionof{figure}{Construction of the function $u$.}
\end{center}
\begin{proof}
Note that the existence of a solution $u$ to \eqref{countereq} is guaranteed, see Remark \ref{exists}. 
As we consider $u$ to be continuous up to the boundary of $B_R(0)$ and $u=0$ on $\partial B_R(0)$, we have $\inf_{B_{R}(0)}u\leq 0$. We show that the infimum is strictly negative. First we observe that, $\mathbb{1}_{B_{R}(0)} \in W_0^{s+\log,2}(B_R(0))$ for $s< \frac{1}{2}$ is a valid test function (see Lemma \ref{lem:test function} for a proof). From the weak formulation, we get
\begin{align*}
    &0=\iint_{B_{R}(0)\times B_{R}(0)} \K(|x-y|)(u(x)-u(y))(\mathbb{1}_{B_{R}(0)}(x)-\mathbb{1}_{B_{R}(0)}(y))\dxy\\
    &+2\iint_{B_{R}(0)\times \rd\setminus B_{R}(0)} \K(|x-y|) (u(x)-u(y)) (\mathbb{1}_{B_{R}(0)}(x)-\mathbb{1}_{B_{R}(0)}(y)) \dxy\\
    &=2C_{d,s}\int_{B_{R}(0)}\int_{B_{R_e}(0)\setminus B_{R}(0)}\frac{(B_{d,s}-2\log(|x-y|))}{|x-y|^{d+2s}}u(y)  \dxy\\
    &\quad+2C_{d,s}\int_{B_{R}(0)}\int_{B^c_{R_e}(0)}\frac{(B_{d,s}-2\log(|x-y|))}{|x-y|^{d+2s}}(u(y)-1)  \dxy.
    \end{align*}
Using the above identity, observe that 
\begin{align}\label{eq7.1}
    0> \int_{B_R(0)}\int_{B^c_{R_e}(0)}\frac{B_{d,s}-2\log(|x-y|)}{|x-y|^{d+2s}}\dxy=\int_{B_R(0)}\int_{B^c_R(0)}\frac{B_{d,s}-2\log(|x-y|)}{|x-y|^{d+2s}}u(y)\dxy.
\end{align}
Now, we claim that for any $y\in B_R(0)$,
\begin{align}\label{claim}
    \int_{B^c_R(0)}\frac{B_{d,s}-2\log(|x-y|)}{|x-y|^{d+2s}}\dx >0.
\end{align}
If the claim holds, then using \eqref{eq7.1}, we get $\inf_{B_R(0)}u < 0.$\vspace{0.1 cm}\\

\underline{Proof of \eqref{claim}}: In view of \eqref{claim} and Lemma \ref{lem: 1st section 7}, it is enough to show that
\begin{align*}
 \int_{B^c_R(0)}\frac{B_{d,s}-2\log(|x|)}{|x|^{d+2s}}\dx>0.   
\end{align*}
Using  polar
coordinates, $|x|=\rho$,
\begin{align*}
    \int_{B^c_R(0)}\frac{B_{d,s}-2\log(|x|)}{|x|^{d+2s}}\dx
    =\omega_{d-1}\int_R^{\infty}\frac{(B_{d,s}-2\log \rho)\,\rho^{d-1}}{\rho^{d+2s}}\,\mathrm{d}\rho
    =\omega_{d-1}\int_R^{\infty}\frac{B_{d,s}-2\log \rho}{\rho^{1+2s}}\,\mathrm{d}\rho.
\end{align*}
First, we note that
\begin{align*}
    \frac{\mathrm{d}}{\mathrm{d}\rho}\left[\frac{\rho^{-2s}}{2s}\Big(2\log\rho-B_{d,s}+\frac{1}{s}\Big)\right]
    &=-\rho^{-1-2s}\Big(2\log\rho-B_{d,s}+\frac{1}{s}\Big)+\frac{\rho^{-2s}}{2s}\frac{2}{\rho}=\frac{B_{d,s}-2\log\rho}{\rho^{1+2s}},
\end{align*}
and hence
\begin{align*}
    \int_R^{\infty}\frac{B_{d,s}-2\log \rho}{\rho^{1+2s}}\mathrm{d}\rho
    =\frac{\rho^{-2s}}{2s}\Big(2\log\rho-B_{d,s}+\frac{1}{s}\Big)\Bigg|_{R}^{\infty}
    =\frac{R^{-2s}}{2s}\Big(B_{d,s}-\frac{1}{s}-2\log R\Big).
\end{align*}
Recalling (\cite[Proposition 1.1]{ChenChenHauer}),
\begin{align*}
    B_{d,s}:=\log 4+\frac{1}{s}+\psi(1-s)+\psi\Big(\frac{d+2s}{2}\Big),
\end{align*}
we obtain
\begin{align*}
    \int_{B^c_R(0)}\frac{B_{d,s}-2\log(|x|)}{|x|^{d+2s}}\dz
    =\omega_{d-1}\,\frac{R^{-2s}}{2s}\left(\log 4-2\log R+\psi(1-s)+\psi\Big(\frac{d+2s}{2}\Big)\right).
\end{align*}
Since $\operatorname{diam}(B_R(0))\le 1$ gives $R\le\frac{1}{2}$, we have
$-2\log R\ge 2\log 2$. Moreover, for $s\in(0,\frac{1}{2})$ and $d\ge 2$ the
digamma function $\psi$ is increasing, so
\begin{align*}
    \psi(1-s)+\psi\Big(\frac{d+2s}{2}\Big)\ge \psi\Big(\frac{1}{2}\Big)+\psi(1)
    =\big(-\gamma-2\log 2\big)-\gamma=-2\gamma-2\log 2.
\end{align*}
Therefore, we get
\begin{align*}
    \log 4-2\log R+\psi(1-s)+\psi\Big(\frac{d+2s}{2}\Big)
    \ge 2(\log 2-\gamma)>0,
\end{align*}
where we used $\log 2>\gamma$. Due to \eqref{comparison}, 
\begin{align*}
     - (-\Delta)^{s+\log}\mathbb{1}_{B_R^c(0)}(y)&=\int_{B^c_R(0)}\frac{B_{d,s}-2\log(|x-y|)}{|x-y|^{d+2s}}\dx\\
     &\geq \int_{B^c_R(0)}\frac{B_{d,s}-2\log(|x|)}{|x|^{d+2s}}\dx=-(-\Delta)^{s+\log}\mathbb{1}_{B^c_R}(0)>0,
\end{align*}
as required.
\end{proof}

Using the continuity of $u$ up to the boundary, and noting the boundary data for $u$, we note that the infimum of $u$ is achieved at some interior $x_0 \in B_R(0)$, i.e., $\inf_{B_R(0)}u(x)=u(x_0).$ Now, using Lemma \ref{example-1},  we show that the tail term in the Harnack inequality cannot be dropped.
Consider the following function on $\rd$,  
\begin{align*}
    \widetilde u =u-\inf_{B_R(0)}u. 
\end{align*}
We observe that $\widetilde u \ge 0$ a.e. in $B_R(0)$.
Also, $\widetilde u$ weakly solves $(-\Delta)^{s+\log}\widetilde{u}=0$ in $B_R(0)$ and $\inf_{B_R(0)}\widetilde u=\widetilde{u}(x_0)=0.$
Since $u$ is regular up to the boundary, $\widetilde{u}$ is also regular up to the boundary and by using Lemma \ref{example-1}, we have $\widetilde{u}=-\inf_{B_R(0)}u>0$ on $\partial B_R(0)$ as $u=0$ on $\partial B_R(0)$.\\
Now note that $u$ is a non-constant function in $B_R(0)$ as it is continuous in $\overline{B_R(0)}$ and has negative infimum, by Lemma \ref{example-1}. This gives $\widetilde u$ is also a non-constant function i.e. $\sup_{B_R(0)}\widetilde u>0$.
Now we claim that, there exists $r_0 \in (0,R)$ such that $\sup_{B_{r_0}(x_0)} \widetilde u>0.$ 
On the contrary, suppose $\sup_{B_r(x_0)}\widetilde u=0$ for every $0<r<R$. 
Denoting $d:=\text{dist}(x_0, \partial B_R(0))=|x_0-\widetilde x|$ for $\widetilde x \in \partial B_R(0),$ we get $\widetilde u\equiv 0$ in $B_{d}(x_0)$ since $d<R$. This again contradicts the boundary regularity of $\widetilde u$ at $\widetilde x.$ Thus we find a radius $r_0 \in (0,R)$ such that $\sup_{B_{r_0}}\widetilde u>0.$ So, an application of Harnack inequality (Theorem \ref{harnack_intro}) in the balls with radius $0<r_0<r_1\leq R$ gives
 \begin{align*}  
 \sup_{B_{r_0}(x_0)} \widetilde u\leq C \inf_{B_{r_0}(x_0)} \widetilde{u}+ \left(\frac{r_0}{r_1}\right)^{2s}\left[\operatorname{Tail}_{+}(\widetilde{u}_{-}; x_0, r_1)+ \operatorname{Tail}_{-}(\widetilde{u}_{+}; x_0, r_1)\right].
\end{align*}
Since $\inf_{B_{r_0}(x_0)}\widetilde u=0,$ and $\widetilde u\geq 0$ in $\rd$ we get
 \begin{align}\label{contradict harnack}
 0<\sup_{B_{r_0}(x_0)} \widetilde u\leq C \operatorname{Tail}_{-}(\widetilde{u}_{+}; x_0, r_1).    
 \end{align}
  Using $u \in L_{\log, 2}(\rd)$, we have (see Remark \ref{tail remark}), $0<\operatorname{Tail}_{-}(\widetilde{u}_{+}; x_0, r_1) < +\infty. $
 Hence, we cannot drop the tail quantity in \eqref{contradict harnack}.
 
In the following, we show that the test function used in the above lemma is admissible.
 \begin{lemma}\label{lem:test function}
    We have $\mathbb{1}_{B_{R}(0)}\in W^{s+\log,2}(\rd)$ for $0<s<\frac{1}{2}.$ 
\end{lemma}
\begin{proof}
Denote $v=\mathbb{1}_{B_{R}(0)}$. It is easy to see that $v\in L^2(\rd)$. We now show 
\begin{align*}
      [v]_{s+\log,2} := \iint_{\rd \times \rd} \Kp(|x-y|) |v(x) - v(y)|^2 \dxy<\infty.
\end{align*}
Clearly, we have
\begin{align*}
\iint_{\rd \times \rd} \Kp(|x-y|) |v(x) - v(y)|^2 \dxy&=\iint_{B_R(0) \times B_R(0)} \Kp(|x-y|) |v(x) - v(y)|^2 \dxy\\&\quad+2\iint_{B_R(0) \times \rd \setminus B_R(0)} \Kp(|x-y|) |v(x) - v(y)|^2 \dxy\no\\&=   2\iint_{B_R(0) \times \rd\setminus B_R(0)} \Kp(|x-y|)  \dxy\\&=2\int_{B_R(0)}\left(\int_{\rd\setminus B_R(0)} \frac{(-\log(|x-y|))_+}{|x-y|^{d+2s}}\dx\right)\dy. 
\end{align*}
Denoting $\delta(y)=\operatorname{dist}(y,\partial B_R(0))$, we note that $B_R^c(0)\subset B_{\delta(y)}^c(y)$ for any $y\in B_R$. Thus, we estimate as
\begin{align*}
    \int_{B_R(0)}\left(\int_{\rd\setminus B_R(0)} \frac{(-\log(|x-y|))_+}{|x-y|^{d+2s}}\dx\right)\dy&\leq \int_{B_R(0)}\left(\int_{B_{\delta(y)}^c(y)} \frac{(-\log(|x-y|))_+}{|x-y|^{d+2s}}\dx\right)\dy.
\end{align*}
Using polar coordinates, we now have
\begin{align*}
\int_{B_R(0)}\left(\int_{B_{\delta(y)}^c(y)} \frac{(-\log(|x-y|))_+}{|x-y|^{d+2s}}\dx\right)\dy&=\omega_{d-1}\int_{B_R(0)}\left(\int_{\delta(y)}^1 \tau^{d-1}\frac{(-\log(\tau))_+}{\tau^{d+2s}}\mathrm{d}\tau\right)\dy\\&\leq \omega_{d-1}\int_{B_R(0)}\left(\int_{\delta(y)}^1 \frac{|\log\tau|}{\tau^{2s+1}}\mathrm{d}\tau\right)\dy.
\end{align*}
Since $\tau^{\varepsilon}|\log (\tau)|\to 0$ as $\tau \to 0,$ for any $\varepsilon >0$, we obtain
\begin{align*}
    \int_{\delta(y)}^1 \frac{|\log\tau|}{\tau^{2s+1}}\mathrm{d}\tau\lesssim \int_{\delta(y)}^1  \tau^{-\varepsilon-2s-1}\, \mathrm{d}\tau=\frac{\tau^{-\varepsilon-2s}}{-2s-\varepsilon}\Bigg|^1_{\delta(y)}=\frac{\delta(y)^{-2s-\varepsilon}-1}{2s+\varepsilon}\leq \frac{\delta(y)^{-2s-\varepsilon}}{2s+\varepsilon}.
\end{align*}
Therefore,
\begin{align*}
    \iint_{\rd \times \rd} \Kp(|x-y|) |v(x) - v(y)|^2 \dxy\leq C\int_{B_R(0)}\delta(y)^{-2s-\varepsilon}\dy,
\end{align*}
for any $\varepsilon>0$. For $s<\frac 12$, we have $-2s-\varepsilon>-1$, if we choose $0<\varepsilon<1-2s$, and the above integral is finite. This completes the proof.
\end{proof}

Finally, the following remark discusses the existence of a non-trivial weak solution that satisfies the hypothesis of Lemma \ref{example-1}. 

\begin{remark}\label{exists}
The effect of $\widetilde R:= \widetilde{R}(d, s)$ comes into play for the existence of a weak solution to \eqref{countereq}. We choose $\widetilde R>0$ satisfying
\begin{align*}
 \left(S_{d,s}\frac{C_{d,s}\omega_{d-1}}{s^2}\right)\widetilde R^{2s}<{(-\log(2\widetilde{R}))_{+}},   
\end{align*}
where $C_{d,s}$ is the normalization constant for fractional Laplacian and $S_{d,s}$ is the constant appearing in the following fractional Poincar\'e inequality (Proposition \ref{poincare}):
\begin{align*}
    \|u\|_{L^2(B_R(0))}^2\leq S_{d,s}\frac{R^{2s}}{(-\log(2R))_{+}}\iint_{B_R(0) \times B_R(0)}\mathcal{K}^{s+\log,2}_+(|x-y|)(u(x)-u(y))^2\dxy.
\end{align*}
For $\text{diam}(B_R)\le \widetilde R$, we consider the following Dirichlet boundary value problem: 
\begin{align}\label{ceq}
(-\Delta)^{s+\log}v=f \text{ in } B_{R}(0),\quad v=0  \text{ in }  \rd \setminus B_{R}(0),
\end{align}
 with $f \in (W_0^{s+\log,2}(B_{R}))^*$. Using \cite[Theorem 1.1]{ChenChenHauer}, the energy functional of \eqref{ceq} is coercive for the above choice of $\widetilde R$. Therefore, using the Lax-Milgram lemma, \eqref{ceq} admits a weak solution ${v} \in W_0^{s+\log,2}(B_{R})$.

We observe that for $s< \frac{1}{2}$, 
\begin{align*}
 \mathbb{1}_{B_{R_e}^c} \in L_{\log, 2}(\rd) \cap W^{s+\log,2}(B_{R}) \text{ and } -(-\Delta)^{s+\log} \mathbb{1}_{B_{R_e}^c} \in \left(W^{s+\log,2}_0(B_{R})\right)^*.
\end{align*}
Thus, for $f:= -(-\Delta)^{s+\log}g$ and $g=\mathbb{1}_{B_{R_e}^c}$, there exists $\widetilde v\in W^{s+\log,2}_0(B_{R}) $ which satisfies the following identity for every $\phi\in W^{s+\log,2}_0(B_{R})$,
\begin{align}\label{weak}
    \iint_{\rd \times \rd} \K(|x-y|)(\widetilde{v}(x)-\widetilde{v}(y)) (\phi(x) -\phi(y)) \dxy = \left<-(-\Delta)^{s+\log}g, \phi\right>, 
\end{align}
Using $W^{s+\log,2}_0(B_{R}) \subset L_{\log, 2}(\rd)$, the function $u:=\widetilde{v}+g\in W^{s+\log,2}(B_{R}) \cap L_{\log, 2}(\rd)$ and satisfies $u=g$ in $B^c_R(0)$. Furthermore, by \eqref{weak}, $u$ weakly solves \eqref{countereq}.

Moreover, for $x \in B_R(0)$, 
\begin{align*}
 |f(x)|\leq C\int_{B^c_{1}(x)}\frac{1+\log(|x-y|)}{|x-y|^{d+2s}}\dy=C\omega_{d-1}\int_1^\infty \frac{1+\log(\tau)}{\tau^{d+2s}}\dtau=M,   
\end{align*}
for some $M>0$, i.e., $f \in L^\infty(B_R)$. Thus, the set up of \cite{RosOtonSerra2014} can be used (with a suitable radius) to get the global regularity of $v$, which gives that $u$ is also continuous up to $\partial B_R(0)$.
\end{remark}

\subsection*{Acknowledgment} Nirjan Biswas acknowledges the support of National Board for Higher Mathematics Postdoctoral Fellowship (0204/16(9)/2024/RD-II/6761). Abhrojyoti Sen is supported by the Alexander von Humboldt Foundation, Germany. He thanks Tobias Weth for several helpful discussions regarding the counterexample. In addition, he thanks Alberto Salda\~na whom he met at the conference funded by Ettore Majorana Foundation in Erice, Italy, and had a discussion about fractional logarithmic operators. 

\medskip
\noindent\textbf{Data availability.} This manuscript does not have associated data.

\bibliographystyle{abbrv}
\bibliography{ref}
\end{document}